\documentclass[reqno,12pt]{amsart}

\calclayout

\usepackage{amssymb}
\usepackage{latexsym}
\usepackage{exscale}
\usepackage{graphicx}

\usepackage[colorlinks=true, pdfstartview=FitV, linkcolor=blue, citecolor=red, urlcolor=blue]{hyperref}

\usepackage{color}
\usepackage[usenames,dvipsnames]{xcolor}

\usepackage{pgf}

\let\oldtocsection=\tocsection

\let\oldtocsubsection=\tocsubsection

\let\oldtocsubsubsection=\tocsubsubsection

\renewcommand{\tocsection}[2]{\hspace{0em}\oldtocsection{#1}{#2}}
\renewcommand{\tocsubsection}[2]{\hspace{5em}\oldtocsubsection{#1}{#2}}
\renewcommand{\tocsubsubsection}[2]{\hspace{2em}\oldtocsubsubsection{#1}{#2}}

\allowdisplaybreaks

\numberwithin{equation}{section}

\theoremstyle{plain}
\newtheorem{theor}[equation]{Theorem}
\newtheorem{lemma}[equation]{Lemma}
\newtheorem{corol}[equation]{Corollary}
\newtheorem{prop}[equation]{Proposition}

\theoremstyle{definition}
\newtheorem{defi}[equation]{Definition}

\theoremstyle{remark}
\newtheorem{remark}[equation]{Remark}

\newcommand{\re}{\mathbb{R}}
\newcommand{\rn}{{\mathbb{R}^n}}
\newcommand{\co}{\mathbb{C}}

\newcommand{\RR}{\mathbb{R}}
\newcommand{\NN}{\mathbb{N}}
\newcommand{\HH}{\mathbb{H}}

\newcommand{\J}{\mathcal{J}}
\newcommand{\K}{\mathcal{K}}

\renewcommand{\H}{\mathcal{H}}

\newcommand{\D}{\mathcal{D}}
\newcommand{\cals}{\mathcal{S}}

\def\E{{\mathcal E}}
\def\grad{{\nabla}}
\def\a{{\mathfrak a}}
\def\C{{\mathbb C}}

\newcommand{\A}{\mathcal{A}}

\newcommand{\ep}{\varepsilon}
\newcommand{\normH}[1]{|\!|\!|#1|\!|\!|}

\newcommand{\bigchi}{\mathop{\mathchoice%
{\mbox{\large$\chi$}}{\mbox{\large$\chi$}}{\mbox{\normalsize$\chi$}}%
{\mbox{\small$\chi$}}}\nolimits}

\DeclareMathOperator*{\esssup}{ess\,sup}

\renewcommand{\Re}{{\rm Re}\,}

\def\div{\mathop{\rm div}}
\def\Int{\mathop{\rm Int}}
\def\supp{\mathop{\rm supp}}

\newcommand{\expt}[1]{e^{\textstyle #1}}

\def\Xint#1{\mathchoice
  {\XXint\displaystyle\textstyle{#1}}%
  {\XXint\textstyle\scriptstyle{#1}}%
  {\XXint\scriptstyle\scriptscriptstyle{#1}}%
  {\XXint\scriptscriptstyle\scriptscriptstyle{#1}}%
  \!\int}
\def\XXint#1#2#3{{\setbox0=\hbox{$#1{#2#3}{\int}$}
    \vcenter{\hbox{$#2#3$}}\kern-.5\wd0}}

\def\avgint{\Xint-}
\def\aver#1{\avgint_{#1}}

\newcommand{\off}[2]{\mathcal{O}\big(L^{#1}\rightarrow L^{#2}\big)}
\newcommand{\offw}[2]{\mathcal{O}\big(L^{#1}(w) \rightarrow L^{#2}(w)\big)}
\newcommand{\offwmu}[3]{\mathcal{O}\big(L^{#1}(w) \rightarrow L^{#2}(w),\Sigma_{#3}\big)}

\newcommand{\fullw}[2]{\mathcal{F}\big(L^{#1}(w) \rightarrow L^{#2}(w)\big)}

\newcommand{\dec}[1]{\Upsilon\!\left(#1\right)}

\begin{document}

\title[On the Kato problem and extensions for degenerate elliptic operators]{On the Kato problem and extensions for degenerate elliptic operators}

\author{David Cruz-Uribe, OFS}

\address{David Cruz-Uribe, OFS\\
Department of Mathematics \\ University of Alabama\\
Tuscaloosa, AL 35487-0350, USA}
\email{dcruzuribe@ua.edu}

\author{Jos\'e Mar{\'\i}a Martell}
\address{Jos\'e Mar{\'\i}a Martell\\
Instituto de Ciencias Matem\'aticas CSIC-UAM-UC3M-UCM\\
Consejo Superior de Investigaciones Cient{\'\i}ficas\\
C/ Nicol\'as Cabrera, 13-15\\
E-28049 Madrid, Spain
}
\email{chema.martell@icmat.es}
\address{and}
\address{
Department of Mathematics
\\
University of Missouri
\\
Columbia, MO 65211, USA} 
\email{martellj@missouri.edu}

\author{Cristian Rios}
\address{Cristian Rios\\
Department of Mathematics and Statistics\\
University of Calgary\\
2500 University Drive NW\\
Calgary, AB, T2N-1N4, Canada}
\email{crios@ucalgary.ca}

\thanks{The first author is supported by NSF grant 1362425 and
  research funds provided by the Dean of Arts \& Sciences at the
  University of Alabama.  While substantial portions of this work was
  done he was supported by the Stewart-Dorwart faculty
  development fund at Trinity College.
	The second authors acknowledges financial support from the Spanish Ministry of Economy and Competitiveness, through the ``Severo Ochoa Programme for Centres of Excellence in R\&D'' (SEV-2015-0554). He also acknowledges that the research leading to these results has received funding from the European Research Council under the European Union's Seventh Framework Programme (FP7/2007-2013)/ ERC
agreement no. 615112 HAPDEGMT. The third author is   supported by the Natural Sciences and Engineering Research Council
  of Canada Discovery Grant RT733901. 
The authors warmly thank P.~Auscher and S.~Hofmann for several useful
discussions and their many suggestions.}

\date{October 6, 2016. \textit{Revised:} \today}

\subjclass[2010]{35B45, 35J15, 35J25, 35J70, 42B20, 42B37, 47A07, 47B44, 47D06}

\keywords{Muckenhoupt weights, degenerate elliptic
  operators, Kato problem, semigroups, holomorphic functional
  calculus, square functions, square roots of elliptic operators, Riesz
  transforms, Dirichlet problem, regularity problem, and Neumann problem.}

\begin{abstract}
We study the Kato problem for divergence form operators whose ellipticity may be degenerate.  The study of the
  Kato conjecture for degenerate elliptic equations was begun
  in~\cite{cruz-uribe-riosP,cruz-riosP,DCU-CR2013}.  In these papers the authors proved 
  that given an operator $L_w=-w^{-1}\div (A\grad)$, where $w$ is in the Muckenhoupt class $A_2$ and $A$ is a $w$-degenerate elliptic measure (that is, $A=w\,B$ with $B(x)$ an $n\times n$ bounded, complex-valued, uniformly elliptic matrix), then $L_w$ satisfies the weighted estimate   $\|\sqrt{L_w}f\|_{L^2(w)}\approx \|\nabla f\|_{L^2(w)}$. In the   present paper we solve the $L^2$-Kato problem for a family of   degenerate elliptic operators. We prove that under some   additional conditions on the weight $w$,  the following unweighted   $L^2$-Kato estimates hold:
$$
\|L_w^{1/2}f\|_{L^2(\re^n)}\approx \|\nabla f\|_{L^2(\re^n)}.
$$

This extends the celebrated solution to the Kato conjecture by Auscher, Hofmann, Lacey, McIntosh, and Tchamitchian, allowing the
differential operator to have some degree of degeneracy in its ellipticity.  For example, we consider the family of operators $L_\gamma=-|x|^{\gamma}\div (|x|^{-\gamma} B(x)\grad)$, where $B$  is any bounded, complex-valued, uniformly elliptic matrix.  We prove  that there exists $\epsilon>0$, depending only on dimension and the ellipticity constants, such that
$$
\|L_\gamma^{1/2}f\|_{L^2(\re^n)}\approx \|\nabla f\|_{L^2(\re^n)},
\qquad -\epsilon< \gamma< \frac{2\,n}{n+2}.
$$
The case $\gamma=0$ corresponds to the case of uniformly elliptic matrices. Hence, our result gives a range of $\gamma$'s for which the classical Kato square root proved in \cite{auscher-hofmann-lacey-mcintosh-tchamitchian02} is an interior
point.

Our main results are obtained as a consequence of a rich Calder\'on-Zygmund theory developed for certain operators naturally associated with $L_w$. These results, which are of independent interest, establish estimates on $L^p(w)$, and also on $L^p(v\,dw)$ with $v\in A_\infty(w)$, for the associated semigroup, its gradient, the functional calculus, the Riesz
transform, and vertical square functions.   As an application, we solve some   unweighted $L^2$-Dirichlet, Regularity and Neumann boundary value problems for degenerate elliptic operators. 
\end{abstract}

\maketitle

\tableofcontents

\section{Introduction}
\label{section:introduction}

We study the degenerate elliptic
operators $L_w=-w^{-1}\div A\grad$, where $w$ is in the Muckenhoupt
class $A_2$ and $A(x)$ is an $n\times n$ complex-valued matrix that
satisfies the degenerate ellipticity condition
\[  \lambda w(x) | \xi | ^{2}\leq {\Re}\langle A(x)\xi
    ,\xi \rangle, \qquad
  |\langle A(x)\xi ,\eta \rangle |\leq \Lambda w(x)|\xi
  ||\eta |, \quad \xi ,\,\eta \in \mathbb{C}^{n},
	\ \mbox{a.e.~}x\in\re^n. \]
Equivalently, $A(x)=w(x)B(x)$, where $B$ is an $n\times n$
complex-valued matrix that satisfies the uniform ellipticity
conditions 
\[  \lambda | \xi | ^{2}\leq {\Re}\langle B(x)\xi
    ,\xi \rangle, \qquad
  |\langle B(x)\xi ,\eta \rangle |\leq \Lambda |\xi
  ||\eta |, \quad \xi ,\,\eta \in \mathbb{C}^{n},
	\ \mbox{a.e.~}x\in\re^n. \] 
Such operators were first studied (with $A$ a real symmetric matrix)
by Fabes, Kenig and Serapioni~\cite{fabes-kenig-serapioni82}.  When
$A$ is complex-valued and uniformly elliptic (i.e. $w\equiv 1$), a
landmark result was the
proof of the Kato conjecture by Auscher,  Hofmann, Lacey, McIntosh, and Tchamitchian~\cite{auscher-hofmann-lacey-mcintosh-tchamitchian02}:  that for all $f\in H^1$,
\[ \|L^{1/2} f\|_2 \approx \|\grad f\|_2. \]
The proof of this long-standing conjecture led naturally to the study
of the operators associated with $L$:  the semigroup $e^{-tL}$, its
gradient $\sqrt{t}\grad e^{-tL}$, the Riesz transform $\grad
L^{-1/2}$, the $H^\infty$ functional calculus and square functions:
for details and complete references, see~Auscher~\cite{auscher07}.
These estimates are interesting in themselves; moreover,
it is well known that $L^p$ estimates for these operators yield
regularity results for boundary value problems for $L$:  for details, see the
introduction to~\cite{auscher-tchamitchian98}.

In~\cite{DCU-CR2013} (see also~\cite{cruz-uribe-riosP,cruz-riosP, Auscher-Rosen-Rule}) the
first and third authors solved the Kato problem for degenerate elliptic
operators:  they showed that
if $w\in A_2$ and $A$ satisfies the degenerate ellipticity conditions,
then for all $f\in H^1(w)$,
\begin{equation}  \label{eqn:Lw2-kato}
 \|L_w^{1/2} f\|_{L^2(w)} \approx \|\grad f\|_{L^2(w)}.
\end{equation}

In this paper we consider the problem of determining those $A_2$ weights
such that the classical Kato problem can be solved for $L_w$:  that
is, finding weights such that $L_w$ satisfies the unweighted estimate 
\[ 
 \|L_w^{1/2} f\|_{L^2(\re^n)} \approx \|\grad f\|_{L^2(\re^n)}, 
\]
for $f$ in a class  of nice functions ({\em a posteriori}, by standard density arguments, the estimate can be extended to all $f\in H^1(\re^n)$). 
We solve this problem in two steps.  The first is to prove weighted
$L^p$ estimates for some operators associated with $L_w$ (the
semigroup, its gradient, the Riesz transform, the functional calculus,
and square functions.)  These results, which are of interest in their
own right, are analogous to those obtained in the uniformly elliptic
case.   However, a significant technical obstruction is that given
a weight $w\in A_2$, while it is the case that there exists
$\epsilon>0$ such that $w\in A_{2-\epsilon}$, it is easy to construct
examples to show that $\epsilon$ may be arbitrarily small.  Therefore,
our bounds in the range $1<p<2$ need to take this into account.

The second step is to find conditions on the weight $w$ so that these
operators satisfy {\em unweighted} $L^2$ estimates.   Both steps are
carried out simultaneously, and the proofs are
intertwined.  Our approach is to apply the theory of off-diagonal
estimates on balls developed by Auscher and the second
author~\cite{auscher-martell06,auscher-martell07b,
  auscher-martell07,auscher-martell-08}.   We will in fact prove
weighted estimates on $L^p(v\,dw)$, where $v$ satisfies
Muckenhoupt and reverse H\"older conditions with respect to the
measure $dw=w\,dx$:  $L^p(w)$ estimates are then obtained by
taking $v=1$, and unweighted estimates by taking
$v=w^{-1}$.

The unweighted $L^2$ estimates are 
delicate, since they require a careful estimate of the constants that
appear.  Nevertheless, we are able to give useful sufficient
conditions: e.g.,  $w\in A_1\cap RH_{\frac{n}{2}+1}$.  (For definitions of
these classes, see Section~\ref{section:prelim} below.)  For example,
we have the following result that is a special case of one of our main
results (cf. Theorem~\ref{corol:super-Kato}).  

\begin{theor} \label{thm:special-case}
  Let $L_w=-w^{-1}\div A\grad$ be a degenerate elliptic operator as
  above. If $w\in A_1\cap RH_{\frac{n}{2}+1}$, then the Kato problem
  can be solved for $L_w$: for every $f\in H^1(\re^n)$, 
$$
 \|L_w^{1/2} f\|_{L^2(\re^n)} \approx \|\grad f\|_{L^2(\re^n)}.
$$
The implicit constants depend only on the dimension, the ellipticity
constants, and the $A_1$ and $RH_{\frac{n}{2}+1}$ constants of $w$.

 Furthermore, if we define
$L_\gamma=-|x|^{\gamma}\div(|x|^{-\gamma} B(x)\grad)$, where $B$ is an
$n\times n$ complex-valued matrix that satisfies the uniform
ellipticity condition,  then there exists $0<\epsilon<\frac12 $ small
enough (depending only on the dimension and the ratio $\Lambda/\lambda$)
such that
$$
 \|L_\gamma^{1/2} f\|_{L^2(\re^n)} \approx \|\grad f\|_{L^2(\re^n)},\qquad
-\epsilon< \gamma<\frac{2n}{n+2}.
$$
\end{theor}

\begin{remark}
In Theorem \ref{thm:special-case} the operator $L_w^{1/2}$ is {\em a priori} only defined on $H^1(w)$; however this means that it is defined on $C_0^\infty(\re^n)$ and so by a standard density argument we can extend our results to all $f\in H^1(\re^n)$.   Hereafter we will make this extension without further comment.
\end{remark}

We emphasize that in Theorem~\ref{thm:special-case}, when $\gamma=0$
  we are back at the uniformly elliptic case, which is the celebrated solution to
  the Kato square root problem by Auscher, Hofmann, Lacey, McIntosh,
  and Tchamitchian
  in~\cite{auscher-hofmann-lacey-mcintosh-tchamitchian02}.  Here we
  are able to find a range of $\gamma$'s for which the same estimates
  hold and the classical Kato square root problem (i.e., $\gamma=0$)
  is an interior point in that range.

These unweighted $L^2$ estimates have important applications to
boundary value problems for degenerate elliptic operators.  Consider,
for example, the following Dirichlet problem on $\re^{n+1}_+=\rn\times
[0,\infty)$:
\[ \begin{cases}
\partial_t^2 u - L_w u = 0, & \text{on }  \re^{n+1}_+  \\
u= f & \text{on } \partial\re^{n+1}_+ =\re^n.
\end{cases}
\]
If $f\in L^2(\re^n)$, then $u(x,t)=e^{-tL_w^{1/2}}f(x)$ is a solution,
  and if $L_w$ has a bounded $H^\infty$ functional calculus on $L^2$,
  then
$  \sup_{t>0} \|u(\cdot,t)\|_2 \lesssim \|f\|_2. $
Similar results hold for the corresponding Neumann  and Regularity problems.

\medskip

Our proofs are unavoidably technical, and the results for each operator
considered build upon what was proved previously for other operators.
We have organized the material as follows.  In
Section~\ref{section:prelim} we gather some essential definitions and results about
weights, degenerate elliptic operators, and off-diagonal estimates.
Central to all of our subsequent work are Theorems~\ref{theorem:2.2}
and~\ref{theorem:2.4} (which were proved in~\cite{auscher-martell06}).

In Sections~\ref{section:semigroup},~\ref{section:functional},
and~\ref{section:square-function} we prove estimates for the
semigroup $e^{-tL_w}$, $t>0$, the $H^\infty$ functional
calculus (i.e., operators $\varphi(L_w)$ where $\varphi \in
\H^\infty$), the vertical square function associated to the semigroup,
\[ g_{L_w} f(x)  =\bigg( \int_{0}^{\infty }\left\vert \left(
tL_{w}\right) ^{1/2}e^{-tL_{w}}f(x) \right\vert ^{2}\frac{dt}{t}%
\bigg) ^{1/2}, \]
and its discrete analog. Here and in subsequent sections we
prove both  $L^p(w)$ estimates and weighted $L^p(v\,dw)$
estimates.  In many cases these results are proved simultaneously,
with the unweighted results (i.e., in $L^p(w)$) following from the weighted results (i.e.,  in $L^p(v\,dw)$) by taking $v=1$. 

In Section~\ref{section:reverse} we prove the so-called reverse
inequality, $\|L_w^{1/2}\|_{L^p(w)} \lesssim \|\grad f\|_{L^p(w)}$,
that generalizes the $L^2(w)$ estimate in~\eqref{eqn:Lw2-kato}.  We
note that while the equivalence in~\eqref{eqn:Lw2-kato}
follows at once from the reverse inequality for $p=2$ by
duality, the two inequalities behave differently when $p\neq 2$.

In Sections~\ref{section:gradient} and~\ref{section:q-plus} we prove
estimates for the gradient of the semigroup,
$\sqrt{t}\grad e^{-tL_w}$.  The proof that there exists $q_+>2$ such
that this operator satisfies $L^p(w)$ estimates for $2<p<q_+$ is quite
involved as it requires preliminary estimates for the Riesz transform
and the Hodge projection.  We note that, as opposed to the
non-degenerate case, here we cannot use ``global'' embeddings, nor
can we rescale. Also we cannot expect to obtain that the gradient
of the semigroup maps globally $L^2(w)$ into $L^p(w)$ for $p\neq
2$.
All these difficulties arise naturally from the lack of isotropy of
the natural underlying measure $w(x)\,dx$ and make the typical
arguments used in the uniformly elliptic case
(cf.~\cite[Chapter~4]{auscher07}) unusable. We also note that in
some sense our result is the best possible: even in the non-degenerate
case it is known~\cite{auscher07} that given any $p>2$ there exists a
matrix $A$ and operator $L$ such that gradient of the semigroup is not
bounded on $L^p$.

In Section~\ref{section:riesz} we prove $L^p(w)$ estimates for the
Riesz transform $\grad L^{-1/2}$, and in
Section~\ref{section:square-function-gradient} we prove $L^p(w)$
estimates for the square function associated to the gradient of the
semigroup,
\[ G_{L_{w}}f( x)  =\bigg( \int_{0}^{\infty }
|t^{1/2}\nabla e^{-tL_{w}}f( x)|^{2}\frac{dt}{t}\bigg)^{1/2}. \]

In Section~\ref{section:L2-kato} we prove unweighted $L^2$
inequalities for the operators we have considered in previous
sections.  These are a consequence of the weighted estimates and are
obtained by taking $v=w^{-1}$.  The main problem is determining
conditions on $w$ for these to hold.  We essentially have two
different kinds of estimates: one for operators that do not involve
the gradient, and one for those that do.  The latter are more delicate
as they involve careful bounds for the parameter $q_+$ from
Section~\ref{section:q-plus} in terms of the weight~$w$.  We also show
that we get unweighted $L^p$ estimates for $p$ very close to $2$.

Finally, in Section~\ref{section:BVP} we describe in more detail the
application of our results to $L^2$ boundary value problems for
degenerate elliptic operators.   The results in this section our the culmination of
our work as they depend on all the estimates derived in previous
sections.

\medskip

As we were completing this project, we learned that related results
had been obtained  independently by other authors.  
In~\cite{LePhi} Le studied (among other things) the $L^p(w)$ theory for some of the
operators considered here and proved estimates for values of $p$ in the range
$(2-\epsilon,2+\epsilon)$. His proofs differ from ours in a number
of details. In~\cite{HLM} Hofmann, Le and Morris
established some Carleson measure estimates and considered the Dirichlet problem for degenerate elliptic operators.  Also, very recently we learned that Yang and Zhang~\cite{YZ} proved
Kato type estimates in $L^p(w)$ for $p$ in the range $(p_0,2]$.
Finally, we note that the paper
\cite{LiMartellPrisuelos} complements our work here as it considers the
conical square functions associated to the operator $L_w$.

\section{Preliminaries}
\label{section:prelim}

Throughout $n$ will denote the dimension of the underlying space $\rn$
and we will always assume $n\geq 2$.  If we write $A\lesssim B$ we
mean that there exists a constant $C$ such that $A\leq CB$.  We write
$A\approx B$ if $A\lesssim B$ and $B\lesssim A$.  The
constant $C$ in these estimates may depend on the dimension $n$ and other (fixed)
parameters that should be clear from the context.  All constants,
explicit or implicit, may change at each appearance.

Given a ball $B$, let $r(B)$ denote the radius of $B$.  Let
$\lambda B$ denote the concentric ball with radius $r(\lambda B) = \lambda r(B)$.

\subsection*{Weights}

By a weight $w$ we mean a non-negative, locally integrable function.
For brevity, we will often write $dw$ for $w\,dx$.   We will use the following notation for averages:   given a set $E$ such that
$0<w(E)<\infty$,
\[ \avgint_E f\,dw = \frac{1}{w(E)}\int_E f\,dw, \]
or, if $0<|E|<\infty$,
\[ \avgint_E f\,dx = \frac{1}{|E|}\int_E f\,dx. \]

We state some definitions and basic properties of Muckenhoupt
weights.  For further details,
see~\cite{duoandikoetxea01,garciacuerva-rubiodefrancia85}.
We say that $w\in A_p$, $1<p<\infty$, if
\[ [w]_{A_p} = \sup_Q \avgint_Q w(x)\,dx \left(\avgint_Q
  w(x)^{1-p'}\,dx\right)^{p-1} < \infty. \]
When $p=1$, we say that $w\in A_1$ if
\[ [w]_{A_1} = \sup_Q \avgint_Q w(x)\,dx  \esssup_{x\in Q} w(x)^{-1}<
\infty.  \]
We say that $w\in RH_s$, $1<s<\infty$, if
\[ [w]_{RH_s} = \sup_Q \left(\avgint_Qw(x)\,dx\right )^{-1}
\left(\avgint_Qw(x)^s\,dx\right )^{1/s} < \infty, \]
and we say that $w\in RH_\infty$ if
\[ [w]_{RH_\infty} = \sup_Q\left(\avgint_Q w(x)\,dx\right)^{-1} \esssup_{x\in Q} w(x)  <
\infty.  \]
Let
\[ A_\infty = \bigcup_{1\leq p <\infty} A_p  = \bigcup_{1<s\le \infty}
RH_s.  \]
Weights in the $A_p$ and $RH_s$ classes have a self-improving
property: if $w\in A_p$, there exists $\epsilon>0$ such that $w\in
A_{p-\epsilon}$, and similarly if $w\in RH_s$, then $w\in
RH_{s+\delta}$ for some $\delta>0$.  Hereafter, given $w\in A_p$, let
\[ r_w=\inf\{p: w\in A_p\}, \qquad s_w=\sup\{q: w\in RH_q\}. \]

An important property of $A_p$ weights is that they are doubling:
given $w\in A_p$, for all $\tau\ge 1$ and any ball $B$,
\[ w(\tau B)\leq [w]_{A_p} \tau^{pn} w(B).  \]
In particular, hereafter let $D\leq pn$ be the doubling order of $w$:
that is the smallest exponent such that this inequality holds.

As a consequence of this doubling property, we have that with the ordinary Euclidean distance
$|\cdot|$, $(\rn,dw,|\cdot|)$ is a space of homogeneous type.
In this setting we can define the new weight classes $A_p(w)$
and $RH_s(w)$ by replacing Lebesgue measure in the definitions above with
$dw$: e.g., $v\in A_p(w)$ if
\[ [v]_{A_p(w)} = \sup_Q \avgint_Q v(x)\,dw \left(\avgint_Q
  v(x)^{1-p'}\,dw\right)^{p-1} < \infty. \]
It follows at once from these definitions that there is a
``duality'' relationship between the weighted and unweighted
$A_p$ and $RH_s$ conditions:  $v=w^{-1} \in A_p(w)$ if and only if $w \in
RH_{p'}$ and $v=w^{-1}\in RH_s(w)$ if and only if $w\in A_{s'}$.

\medskip

Weighted Poincar\'e-Sobolev inequalities were proved
in~\cite{fabes-kenig-serapioni82}.

\begin{theor} \label{thm:wtd-poincare}
Given $w\in A_p$, $p\geq 1$, let
$p_w^*=\frac{p\,n\,r_w}{n\,r_w-p}$ if $p<n\,r_w$ and $p_w^*=\infty$
otherwise. Then for every $p\le q<p_w^*$,  ball $B$ and $f\in C_0^\infty(B)$,
\begin{equation} \label{eqn:wtd-imbedding}
\left(\avgint_B |f(x)|^q\,dw(x)\right)^{1/q} \leq
Cr(B)\left(\avgint_B |\grad f(x)|^p\,dw\right)^{1/p}.
\end{equation}

Moreover, if $f\in C^\infty(B)$, then
\begin{equation}\label{w-Poincare}
\left(\avgint_B |f(x)-f_{B,w}|^{q}\,dw(x)\right)^{1/q}
\le
C\,
r(B)\left(\aver{B} |\nabla f(x)|^{p}\,dw\right)^{1/p},
\end{equation}
where $f_{B,w}=\avgint_B f\,dw$.
\end{theor}

\begin{remark}
In the special
case when $w\in A_1$ and $1<p<n$ we can also take $q=p_w^*=p^*$, i.e.,
the regular Sobolev exponent.  See
P\'erez~\cite[Theorem~2.5.2]{perez1999}.
\end{remark}

\begin{remark}\label{remark-best-Poi}
If we let $q= \frac{np}{n-1}<p_w^*$, then we can get a sharp
  estimate for the constant $C$ in~\eqref{eqn:wtd-imbedding}
  and~\eqref{w-Poincare}: it is of the form $C(p,n)[w]_{A_p}^\kappa$
  where $\kappa=\frac{n\,p-1}{n\,p\,(p-1)}$.  This follows from the
  sharp weighted estimates for the fractional integral operator due to
  Alberico, Cianchi and Sbordone~\cite{MR2561035} and the standard
  pointwise estimates used to prove Poincar\'e-Sobolev inequalities;
  see~\cite{fabes-kenig-serapioni82} for details.
\end{remark}

\begin{remark}\label{remark:Poincare-non-smooth}
  By a standard density argument, once we know that \eqref{w-Poincare}
  holds for smooth functions in $B$ we can easily extend that estimate to
  any function $f\in L^q(w)$ with $\nabla f\in L^p(w)$. 
  Details are left to the reader.
\end{remark}

\subsection*{Degenerate elliptic operators}
Given $%
w\in A_2$ and constants $0<\lambda\leq \Lambda<\infty$,
let ${\E}_n(w, \lambda, \Lambda )$ denote the class of $n\times n$
matrices $A=\left( A_{ij}(x) \right) _{i,j=1}^{n}$ of
complex-valued, measurable functions satisfying the degenerate ellipticity
condition
\begin{equation}  \label{eqn:degen}
  \lambda w(x) | \xi | ^{2}\leq {\Re}\langle A\xi
    ,\xi \rangle, \quad
  |\langle \A\xi ,\eta \rangle |\leq \Lambda w(x)|\xi
  ||\eta |, \quad \xi ,\,\eta \in \mathbb{C}^{n}.
\end{equation}%

Given $A\in {\E}_n(w,\lambda,\Lambda)$, we define the
degenerate elliptic operator in divergence form
$$
{L}_{w}=-w^{-1}{{\div}}A{\grad}.
$$   
These operators
were developed in~\cite{cruz-uribe-riosP} and we refer the reader there for complete
details.  Here we sketch the key ideas.

Given a weight $w\in A_2$, the space $H^1(w)$ is the weighted Sobolev
space that is the completion of $C_c^\infty$ with respect to the norm
\begin{equation*}
\|f\|_{H^1(w)} = \left(\int_{\rn} \left(|f(x)|^2+|\grad
f(x)|^2\right)\,dw\right)^{1/2}.
\end{equation*}
Note that the space defined above would usually be denoted by $H^1_0(w)$. The space $H^1(w)$ is defined as the set of distributions for which both $f$ and $|\nabla f|$ belong to $L^2(w)$. However, since the underlying domain is $\rn$ this definition implies that the ``boundary" values vanish in the $L^2(w)$-sense,  and both definitions agree~\cite{miller82}. 

Given a matrix $A\in {\E}_n(w,\lambda,\Lambda)$, define ${%
\a}(f,g)$ to be the sesquilinear form
\begin{equation}  \label{eqn-form}
{\a}(f,g) = \int_{\rn} A(x)\grad f(x) \cdot
\overline{\grad g(x)} \,dx.
\end{equation}
Since $w\in A_2$ and $A$ satisfies \eqref{eqn:degen}, $\a
$ is a closed, maximally accretive, continuous sesquilinear form. Therefore,
there exists an operator  ${L}_w$ whose domain $\D(L_w)\subset H^1(w)$ is dense in
$L^2(w)$ and such
that for every $f \in \D(L_w)$ and every $g\in H^1(w)$,
\begin{equation}  \label{eqn-a2}
{\a}(f,g) = \langle {L}_wf, g \rangle_w= \int_{\rn} {L}_wf(x)\overline{g(x)}\,dw.
\end{equation}
We note that the operator $L_w$ is one to one. Indeed, if $u,v\in\D(L_w)$ are such that $L_wu=L_wv$, then
for all $g\in H^1(w)$
\[0=\int_\rn A(x)\grad (u(x)-v(x)) \cdot
\overline{\grad g(x)} \,dx.\]
Taking $g=u-v$ implies $\grad u(x)=\grad v(x)$ and so $u=v$.

The properties of the sesquilinear form guarantee
that on $L^2(w)$ there exists a bounded, strongly continuous semigroup $e^{-t{L}_w}$.
Further, it has a holomorphic extension.
Let
\[ \Sigma_\omega= \{ z\in \C : z\neq 0, |\arg(z)| < \omega \}  \]
and define $\vartheta, \vartheta^* \in
[0,\pi/2)$  by
\[ \vartheta = \sup\{ |\arg \langle Lf,f\rangle_w| : f \in \D(L_w)
\},  \qquad \vartheta^*=\arctan\sqrt{\frac{\Lambda^2}{\lambda^2}-1}. \]
Then there exists a complex semigroup $e^{-zL_w}$ on
$\Sigma_{\pi/2-\vartheta}$ of bounded operators on $L^2(w)$. By
the weighted ellipticity condition~\eqref{eqn:degen}, we have that
$0\le\vartheta\le\vartheta^* <\pi/2$.

\subsection*{Holomorphic functional calculus}
Our operator $L_w$ is ``an operator of type $\omega$" with
$\omega=\vartheta$ , as defined in~\cite{mcintosh86}.  Indeed, the
ellipticity conditions imply that $L_w$ is closed and densely defined,
its spectrum is contained in $\Sigma_{\vartheta}$, and its resolvent
satisfies standard decay estimates~\cite{cruz-uribe-riosP}. Therefore, we
can define an $L^2(w)$ functional calculus as in~\cite{mcintosh86}.

Given $\mu\in
(\vartheta, \pi)$, let $\H^\infty(\Sigma_\mu)$ be the collection of
bounded holomorphic functions on $\Sigma_\mu$.  To define
$\varphi(L_w)$ for $\varphi\in \H^\infty(\Sigma_\mu)$ we first
consider a smaller class:  we say that  $\varphi\in
\H^\infty_0(\Sigma_\mu)$  if for some $c,\,s>0$ it satisfies
\[ |\varphi(z)| \leq c|z|^s(1+|z|)^{-2s}, \quad z \in \Sigma_\mu.   \]
We then have an integral representation of $\varphi(L_w)$.  Let
$\Gamma_\theta$ be the boundary of $\Sigma_\theta$ with positive orientation,
and let
$\vartheta < \theta < \nu < \min(\mu , \pi/2)$; then
\begin{equation} \label{eqn:L2-holo-rep}
 \varphi(L_w) =  \int_{\Gamma_{\pi/2-\theta}} e^{-zL_w}\eta(z)\,dz,
\end{equation}
where %
\begin{equation} \label{eqn:L2-holo-rep-eta}
\eta(z) = \frac{1}{2\pi i} \int_{\gamma_{\nu}(z)} e^{\zeta z}
\varphi(\zeta)\,d\zeta
\end{equation}
and $\gamma_{\nu}(z)=\mathbb{R}^+e^{i\mathrm{sign}(\mathrm{Im}(z))\nu} $. Note that
\[ |\eta(z)| \lesssim \min(1,|z|^{-s-1}), \quad z \in
\Gamma_{\pi/2-\theta}, \]
so the representation \eqref{eqn:L2-holo-rep} converges in $L^2(w)$, and we have the bound
\begin{equation}\label{eqn:L2-holo-bd}
\|\varphi(L_w)f\|_{L^2(w)} \leq C\|\varphi\|_\infty
\|f\|_{L^2(w)},\qquad f\in \H^\infty_0(\Sigma_\mu). 
\end{equation}

Now, since $L_w$ is a one-to-one
operator of type $\omega$, it has dense range~\cite[Theorem
2.3]{couling96}, and so the results in~\cite{mcintosh86} (see also~\cite[Corollary
2.2]{couling96}) imply that $L_w$ has an ${H}^\infty$
functional calculus and~\eqref{eqn:L2-holo-bd} extends to all of
$\H^\infty(\Sigma_\mu)$. Moreover, in~\cite[Section 8]{mcintosh86}
the equivalence between the existence of this ${H}^\infty$
functional calculus and square function estimates for $L_w$ and
$L_w^*$ is established:
\begin{equation}\label{eqn:holo-sqfe}
\left\{\int_0^\infty\| \varphi(tL_w)\|_{L^2(w)}^2\,\frac{dt}{t}\right\}^{\frac{1}{2}}\le C\|\varphi\|_\infty\| f\|_{L^2(w)},\quad \varphi\in\H^\infty_0(\Sigma_\mu),
\end{equation}
with similar estimates for $L_w^*$. 

The operators $\varphi(L_w)$ also have the following properties:
\begin{itemize}

\item If $\varphi$ and $\psi$ are bounded holomorphic functions, then we
have the operator identity $\varphi(L)\psi(L) = (\varphi \psi)(L)$.

\item Given any sequence $\{\varphi_k\}$ of bounded holomorphic functions
converging uniformly on compact subsets of $\Sigma_\mu$ to $\varphi$,
we have that $\varphi_k(L_w)$ converges to $\varphi(L_w)$ in the
strong operator topology (of operators on $L^2(w)$).

\end{itemize}

\begin{remark}
The $H^\infty$ functional calculus can be extended to more general
holomorphic functions, such as powers, for which the operators $\varphi (L_{w})$ can be
defined as unbounded operators:  see \cite{haase06,mcintosh86}.
\end{remark}

\medskip

\subsection*{Gaffney-type estimates}
The semigroup and its gradient satisfy
Gaffney-type estimates on $L^2(w)$.  Below, we will see that these are
a particular case of what we will call full off-diagonal estimates:
see Definition~\ref{defn:full-offdiagonal}.

\begin{theor} \label{thm:L2-gaffney}
Given $w\in A_2$ and $A\in  {\E}_n(w,\lambda,\Lambda)$, 
for any closed sets $E$ and $F$,  for $f\in L^2(w)$ and for all $z\in
\Sigma_\nu$, where
$0<\nu<\frac{\pi}{2}-\vartheta$,
\begin{enumerate}
\setlength{\itemsep}{8pt}

\item $\|e^{-z\,L_w} (f\,\bigchi_E)\bigchi_F\|_{L^2(w)}
\le
C\, e^{-\frac{c\, d(E,F)^2}{|z|}}\,\|f\bigchi_E\|_{L^2(w)}$,

\item $\|\sqrt{z}\grad e^{-z\,L_w} (f\,\bigchi_E)\bigchi_F\|_{L^2(w)}
\le
C\, e^{-\frac{c\, d(E,F)^2}{|z|}}
\,\|f\bigchi_E\|_{L^2(w)}$,

\item $\|z\, L_w e^{-z\,L_w} (f\,\bigchi_E)\bigchi_F\|_{L^2(w)}
\le
C\, e^{-\frac{c\, d(E,F)^2}{|z|}}
\,\|f\bigchi_E\|_{L^2(w)}$.

\end{enumerate}
\end{theor}

\begin{proof}
The semigroup estimate (1) was proved
in~\cite[Theorem~1.6]{cruz-uribe-riosP} for real $z$, but the same
proof can be readily modified to prove the analytic version.
Alternatively, estimates (1) and (2)  follow from the resolvent bounds
\begin{equation} \label{eqn:res01}
\|(1+z^2L_w)^{-1}(f\bigchi_E)\bigchi_F\|_{L^2(w)}
\leq Ce^{-\frac{cd(E,F)}{|z|}}\|f\bigchi_E\|_{L^2(w)},
\end{equation}
and
\begin{equation} \label{eqn:res02}
 \|z\grad(1+z^2L_w)^{-1}(f\bigchi_E)\bigchi_F\|_{L^2(w)}
\leq Ce^{-\frac{cd(E,F)}{|z|}}\|f\bigchi_E\|_{L^2(w)},
\end{equation}
obtained in~\cite[Lemma~2.10]{DCU-CR2013} for
$z\in\Sigma_{\frac{\pi}{2}+\nu}$, together with the integral
representation of the semigroup
\[ e^{-zL_w}f=\frac{1}{2\pi}\int_{\Gamma} e^{z\zeta}\left(\zeta +L_w\right)^{-1}f\, d\zeta,%
\]
where $\Gamma$ is the boundary of $\Sigma_{\theta}$ with positive
orientation and $\frac{\pi}{2}<\theta<\frac{\pi}{2}+\nu-\arg(z)$.

Finally, from \eqref{eqn:res01} and \eqref{eqn:res02} we obtain the estimate
\[\|z^2 L_w(1+z^2L_w)^{-1}(f\bigchi_E)\bigchi_F\|_{L^2(w)}
\leq Ce^{-\frac{cd(E,F)}{|z|}}\|f\bigchi_E\|_{L^2(w)}, \]
and then by the same kind of argument we get (3).
\end{proof}

\medskip

\subsection*{The Kato estimate}
The starting point for all of our estimates is the $L^2(w)$ Kato
estimates for the square root operator $L^{1/2}_w$ proved
in~\cite{DCU-CR2013} (see also \cite{Auscher-Rosen-Rule} for a different proof).  This operator is the
unique, maximal accretive operator such that $L^{1/2}_w
L^{1/2}_w=L_w$.   It has the integral representation
\[ L_w^{1/2}  = \frac{1}{\sqrt{\pi}}\int_0^\infty
\sqrt{t}L_we^{-tL_w}\, \frac{dt}{t}. \]
(For further details, see~\cite{auscher-tchamitchian98,mcintosh86}.)

\begin{theor}{\cite[Theorem~1.1]{DCU-CR2013}}\label{theorem:degen-kato}
Given $w\in A_2$ and $A \in {\E}_n(w,\lambda,\Lambda)$, the domain of $L_w$ is  $H^1(w)$ and 
there exist
constants $c$ and $C$, depending on $n$, $\Lambda/\lambda$ and
$[w]_{A_2}$, such that for all
$f\in H^1(w)$,
\begin{equation} \label{eqn:degen-kato1}
 c\|\grad f\|_{L^2(w)} \leq \|L_w^{1/2}f\|_{L^2(w)} \leq C\|\grad
f\|_{L^2(w)}.
\end{equation}
\end{theor}

\medskip

The Riesz transform associated to $L_w$ is the operator $\grad
L_w^{-1/2}$.  Formally, by \eqref{eqn:degen-kato1} we have that the
Riesz transform is a bounded operator on $L^2(w,\C^n)$.  To legitimize
this, we define
\begin{equation}\label{defi-RT}
\grad L_w^{-1/2} = \frac{1}{\sqrt{\pi}} \int_0^\infty
\sqrt{t} \grad e^{-t
  L_w} \frac{dt}{t}.
\end{equation}
However, it is not immediate that this integral converges at $0$ or
$\infty$.  To rectify this, for $\epsilon>0$ define
\begin{equation}\label{defi-RT:trunc}
S_\epsilon = S_\epsilon(L_w) = \frac{1}{\sqrt{\pi}} \int_\epsilon^{1/\epsilon}
\sqrt{t} e^{-tL_w} \frac{dt}{t}.
\end{equation}
Since $S_\epsilon(z)$ is a uniformly bounded holomorphic function on the right
half plane for all $0<\epsilon<1$,  by the $L^2(w)$ functional calculus described above, $S_\epsilon(L_w)$ is uniformly bounded on $L^2(w)$ for that range of $\epsilon$.  Further, for $f\in L_c^\infty$, $S_\epsilon f \in \D(L_w)
\subset \D(L_w^{1/2})$, and so by inequality~\eqref{eqn:degen-kato1} and the functional calculus,
\begin{equation} \label{eqn:trunc}
\|\grad S_\epsilon f \|_{L^2(w)} \lesssim \|L^{1/2} S_\epsilon
f\|_{L^2(w)} = \|\varphi_\epsilon (L_w)f\|_{L^2(w)},
\end{equation}
where
\[ \varphi_\epsilon(z)  = \frac{1}{\sqrt{\pi}}
\int_\epsilon^{1/\epsilon}
\sqrt{t} \sqrt{z} e^{-tz}\frac{dt}{t}. \]
The sequence $\{\varphi_\epsilon\}$ is uniformly bounded and converges
uniformly to $1$ on compact subsets of the sector $\Sigma_\mu$,
$0<\mu<\pi/2$.  Therefore, $L^{1/2} S_\epsilon f \rightarrow f$
strongly in $L^2(w)$.   If we combine this fact with \eqref{eqn:trunc}
we see that $\{\grad S_\epsilon f\}$ is Cauchy and so it converges in
$L^2(w)$.  We therefore define
\[ \grad L^{-1/2} f = \lim_{\epsilon\rightarrow 0} \grad S_\epsilon f, \]
where the limit is in $L^2(w)$.

Given this definition, hereafter, when we are proving $L^2(w)$
estimates for the Riesz transform, we should actually prove estimates
for $\grad S_\epsilon$ that are independent of $\epsilon$.  These
arguments will remain implicit unless there are details we need to emphasize.

\bigskip

\subsection*{Off-diagonal estimates}
Off-diagonal estimates as we define them were introduced in~\cite{auscher-martell07} and we
will refer repeatedly to this paper for further information and
results.  Throughout this section we will assume that given a weight
$w$, $w\in A_2$.

\smallskip

Given a ball $B$, for $j\geq 2$ we define the annuli
$C_{j}(B)=2^{j+1}\, B\setminus 2^j\, B$.  We let $C_{1}(B)=4B$.
By a slight abuse of notation, we will define
\[ \aver{C_j(B)} h\,dw
=
\frac1{w(2^{j+1}B)}\,\int_{C_{j}(B)} h\,dw. \]
If $w\in A_2$ (as it will be hereafter), then $w(2^{j+1}B)\approx
w(C_j(B))$, so this definition is equivalent to the one given above up
to a constant.  Finally, for $s>0$ we set  $\dec{s}=\max\{s,s^{-1}\}$.

\begin{defi}\label{defi:off-d:weights}
  Given $1\le p\le q\le \infty$, a family $\{T_t\}_{t>0}$
  of sublinear operators satisfies $L^{p}(w)-L^{q}(w)$ off-diagonal
  estimates on balls, denoted by
\[ T_t \in\offw{p}{q},\]
if there exist constants
  $\theta_1, \theta_2>0$ and $c>0$ such that for every $t>0$ and for
  any ball $B$, setting $r=r(B)$,
\begin{equation}\label{w:off:B-B}
\left(\aver{B} |T_t( \bigchi_B \, f) |^{q}\,dw\right)^{\frac 1 q}
\lesssim
\dec{\frac{r}{\sqrt{t}}}^{\theta_2} \,\left(\aver{B}
|f|^{p}\,dw\right)^{\frac 1 p };
\end{equation}
and for all $j\ge 2$,
\begin{equation}\label{w:off:C-B}
\left(\aver{B}|T_t( \bigchi_{C_j(B)}\, f) |^{q}\,dw\right)^{\frac 1 q}
\lesssim
2^{j\,\theta_1} \dec{\frac{2^j\,r}{\sqrt{t}}}^{\theta_2}\,
\expt{-\frac{c\,4^{j}\,r^2}{t}} \,
\left(\aver{C_j(B)}|f|^{p}\,dw\right)^{\frac 1 p }
\end{equation}
and
\begin{equation}\label{w:off:B-C}
\bigg(\aver{C_j(B)}|T_t( \bigchi_B \, f) |^{q}\,dw\bigg)^{\frac 1 q}
\lesssim
2^{j\,\theta_1} \dec{\frac{2^j\,r}{\sqrt{t}}}^{\theta_2}\,
\expt{-\frac{c\,4^{j}\,r^2}{t}}
\,\left(\aver{B}|f|^{p}\,dw\right)^{\frac 1 p }.
\end{equation}
If the family of sublinear operators
$\left\{ T_z\right\}_{z\in\Sigma_\mu}$ is defined on a complex sector
$\Sigma_\mu$, we say that it satisfies $L^{p}(w)-L^{q}(w)$
off-diagonal estimates on balls in $\Sigma_\mu$ if \eqref{w:off:B-B},
\eqref{w:off:C-B} and \eqref{w:off:B-C}  hold for
$z\in\Sigma_\mu$ with $t$ replaced by $|z|$ in the
righthand terms. We denote this by
$T_z\in\offwmu{p}{q}{\mu}$.
 \end{defi}

We give some basic properties of off-diagonal estimates on balls as a
series of lemmas taken from~\cite[Section~2.2]{auscher-martell07}.  The
first follows immediately by real interpolation, the second by
H\"older's inequality, and the third by duality.

\begin{lemma} \label{lemma:interpolate}
Given $1\leq p_i \le q_i\leq \infty$, $i=1,\,2$, if $T_t \in
\offw{p_1}{q_1}$, and $T_t : L^{p_2}(w)\rightarrow
L^{q_2}(w)$ is uniformly bounded, then $T_t\in \offw{p_\theta}{q_\theta}$,
$0<\theta<1$, where
\[ \frac{1}{p_\theta} = \frac{\theta}{p_1}+\frac{1-\theta}{p_2},
\qquad
\frac{1}{q_\theta} = \frac{\theta}{q_1}+\frac{1-\theta}{q_2}. \]
\end{lemma}

\begin{lemma} \label{lemma:nested}
If $1\leq p\leq p_1\leq q_1\leq q\leq \infty$, then
\[ \offw{p}{q}\subset \offw{p_1}{q_1}. \]
\end{lemma}

\begin{lemma} \label{lemma:duality}
If for some $1\leq p\leq q\leq \infty$, $T_t \in \offw{p}{q}$, and the
operators $T_t$ are linear, then $T_t^* \in \offw{q'}{p'}$. (Here $T_t^*$ 
is the dual operator for
the inner product $\int_{\re^n} f\,g\,dw$.) 
\end{lemma}

\begin{lemma}[{\cite[Theorem~2.3]{auscher-martell07}}] \label{lemma:unif-comp} \
\begin{enumerate}
\item If $T_t \in \offw{p}{p}$, $1\leq p \leq \infty$, then $T_t :
  L^p(w) \rightarrow L^p(w)$ is uniformly bounded.

\item If $1\leq p \leq q \leq r \leq \infty$, $T_t\in \offw{q}{r}$, and
  $S_t\in \offw{p}{q}$, then $T_t\circ S_t \in \offw{p}{r}$.
\end{enumerate}
\end{lemma}

\begin{remark}
If $p<q$, then $T_t \in \offw{p}{q}$ does not guarantee that $T_t$ is
bounded from $L^p(w)$ to $L^q(w)$.
\end{remark}

\begin{remark}\label{rem:complex-offd}
Since complex sectors $\Sigma_\mu$, $0\le\mu<\pi$, are closed under
addition, the proof of Lemma~\ref{lemma:unif-comp} extends to
give off-diagonal estimates on complex sectors $\offwmu{p}{q}{\mu}$.
\end{remark}

\begin{defi} \label{defn:full-offdiagonal}
Given $1\leq p \leq q \leq \infty$, a family of operators
$\{T_t\}$ satisfies full
off-diagonal estimates from $L^p(w)$ to $L^q(w)$, denoted by
\[ T_t \in \fullw{p}{q}, \]
if there exist constants $C,\,c,\,\theta>0$
such that given any closed sets $E,\,F$,
\[ \|T_t(f\bigchi_E)\bigchi_F\|_{L^q(w)} \leq Ct^{-\theta}
e^{-\frac{c d^2(E,F)}{t}}\|f\chi_E\|_{L^p(w)}. \]
\end{defi}

The connection between full off-diagonal estimates and off-diagonal
estimates on balls is given in the following lemma from \cite[Section~3.1]{auscher-martell07}.
\begin{lemma} \label{lemma:full} 
Given $1\leq p \leq q \leq \infty$:
\begin{enumerate}

\item if $T_t\in \fullw{p}{q}$, then
  $T_t : L^p(w) \rightarrow L^q(w)$ is uniformly bounded;

\item $T_t\in \fullw{p}{p}$ if and only if $T_t \in \offw{p}{p}$.

\end{enumerate}
\end{lemma}

\medskip

The importance of  off-diagonal estimates is that they will let us
prove weighted norm inequalities for the operators we are
interested in.  To do so we will make repeated use of
two results first proved in~\cite{auscher-martell07b}; however, we
will use special cases of these results as given in~\cite[Theorems~2.2
and~2.4]{auscher-martell06}.

\begin{theor} \label{theorem:2.2}
 Given $w\in A_2$ and $1\leq p_0<q_0\le \infty$, let $T$
be a sublinear operator acting on $L^{p_0}(w)$,  $\{\A_r\}_{r>0}$ a family of operators acting from a
subspace $\D$ of $L^{p_0}(w)$ into $L^{p_0}(w)$, and $S$ an operator from
$\D$ into the space of measurable functions on $\rn$.  Suppose that
for every $f\in \D$ and ball $B$ with radius $r$,
\begin{equation} \label{eqn:2.1}
\bigg(\avgint_B |T(I-\A_r)f|^{p_0}\,dw\bigg)^{1/p_0}
\leq
\sum_{j\geq 1} g(j) \bigg(\avgint_{2^{j+1}B}
|Sf|^{p_0}\,dw\bigg)^{1/p_0}
\end{equation}
and
\begin{equation} \label{eqn:2.2}
\bigg(\avgint_B |T\A_{r}f|^{q_0}\,dw\bigg)^{1/q_0}
\leq
\sum_{j\geq 1} g(j) \bigg(\avgint_{2^{j+1}B}
|Tf|^{p_0}\,dw\bigg)^{1/p_0},
\end{equation}
where $\sum g(j) < \infty$.   Then for every $p$, $p_0<p<q_0$, and
weights 
\[ v\in A_{p/p_0}(w)\cap RH_{(q_0/p)'}(w), \]
there is a constant
$C$ such that for all $f\in \D$,
\[ \|Tf\|_{L^p(v\,dw)} \leq C\|Sf\|_{L^p(v\,dw)}. \]
\end{theor}

\begin{remark}
  In Theorem~\ref{theorem:2.2} and Theorem~\ref{theorem:2.4} below,
  the case $q_0 = \infty$ is understood in the sense that the
  $L^{q_0}(w)$-average is replaced by the essential supremum. Also
  in Theorem~\ref{theorem:2.2}, if $q_0=\infty$, then the condition on $v$
  becomes $v\in A_{p/p_0}$.
\end{remark}

\begin{theor} \label{theorem:2.4}
Given $w\in A_2$ with doubling order $D$, and $1\leq p_0<q_0\leq \infty$, let $T : L^{q_0}(w) \rightarrow
L^{q_0}(w)$ be a sublinear operator, and $\{\A_r\}_{r>0}$ a family of
linear operators acting from
$L_c^\infty$ into $L^{q_0}(w)$.
  Suppose that
for every ball $B$ with radius $r$,  $f\in L_c^\infty$ with
$\supp(f)\subset B$ and $j\geq 2$,
\begin{equation} \label{eqn:2.4}
\bigg(\avgint_{C_j(B)} |T(I-\A_r)f|^{p_0}\,dw\bigg)^{1/p_0}
\leq
 g(j) \bigg(\avgint_{B}
|f|^{p_0}\,dw\bigg)^{1/p_0}.
\end{equation}
Suppose further that for every $j\geq 1$,
\begin{equation} \label{eqn:2.5}
\bigg(\avgint_{C_j(B)} |\A_{r}f|^{q_0}\,dw\bigg)^{1/q_0}
\leq
g(j) \bigg(\avgint_{B}
|f|^{p_0}\,dw\bigg)^{1/p_0},
\end{equation}
where $\sum g(j)2^{Dj} < \infty$.    Then for all $p$,
$p_0<p<q_0$, there exists a constant $C$ such that for all $f\in L^\infty_c$,
\[ \|Tf\|_{L^p(w)} \leq C\|f\|_{L^p(w)}. \]
\end{theor}

\section{Off-diagonal estimates for the semigroup $e^{-tL_w}$}
\label{section:semigroup}

In this section we consider off-diagonal estimates for the semigroup associated to $L_w$.  Throughout
this and subsequent sections, let $w\in A_2$ and $A\in
\E_n(w,\Lambda,\lambda)$ be fixed.  Our goal is to characterize the
set of pairs $(p,q)$, $p\leq q$ such that these operators are in
$\offw{p}{q}$.  By Theorem~\ref{thm:L2-gaffney} we have that
\[ e^{-tL_w} \in \fullw{2}{2} \subset
\offw{2}{2}. \]
We will show that in the $(p,q)$-plane this set contains a right
triangle:  see Figure~\ref{figure:triangle}.
%

\begin{figure}[h]
\caption{$(p,q)$ such that $e^{-tL_w} \in \offw{p}{q}$}
\label{figure:triangle}
{\begin{pgfpicture}{0cm}{-0.5cm}{6cm}{6cm}

\definecolor{mygray}{gray}{0.70}

\color{mygray}

\color{mygray}
\pgfmoveto{\pgfxy(1,1)}
\pgflineto{\pgfxy(5,5)}
\pgflineto{\pgfxy(1,5)}
\pgflineto{\pgfxy(1,1)}
\pgffill

\color{black}
\pgfmoveto{\pgfxy(1,1)}
\pgflineto{\pgfxy(5,5)}
\pgflineto{\pgfxy(1,5)}
\pgflineto{\pgfxy(1,1)}
\pgfstroke

\pgfmoveto{\pgfxy(-0.1,0.4)}
\pgflineto{\pgfxy(5.5,0.4)}
\pgfstroke

\pgfmoveto{\pgfxy(0.4,-0.1)}
\pgflineto{\pgfxy(0.4,5.5)}
\pgfstroke

\pgfcircle[fill]{\pgfxy(1.7,3.8)}{1pt}
\pgfputat{\pgfxy(1.9,3.8)}{\pgfbox[left,center]{$(p,q)$}}

\pgfputat{\pgfxy(5.4,0.05)}{\pgfbox[center,center]{$p$}}

\pgfputat{\pgfxy(0.1,5.3)}{\pgfbox[center,center]{$q$}}

\end{pgfpicture}}
\end{figure}

Let  $\widetilde \J(L_w)\subset [1,\infty]$  be the set of all
exponents $p$ such that $e^{-t\,L_w} : L^p(w)\rightarrow L^p(w)$ is
uniformly bounded for all $t>0$.    By Theorem~\ref{thm:L2-gaffney} and Lemma~\ref{lemma:full}, $2\in
\widetilde \J(L_w)$, and if  it contains more than one point, then by interpolation $\widetilde
\J(L_w)$ is an interval.  The set of pairs $(p,q)$  such that $e^{-t\,L_w}
\in \offw{p}{p}$ is completely characterized by the next result.

\begin{prop}\label{prop:J}
There exists an interval $\J(L_w) \subset [1,\infty]$ such that
$p,q \in \J(L_w)$ if and only if $e^{-t\,L_w}\in \offw{p}{q}$.
Furthermore, $\J(L_w)$ has the following properties:
\begin{enumerate}
\item $\J(L_w)\subset \widetilde\J(L_w)$;

\item $\Int\J(L_w)=\Int\widetilde\J(L_w)$;

\item if $p_-(L_w)$ and $p_+(L_w)$ are respectively the left and right endpoints
  of $\J(L_w)$, then $p_-(L_w)\le (2^*_w)'$ and $p_+(L_w)\ge  2^*_w$, where $2^*_w$ is as in Theorem~\textup{\ref{thm:wtd-poincare}}.
  In particular,  $2 \in \Int(\J(L_w))$.
\end{enumerate}
\end{prop}

\begin{remark} \label{remark:interval}
The smaller the value of $r_w$, the better our bounds on the size of
the set
  $\J(L_w)$.  In the limiting case when $w\in A_1$, we have that
  $p_-(L_w)\le \frac{2n}{n+2}$ and $p_+(L_w)\ge \frac{2n}{n-2}$.  These values
  should be compared to the estimates in \cite[Corollary~4.6]{auscher07} for the
  non-degenerate case that corresponds to the case $w=1$.
\end{remark}

We get two corollaries to Proposition~\ref{prop:J}.  The first gives
us weighted off-diagonal estimates.

\begin{corol}\label{corollary-weighted-offd} 
  Let $p_-(L_w)<p\le q<p_+(L_w)$.  If
  $v\in A_{p/p_-(L_w)}(w)\cap RH_{(p_+(L_w)/q)'}(w)$, then 
  $e^{-tL_w}\in \mathcal{O}\big(L^{p}(v\,dw)\rightarrow L^{q}(v\,dw)\big)$.
\end{corol}

\begin{proof}
  By Proposition~\ref{prop:J}, if $p_-(L_w)<p\le q<p_+(L_w)$, then
  $e^{-tL_w} \in \offw{p}{q}$.  Therefore,
  by~\cite[Proposition~2.6]{auscher-martell07}, if $v\in
  A_{p/p_-(L_w)}(w)\cap RH_{(p_+(L_w)/q)'}(w)$, then we have that
  $e^{-tL_w}\in \mathcal{O}\big(L^{p}(v\,dw) \rightarrow L^{q}(v\,dw)\big)$.
\end{proof}

As our second corollary we get
off-diagonal estimates for the holomorphic extension of the semigroup.

\begin{corol} \label{cor:holomorphic} For any $\nu$,
  $0<\nu<\frac{\pi}{2}-\vartheta$,
  and for any $p\leq q$ such that $e^{-tL_w} \in \offw{p}{q}$, then
  for all $m\in \NN\cup\{ 0\}$,
  $(zL_w)^me^{-zL_w} \in \offwmu{p}{q}{\nu}$.
\end{corol}

\begin{proof}
This follows from~\cite[Theorem~4.3]{auscher-martell07} and the fact
that, by
Theorem~\ref{thm:L2-gaffney},  for these values of $z$, $e^{-zL_w} \in \fullw{2}{2}$.
\end{proof}

\bigskip

\begin{proof}[Proof of Proposition~\textup{\ref{prop:J}}]
  Fix $2<q<2^*_w$ (If $w\in A_1$ we let $q=2^*_w=2^*$.) We will show
  that $e^{-t\,L_w}\in\offw{2}{q}$.  Given this, then we also
have that $e^{-t\,L_w}\in\offw{q'}{2}$.  For if  $L_w^*$ is the
adjoint of $L_w$ (with respect to $L^2(w)$), then
$L_w^*=-w^{-1}\,\div(A^*\,\nabla f)$ and the same estimates hold for
$L_w^*$.  Hence,
$e^{-t\,L_w^*}\in\offw{2}{q}$, and so by Lemma~\ref{lemma:duality},
$e^{-t\,L_w}\in\offw{q'}{2}$.      Since $e^{-tL_w}$ is a semigroup,
by Lemma~\ref{lemma:unif-comp} we have that
$e^{-t\,L_w}\in\offw{q'}{q}$.  Therefore, by \cite[Proposition
4.1]{auscher-martell07}, we have that there exists an interval $\J(L_w)$
and Properties~(1) and~(2) hold.  Moreover, we have that
$[q',q] \subset \J(L_w)$, so if we let $q\to 2^*_w$, then
we immediately get Property~(3).

\medskip

It therefore remains to prove that $e^{-t\,L_w}\in\offw{2}{q}$.  We
first show \eqref{w:off:B-B}.  Fix $B$ and for brevity write $r=r(B)$
and $C_j=C_j(B)$. By our choice of $q$ the Poincar\'e
inequality~\eqref{w-Poincare} holds.  Moreover, as we noted above
$e^{-t\,L_w}, \,\sqrt{t}\,\nabla e^{-t\,L_w}\in\offw{2}{2}$; we may
assume that the same exponents $\theta_1$, $\theta_2$ hold for both
operators.  We thus get that
\begin{align*}
&\left(\aver{B} |e^{-t\,L_w}( \bigchi_B \, f)|^{q}\,dw\right)^{\frac1q}
\\
&\qquad
\le
\big|\big(e^{-t\,L_w}( \bigchi_B \, f)\big)_{B,w} \big|+
\left(\aver{B} \big|e^{-t\,L_w}( \bigchi_B \, f)(x) - \big(e^{-t\,L_w}( \bigchi_B \, f)\big)_{B,w} \big|^{q}\,dw(x)\right)^{\frac 1 q}
\\
&\qquad
\lesssim
\left(\aver{B} |e^{-t\,L_w}( \bigchi_B \, f) |^{2}\,dw\right)^{\frac12}
+
r\,\left(\aver{B} |\nabla\,e^{-t\,L_w}( \bigchi_B \, f) |^{2}\,dw\right)^{\frac12}
\\
&\qquad
\lesssim
\left(1+\frac{r}{\sqrt{t}}\right)\,\dec{\frac{r}{\sqrt{t}}}^{\theta_2} \,\left(\aver{B} |f|^{2}\,dw\right)^{\frac12}
\\
&\qquad
\lesssim
\dec{\frac{r}{\sqrt{t}}}^{1+\theta_2} \,\left(\aver{B} |f|^{2}\,dw\right)^{\frac12}.
\end{align*}

\medskip
The proof that \eqref{w:off:C-B} holds is gotten by nearly the same
argument:
\begin{align*}
&\left(\aver{B} |e^{-t\,L_w}( \bigchi_{C_j} \, f)|^{q}\,dw\right)^{\frac1q}
\\
&\qquad
\le
\big|\big(e^{-t\,L_w}( \bigchi_{C_j} \, f)\big)_{B,w} \big|+
\left(\aver{B} \big|e^{-t\,L_w}( \bigchi_{C_j} \, f)(x) - \big(e^{-t\,L_w}( \bigchi_{C_j} \, f)\big)_{B,w} \big|^{q}\,dw(x)\right)^{\frac 1 q}
\\
&\qquad
\lesssim
\left(\aver{B} |e^{-t\,L_w}( \bigchi_{C_j} \, f) |^{2}\,dw\right)^{\frac12}
+
r\,\left(\aver{B} |\nabla\,e^{-t\,L_w}( \bigchi_{C_j} \, f) |^{2}\,dw\right)^{\frac12}
\\
&\qquad
\lesssim
2^{j\,\theta_1}\,
\left(1+\frac{r}{\sqrt{t}}\right)\,\dec{\frac{2^j\,r}{\sqrt{t}}}^{\theta_2} \,\expt{-\frac{c\,4^{j}\,r^2}{t}}\,\bigg(\aver{C_j} |f|^{2}\,dw\bigg)^{\frac12}
\\
&\qquad
\lesssim
2^{j\,\theta_1}\,
\dec{\frac{2^j\,r}{\sqrt{t}}}^{1+\theta_2} \,\expt{-\frac{c\,4^{j}\,r^2}{t}}\,\left(\aver{C_j} |f|^{2}\,dw\right)^{\frac12}.
\end{align*}

\medskip

Finally,  to prove that \eqref{w:off:B-C} holds we use a
covering argument.  Fix $j\geq 2$; then  we can cover the annulus $C_j$
by a collection of balls $\{B_k\}_{k=1}^N$, $r(B_k)=2^{j-2}\,r$,  with
centers $x_{B_k}\in C_j$. The number of balls required, $N$, depends only on the dimension.
For any such ball, since $dw$ is a doubling measure we have that
\begin{align*}
&\left(\aver{B_k} |e^{-t\,L_w}( \bigchi_{B} \, f)|^{q}\,dw\right)^{\frac1q}
\\
&\quad
\le
\big|\big(e^{-t\,L_w}( \bigchi_{B} \, f)\big)_{B_k,w} \big|+
\left(\aver{B_k} \big|e^{-t\,L_w}( \bigchi_B \, f)(x) - \big(e^{-t\,L_w}( \bigchi_{B} \, f)\big)_{B_k,w} \big|^{q}\,dw(x)\right)^{\frac 1 q}
\\
&\quad
\lesssim
\left(\aver{B_k} |e^{-t\,L_w}( \bigchi_{B} \, f) |^{2}\,dw\right)^{\frac12}
+
r(B_k)\,\left(\aver{B_k} |\nabla\,e^{-t\,L_w}( \bigchi_{B} \, f) |^{2}\,dw\right)^{\frac12}
\\
&\quad
\lesssim
\left(\aver{2^{j+2}\,B\setminus 2^{j-1}\,B} |e^{-t\,L_w}( \bigchi_{B}
  \, f) |^{2}\,dw\right)^{\frac12} \\
& \quad \qquad +
2^j\,r\,\left(\aver{2^{j+2}\,B\setminus 2^{j-1}\,B} |\nabla\,e^{-t\,L_w}( \bigchi_{B} \, f) |^{2}\,dw\right)^{\frac12}.
\end{align*}

If $j\ge 3$, then $2^{j+2}\,B\setminus 2^{j-1}\,B=C_{j+1}\cup
C_j\cup C_{j-1}$; then to estimate the last two terms we use the fact that $e^{-t\,L_w}, \,\sqrt{t}\,\nabla
e^{-t\,L_w}\in\offw{2}{2}$ and apply \eqref{w:off:B-C} with $p=q=2$ in each
annulus $C_i$, $j-1\leq i \leq j+1$.  (These annuli have comparable
measure since $dw$ is a doubling measure so we can divide the average up
into three averages).   If $j=2$, then
$2^{4}\,B\setminus 2\,B=C_{3}\cup C_2\cup (4\,B\setminus 2\,B)$.  On
$C_3$ and $C_2$ we argue as before using \eqref{w:off:B-C}.  On $4\,B\setminus B$ we
apply \cite[Lemma 6.1]{auscher-martell07}.  (We note that in the
notation there $\widehat C_1(B)=4\,B\setminus 2\,B$.)

If we combine all of
these estimates, we get that for every $j\ge 2$,
\begin{align*}
\left(\aver{B_k} |e^{-t\,L_w}( \bigchi_{B} \, f)|^{q}\,dw\right)^{\frac1q}
&
\lesssim
2^{j\,\theta_1}\,
\left(1+\frac{2^j\,r}{\sqrt{t}}\right)\,\dec{\frac{2^j\,r}{\sqrt{t}}}^{\theta_2} \,\left(\aver{B} |f|^{2}\,dw\right)^{\frac12}
\\
&\qquad
\lesssim
2^{j\,\theta_1}\,
\dec{\frac{2^j\,r}{\sqrt{t}}}^{1+\theta_2}\,\expt{-\frac{c\,4^{j}\,r^2}{t}} \,\left(\aver{B} |f|^{2}\,dw\right)^{\frac12}.
\end{align*}
Since $C_j\subset\bigcup_k B_k$, we can sum in $k$ to get
\begin{multline*}
\left(\aver{C_j(B)}|e^{-t\,L_w}( \bigchi_B \, f) |^{q}\,dw\right)^{\frac 1 q}
\lesssim
\sum_{k=1}^N
\left(\aver{B_k} |e^{-t\,L_w}( \bigchi_{B} \, f)|^{q}\,dw\right)^{\frac1q}
\\
\lesssim
2^{j\,\theta_1}\,
\dec{\frac{2^j\,r}{\sqrt{t}}}^{1+\theta_2}\,\expt{-\frac{c\,4^{j}\,r^2}{t}} \,\left(\aver{B} |f|^{2}\,dw\right)^{\frac12}.
\end{multline*}
This completes the proof that $e^{-tL_w} \in \offw{2}{q}$.
\end{proof}

\section{The functional calculus}
\label{section:functional}

In this section we show that the operator $L_w$ has an  $L^p(w)$ holomorphic
functional calculus.   As we discussed in Section~\ref{section:prelim}
above, we know already that if $\varphi$ is a bounded holomorphic
function on $\Sigma_\mu$, $\mu \in (\vartheta,\pi)$, then
$\varphi(L_w)$ is a bounded operator on $L^2(w)$.
Recall that  for any $\mu \in
(\vartheta ,\pi )$, we say that $\varphi \in
\mathcal{H}_{0}^{\infty}(\Sigma_\mu)$ if for some $c,\,s>0$
\begin{equation} \label{eqn:phi-decay}
|\varphi(z)| \leq c|z|^s(1+|z|)^{-2s}, \qquad z \in \Sigma_\mu.
\end{equation}
We  say that $L_w$ has a  bounded holomorphic functional calculus
on $L^p(w)$ if for any such $\varphi$,
\begin{equation}  \label{eq:fcX}
\Vert \varphi (L_{w})f\Vert _{L^p(w)}\leq C\,\Vert \varphi \Vert _{\infty
}\,\Vert f\Vert _{L^p(w)},\qquad f\in  L^p(w)\cap L^{2}(w),
\end{equation}
where $C$ depends only on $p$, $w$, $\vartheta $ and $\mu $ (but not on the
decay of $\varphi $).  By a standard density argument, \eqref{eq:fcX} implies
that $\varphi (L_{w})$ extends to a bounded operator on all of
$L^p(w)$.  Furthermore,  we then have that
this inequality holds if $\varphi$ is any bounded holomorphic function.  For the
details of this extension, see~\cite{haase06,mcintosh86}.

\begin{prop} \label{prop:B-K:weights}
Let $p_-(L_w)<p<p_+(L_w)$
  and $\mu \in (\vartheta ,\pi )$. Then for any $\varphi \in
\mathcal{H}_{0}^{\infty}(\Sigma_\mu)$,
\begin{equation}  \label{eq:fcw}
\| \varphi (L_{w})f\| _{L^{p}(w)}\leq C\,\| \varphi \| _{\infty
}\,\| f\| _{L^{p}(w)},
\end{equation}%
with $C$ independent of $\varphi $ and $f$. Hence, $L_{w}$ has a bounded
holomorphic functional calculus on $L^{p}(w)$.
Moreover, if $v \in A_{p/p_-(L_w)}(w)\cap
RH_{(p_+(L_w)/p)'}(w) $  then  $L_{w}$ also has a bounded
holomorphic functional calculus on $L^{p}(v\,dw)$:
\begin{equation}
\| \varphi (L_{w})f\| _{L^{p}(v\,dw)}\leq C\,\| \varphi \| _{\infty
}\,\| f\| _{L^{p}(v\,dw)},  \label{eq:fc-vw}
\end{equation}%
with $C$ independent of $\varphi $ and $f$.

\end{prop}

\begin{proof}
  For brevity, let $p_-=p_-(L_w)$ and $p_+=p_+(L_w)$. By density it
  will suffice to assume that $f\in L_c^\infty$. Fix
  $\varphi\in \H_0^\infty(\Sigma_\mu)$; by linearity we may assume
  that $\left\| \varphi \right\| _{\infty }=1$.

  We divide the proof into two steps. We first obtain \eqref{eq:fcw}
  for $p_-<p<2$ by applying Theorem~\ref{theorem:2.4} and following the ideas in \cite{auscher07}. To do so,  we will
  pick $q_0=2$ and $p_0>p_-$ arbitrarily close to $p_-$. In the second
  step, using some ideas from \cite{auscher-martell06},   we will use Theorem~\ref{theorem:2.2} to get \eqref{eq:fc-vw};
  in particular this yields \eqref{eq:fcw} for every $2<p<p_+$ by
  taking $v\equiv 1$. To apply Theorem~\ref{theorem:2.2} we will
  choose $p_0>p_-$ arbitrarily close to $p_-$ and $q_0<p_+$
arbitrarily close to  $p_+$.  We will also use the fact that $\varphi(L_w)$ is
  bounded on $L^{p_0}(w)$; this follows from the first step choosing
  $p_-<p_0<2$.

\medskip

To apply Theorem~\ref{theorem:2.4},   fix $p_-<p_0<p<2$ and let $q_0=2$,  $T=\varphi (L_{w}) $, and
\begin{equation}  \label{eqn:Atm}
\mathcal{A}_{r}f( x) =\big( I-( I-e^{-r^{2}L_{w}})^{m}\big) f( x),
\end{equation}%
where $m$ is a positive integer that will be chosen below.
We first show that inequality~\eqref{eqn:2.5} holds.
 By Proposition \ref{prop:J} we have that $e^{-tL_{w}}\in
 \offw{p_0}{2}$. Using  
\begin{equation}\label{eqn:Asum}
\mathcal{A}_r=
\sum_{k=1}^{m}\binom{m}{k} (-1)^{k+1} e^{-kr^{2}L_{w}},
\end{equation}
and that for each fixed $m$ and $1\le k\le m$
$$
\Upsilon \left( \frac{r}{\sqrt{k}t}\right) \leq \sqrt{m}\Upsilon \left(
\frac{r}{t}\right) 
\quad\text{and}\quad
\exp \left( -\frac{c}{k}\frac{4^{j}r^{2}}{t^{2}}%
\right) \leq \exp \left( -\frac{c}{m}\frac{4^{j}r^{2}}{t^{2}}\right),
$$
Proposition~\ref{prop:J} implies
\begin{equation}\label{eqn:A-od}
\mathcal{A}_{r}\in \offw{p}{q},\qquad\text{for all }p_-(L_w)<p\le q<p_+(L_w).
\end{equation}
In particular, we have that $\mathcal{A}_{r}\in \offw{p_0}{2}$.   Thus, given
any ball $B$ with radius $r$, if
$\supp(f)\subset B$, then for all $j\geq 1$,
\begin{equation}\label{eqn:2.5Ar}
\bigg(
\avgint%
_{C_{j}\left( B\right) }\left\vert \mathcal{A}_{r}f\right\vert ^{2}dw\bigg) ^{1/2}
\lesssim 2^{j\theta _{1}}\Upsilon \left(
2^{j}\right) ^{\theta _{2}}e^{-c4^{j}}\left(
\avgint%
_{B}\left\vert f\right\vert ^{p_0}dw\right) ^{1/p_0}.
\end{equation}%
This establishes \eqref{eqn:2.5}  with $g\left( j\right) =C\,2^{j\left( \theta
_{1}+\theta _{2}\right) }e^{-c4^{j}}$, for in this case we have that 
$$\sum_{j\geq1}2^{j(\theta _{1}+\theta _{2}+D) }e^{-c4^{j}}<\infty ,$$
where $D$ is
the doubling constant of $w$.

\medskip

We next prove that \eqref{eqn:2.4} holds.  Since
$\varphi(z)(1-e^{-r^2z})^m\in
\H_0^\infty(\Sigma_{\{\min\{\mu,\pi/2\}\} })$,
by the functional calculus representation~\eqref{eqn:L2-holo-rep}
we have that

\begin{equation*}
\varphi \left( L_{w}\right) \left( I-\mathcal{A}_{r }\right)
f= \int_{\Gamma }e^{-z\,L_{w}}f\,\eta (z)\,dz,
\end{equation*}
where $\Gamma=\partial \Sigma_{\frac{\pi}{2}-\theta}$, with $0<\vartheta <\theta <\nu <\min\{\mu,\pi/2\}$, and we  choose $\theta$ so that the hypotheses of Corollary~\ref{cor:holomorphic} are
satisfied for $z\in \Gamma$.
Moreover, we have the estimate
\begin{equation*}
\left\vert \eta \left( z\right) \right\vert
\lesssim \frac{r^{2m}}{\left\vert
z\right\vert ^{m+1}};
\end{equation*}
see \cite[Section 5.1]{auscher07} for details.

We can now argue as follows:  given a ball $B$ with radius $r$, for
each $j\geq 2$, by Minkowski's inequality and
Corollary~\ref{cor:holomorphic} (since $p_0 \in \Int \J(L_w)$),
\begin{align}
&\bigg(
\avgint%
_{C_{j}\left( B\right) }\left\vert \varphi \left( L_{w}\right) \left( I-%
\mathcal{A}_{r\left( B\right) }\right) f\right\vert ^{p_0}dw\bigg) ^{1/p_0}
\notag \\
& \qquad \qquad =\bigg(
\avgint%
_{C_{j}\left( B\right) }\left\vert \int_{\Gamma }e^{-z\,L_{w}}f\,\eta
(z)\,dz\right\vert ^{p_0}dw\bigg) ^{1/p_0}  \notag \\
& \qquad \qquad \lesssim\int_{\Gamma }\bigg(
\avgint%
_{C_{j}\left( B\right) }\left\vert e^{-z\,L_{w}}f\,\right\vert ^{p_0}dw\bigg)
^{1/p_0}\frac{r^{2m}}{\left\vert z\right\vert ^{m+1}}\,\left\vert
dz\right\vert  \notag \\
&\qquad \qquad \lesssim \bigg(
\avgint%
_{B}\left\vert f\,\right\vert ^{p_0}dw\bigg) ^{1/p_0}\int_{\Gamma }\frac{r^{2m}%
}{\left\vert z\right\vert ^{m+1}}2^{j\theta _{1}}\Upsilon \left( \frac{%
2^{j}r}{\sqrt{\left\vert z\right\vert }}\right) ^{\theta _{2}}e^{-c\frac{%
r^{2}}{\left\vert z\right\vert }4^{j}}\,\left\vert dz\right\vert  \notag \\
& \qquad \qquad = \bigg(
\avgint%
_{B}\left\vert f\,\right\vert ^{p_0}dw\bigg) ^{1/p_0}2^{j\left( \theta
_{1}-2m\right) }\int_{0}^{\infty }\sigma^{2m}\Upsilon \left(
\sigma\right) ^{\theta _{2}}e^{-c\sigma ^{2}}\,\frac{d\sigma}{%
\sigma}  \notag \\
&\qquad \qquad \lesssim  2^{j\left( \theta _{1}-2m\right) }\bigg(
\avgint%
_{B}\left\vert f\,\right\vert ^{p_0}dw\bigg) ^{1/p_0}; \label{est-down}
\end{align}
the final inequality holds (i.e., the integral in $\sigma$ converges)
provided $2m> \theta _{2}$.
Moreover, if we  choose $2m>\theta _{1}+D$, we
have that \eqref{eqn:2.4} holds with $g\left( j\right) =C\,2^{\left( j-1\right)
\left( \theta _{1}-2m\right) }$ and
\begin{equation*}
\sum_{j\geq 2}g\left( j\right) 2^{jD}\lesssim
\sum_{j\geq 2}2^{j\left( \theta _{1}+D-2m\right) }<\infty .
\end{equation*}

We have shown that inequalities \eqref{eqn:2.4}  and \eqref{eqn:2.5}
hold, and so by Theorem \ref{theorem:2.4}
inequality~\eqref{eq:fcw} holds for all $p$ such that $
p_{-}<p\le 2$.

\medskip

We will now apply Theorem~\ref{theorem:2.2} to show
that~\eqref{eq:fc-vw} holds for $p_-<p<p_+$.  (Inequality~\eqref{eq:fcw}
then follows for $2<p<p_+$ if we take $v\equiv 1$.)
Fix $p$, $p_-<p<p_+$ and  $v \in A_{p/p_-}(w)\cap
RH_{(p_+/p)'}(w) $. By the openness properties of the $A_q$ and $RH_s$
classes there exist $p_0$, $q_0$ such that
\[ p_-<p_0<\min\{p,2\}\le p<q_0<p_+,
\qquad
v \in A_{p/p_0}(w)\cap
RH_{(q_0/p)'}(w).  \]
Let  $T=\varphi \left(
L_{w}\right) $, $\mathcal{A}_{r}=I-( I-e^{-r^{2}L_{w}}) ^{m}$, $S=I$,
and fix the above values of $p_{0}$ and $q_0$. By the previous argument
we have that $\varphi \left( L_{w}\right) $ is bounded on $L^{p_{0}}\left(
w\right) $.

We first show that~\eqref{eqn:2.1} holds.   Fix a ball $B$ and
decompose $f$ as
\begin{equation}\label{decomp-f}
f=\sum_{j\geq 1}f\chi _{C_{j}(B)}:=\sum_{j\geq 1}f_{j}.
\end{equation}
Then, by the
same functional calculus argument as given above, we have that for
each $j$,
\begin{align*}
&\bigg(
\avgint%
_{B}\left\vert \varphi ( L_{w}) ( I-\mathcal{A}_{r})
f_{j}\right\vert ^{p_{0}}dw\bigg) ^{\frac{1}{p_{0}}} \\
&\qquad =\bigg(
\avgint%
_{B}\left\vert \int_{\Gamma }e^{-z\,L_{w}}f_{j}\,\eta (z)\,dz\right\vert
^{p_{0}}dw\bigg) ^{\frac{1}{p_{0}}} \\
& \qquad \lesssim \int_{\Gamma }\bigg(
\avgint%
_{B}\left\vert e^{-z\,L_{w}}f_{j}\,\right\vert ^{p_{0}}dw\bigg) ^{\frac{1}{p_{0}}}%
\frac{r^{2m}}{\left\vert z\right\vert ^{m+1}}\,\left\vert dz\right\vert \\
&\qquad \lesssim \bigg(
\avgint%
_{C_{j}\left( B\right) }\left\vert f\,\right\vert ^{p_{0}}dw\bigg) ^{\frac{1%
}{p_{0}}}2^{j\left( \theta _{1}-2m\right) }\int_{\Gamma }\bigg( \frac{2^{j}r%
}{\sqrt{\left\vert z\right\vert }}\bigg) ^{2m}\Upsilon \bigg( \frac{2^{j}r}{%
\sqrt{\left\vert z\right\vert }}\bigg)^{\theta _{2}}e^{-\frac{c4^{j}r^{2}}{%
\left\vert z\right\vert }}\,\frac{\left\vert dz\right\vert }{\left\vert
z\right\vert } \\
&\qquad \lesssim 2^{j\left( \theta _{1}-2m\right) }\bigg(
\avgint%
_{C_{j}\left( B\right) }\left\vert f\,\right\vert ^{p_{0}}dw\bigg) ^{\frac{1%
}{p_{0}}};
\end{align*}
the last inequality holds provided $2m>\theta _{2}$. Hence, since $2^{j+1}B\supset C_{j}$, by
Minkowski's inequality we have (since the sum $\sum f_{j}$ is finite for
$f\in L_{c}^{\infty }$)
\begin{align*}
\bigg(
\avgint%
_{B}\left\vert \varphi ( L_{w}) ( I-\mathcal{A}_{r }) f\right\vert ^{p_{0}}dw\bigg) ^{\frac{1}{p_{0}}}
& \leq \sum_{j\geq 1} \bigg(
\avgint%
_{B}\left\vert \varphi ( L_{w}) ( I-\mathcal{A}_{r}) f_{j}\right\vert ^{p_{0}}dw\bigg) ^{\frac{1}{p_{0}}} \\
&\lesssim \sum_{j\geq 1}2^{j\left( \theta _{1}-2m\right) }\bigg(
\avgint%
_{2^{j+1} B }\left\vert f\,\right\vert ^{p_{0}}dw\bigg) ^{%
\frac{1}{p_{0}}}.
\end{align*}
This establishes \eqref{eqn:2.1} with $g(j) =C\, 2^{j\left( \theta
    _{1}-2m\right) }$.  If we  take $2m>\max \left\{ \theta
_{1},\theta _{2}\right\} $, then $\sum g(j)<\infty$.

We now show that~\eqref{eqn:2.2} holds.  Fix a ball $B$ and $j\geq 1$.
Since  $\mathcal{A}_{r}\in \offw{p_0}{q_0}$ (see \eqref{eqn:A-od}),  
\begin{equation*}
\bigg( \avgint_{B}\big|\mathcal{A}_{r }\big(
\chi_{C_j(B)}\varphi (L_{w}) f\big)\big| ^{q_{0}}dw\bigg) ^{\frac{1}{q_0}}\lesssim 2^{j\theta
_{1}}\Upsilon \left( 2^{j}\right) ^{\theta _{2}}e^{-c4^{j}}\bigg(
\avgint%
_{C_{j}(B)}\left\vert \varphi \left( L_{w}\right) f\right\vert ^{p_{0}}d\mu
\bigg) ^{\frac{1}{p_0}}.
\end{equation*}%
Therefore, since $\varphi \left( L_{w}\right) $ and $
\mathcal{A}_{r}$ commute, by Minkowski's inequality  we obtain
\begin{equation*}
\left(
\avgint%
_{B}\left\vert \varphi \left( L_{w}\right) \mathcal{A}_{r}
f\right\vert ^{q_{0}}dw\right) ^{\frac{1}{q_0}}
\lesssim \sum_{j\geq 1}2^{j\left(
\theta _{1}+\theta _{2}\right) }e^{-c4^{j}}\bigg(
\avgint%
_{C_{j}(B)}\left\vert \varphi \left( L_{w}\right) f\right\vert ^{p_{0}}d\mu
\bigg) ^{\frac{1}{p_0}}.
\end{equation*}%
This establishes \eqref{eqn:2.2} with $g(j)=C\,2^{j\left(
\theta _{1}+\theta _{2}\right) }e^{-c4^{j}}$; again, $\sum
g(j)<\infty$.  Therefore, our proof is complete.
\end{proof}

\section{Square function estimates for the semigroup}
\label{section:square-function}

In this section we prove $L^p(w)$ norm inequalities for the vertical square
function associated to the semigroup $e^{-tL_w}$:
\[ g_{L_{w}}f( x)
=\bigg( \int_{0}^{\infty }\left\vert \left(
tL_{w}\right) ^{1/2}e^{-tL_{w}}f(x) \right\vert ^{2}\frac{dt}{t}%
\bigg) ^{1/2}. \]

\begin{prop} \label{prop:square-function}
Let $p_-(L_w) < p < p_+(L_w)$.  Then
\begin{equation} \label{eqn:square-function-unwtd}
 \left\Vert g_{L_{w}}f\right\Vert
_{L^{p}\left( w\right) }\approx \left\Vert f\right\Vert _{L^{p}\left(
    w\right)}.
\end{equation}
Conversely if for some $p$ the equivalence~\eqref{eqn:square-function-unwtd} holds,
  then  $p \in \tilde{\J}(L_w)$---i.e., the interior of the interval on
  which~\eqref{eqn:square-function-unwtd} holds is
  $(p_-(L_w),p_+(L_w))$.

Moreover, if $v \in A_{p/p_-(L_w)}(w)\cap
RH_{(p_+(L_w)/p)'}(w)$,  then
\begin{equation} \label{eqn:square-function-wt}
 \left\Vert g_{L_{w}}f\right\Vert
_{L^{p}(v\,dw) }\approx \left\Vert f\right\Vert _{L^{p}(v\,dw)}.
\end{equation}
\end{prop}


We note that the upper bounds in the previous result could be
obtained by combining Proposition \ref{prop:B-K:weights} with the
operator theory methods developed in \cite{couling96}. 
To reach a 
wider audience we present a self-contained harmonic analysis
proof. We will use an auxiliary Hilbert space related to square
functions, following the approach in~\cite{auscher-martell06}.  Let
$\mathbb{H}$ denote the Hilbert space
$L^{2}\left( \left( 0,\infty \right) ,\frac{dt}{t}\right) $ with norm
\begin{equation*}
\normH{h}
=\left( \int_{0}^{\infty }\left\vert h\left( t\right) \right\vert ^{2}%
\frac{dt}{t}\right) ^{\frac{1}{2}}.
\end{equation*}
In particular, we have that
\begin{equation*}
g_{L_{w}}f(x) =%
\normH{\varphi(L,\cdot)f(x)}%
\end{equation*}
where ${\varphi }\left( z,t\right) =\left( tz\right) ^{1/2}e^{-tz}$.
Furthermore, we define $L^p_{\mathbb{H}}(w)$ to be the space of
$\mathbb{H}$-valued functions with the norm
\[ \|h\|_{L^p_{\mathbb{H}}(w)}
= \bigg(\int_{\rn} \normH{h(x,\cdot)}^p\,dw(x) \bigg)^{\frac{1}{p}}. \]
The following lemma lets us extend scalar valued inequalities
to $\mathbb{H}$-valued inequalities. For a proof, see \cite[ Lemma~7.4]{auscher-martell06} and the references given there.

\begin{lemma}  \label{lemma-7.4}
Given a Borel measure  $\mu $ on
  $\mathbb{R}^{n}$,
  let $\mathcal{D}$ be a subspace of $\mathcal{M}$, the space of
  measurable functions in $\mathbb{R}^{n}$, and  let $S,\,T$ be linear
  operators from $\mathcal{D}$ into $\mathcal{M}$. Fix $1\leq
  p\leq q<\infty$ and  suppose there exists
  $C_{0}>0 $ such that for all $f\in \mathcal{D}$,
\begin{equation*}
\left\Vert Tf\right\Vert _{L^{q}\left( \mu \right) }\leq C_{0}\sum_{j\geq
1}\alpha _{j}\left\Vert Sf\right\Vert _{L^{p}\left( F_{j},\mu \right) },
\end{equation*}
where the $F_{j}$ are measurable subsets of $\mathbb{R}^{n}$ and $\alpha
_{j}\geq 0$. Then there is a $\mathbb{H}$-valued inequality with the same
constant: for all $f:\mathbb{R}^{n}\times \left( 0,\infty \right)
\longrightarrow \mathbb{C}$ such that for almost all $t>0$, $f\left( \cdot
,t\right) \in \mathcal{D}$,
\begin{equation*}
\left\Vert Tf\right\Vert _{L_{\mathbb{H}}^{q}\left( \mu \right) }\leq
C_{0}\sum_{j\geq 1}\alpha _{j}\left\Vert Sf\right\Vert _{L_{\mathbb{H}%
}^{p}\left( F_{j},\mu \right) }.
\end{equation*}
\end{lemma}

The extension of a linear operator  $T$ on $\co$-valued functions to
$\HH$-valued  functions is defined  for $x\in \RR^n$ and $t>0$ by $(T
h)(x,t)= T\big( h(\cdot,t)\big)(x)$, that is, $t$ can be considered
as a parameter and $T$ acts only on the variable in $\RR^n$.

\begin{proof}[Proof of Proposition~\textup{\ref{prop:square-function}}]
We shall first prove the upper bound inequalities. We first claim that
the upper bound inequality in \eqref{eqn:square-function-unwtd} holds
for $p=2$. Indeed, since
$\varphi(z)=z^{1/2}e^{-z}\in\H^\infty_0(\Sigma_\mu)$, it follows from
\eqref{eqn:holo-sqfe} that we have the bound
\[
\left\Vert g_{L_{w}}f\right\Vert _{L^{2}(w) } \lesssim \left\Vert
f\right\Vert _{L^{2}(w) }. \]

For brevity, let $p_-=p_-(L_w)$ and $p_+=p_+(L_w)$.  As in previous
proofs, we divide our proof into two steps.  We will first
prove the upper bound in \eqref{eqn:square-function-unwtd} for $p_-<p<2$ by applying
Theorem~\ref{theorem:2.4}.  Fix $p_{-}<p<q_{0}=2$, and let
$\mathcal{A}_{r}=I-( I-e^{-r^{2}L_{w}}) ^{m}$, where $m$ will be
chosen below.   Notice that, by \eqref{eqn:A-od}, $\mathcal{A}_{r}$ is bounded on $L^{q_0}(w)$ for each $m$.
Fix $f\in L_c^\infty$; the result for general $f\in L^p(w)$ then follows by a density argument.

We have that $(tL_{w}) ^{1/2}e^{-tL_{w}}( I-\mathcal{A}_{r})
f=\varphi( L_{w},t) f$, where
\[ {\varphi }( z,t) {=}(tz) ^{1/2}e^{-tz}( 1-e^{-r^{2}z})
^{m}. \]
Moreover, since $\varphi(\cdot,t)\in \mathcal{H}_{0}^{\infty }(\Sigma_{\{\min\{\mu,\pi/2\}\} })$,
by the functional calculus representation~\eqref{eqn:L2-holo-rep}
we have that
\begin{equation*}
\left( tL_{w}\right) ^{1/2}e^{-tL_{w}}\left( I-\mathcal{A}_{r}\right)
f=\int_{\Gamma}\eta \left( z,t\right)
e^{-zL_{w}}f~dz,
\end{equation*}%
where $\Gamma=\partial \Sigma_{\frac{\pi}{2}-\theta}$, with $0<\vartheta <\theta <\nu <\min\{\mu,\pi/2\}$, and we  choose $\theta$ so that the hypotheses of Corollary~\ref{cor:holomorphic} are 
satisfied for $z\in \Gamma$.
Moreover, we have the estimate (see \cite{auscher07, auscher-martell06})
\begin{equation*}
\left\vert \eta \left( z,t\right) \right\vert 
\lesssim \frac{t^{\frac{1}{2}}r^{2m}}{\left(
\left\vert z\right\vert +t\right) ^{m+\frac{3}{2}}},
\qquad z\in \Gamma;
\end{equation*}
therefore,
\begin{equation} \label{eqn:eta-bound}
\normH{\eta(z,\cdot)}%
=\left( \int_{0}^{\infty }\left\vert \eta \left( z,t\right) \right\vert ^{2}~%
\frac{dt}{t}\right) ^{1/2}\lesssim \frac{r^{2m}}{\left\vert z\right\vert
^{m+1}}.
\end{equation}

Now let $f\in L_{c}^{\infty }$ with $\supp\left( f\right)
\subset B$. For $j\geq 2$, we have
\begin{align}
& \bigg( \avgint_{C_j(B)}
\left\vert g_{L_{w}}\left( {I-}\mathcal{A}_{r}\right) f\right\vert
^{p}dw\bigg) ^{1/p}  \notag \\
&\qquad \qquad =\bigg( \avgint_{C_j(B)}%
\bigg\vert \bigg( \int_{0}^{\infty }
\bigg\vert \int_{\Gamma _{\frac{\pi }{2}%
-\theta }}\eta ( z,t) e^{-zL_{w}}f\, dz\bigg\vert ^{2}\frac{dt}{t}%
\bigg) ^{1/2}\bigg \vert ^{p}dw\bigg) ^{1/p}  \notag \\
& \qquad \qquad \leq \bigg( \avgint_{C_j(B)}%
\bigg\vert \int_{\Gamma _{\frac{\pi }{2}-\theta }}\vert
e^{-zL_{w}}f\vert
\normH{\eta(z,\cdot)}%
d\vert z\vert \bigg\vert ^{p}dw\bigg) ^{1/p}  \notag \\
&\qquad \qquad \lesssim \int_{\Gamma _{\frac{\pi }{2}-\theta }}
\bigg(\avgint_{C_j(B)}%
\vert e^{-zL_{w}}f\vert ^{p}dw\bigg) ^{1/p}\frac{r^{2m}}{%
\vert z\vert ^{m+1}}\,d\vert z\vert   \notag \\
&\qquad \qquad \lesssim 2^{j\theta _{1}}
\bigg( \avgint_{B}%
\vert f\vert ^{p}\,dw\bigg) ^{1/p}\int_{\Gamma _{\frac{\pi }{2}%
-\theta }}\Upsilon \bigg( \frac{2^{j}r}{\sqrt{\vert z\vert }}%
\bigg) ^{\theta _{2}}e^{-\frac{c4^{j}r^{2}}{\vert z\vert }}\frac{%
r^{2m}}{\vert z\vert ^{m}}\,\frac{d\vert z\vert }{%
\vert z\vert }  \notag \\
& \qquad \qquad \lesssim 2^{j\theta _{1}}4^{-mj}\bigg(
\avgint_{B}%
\vert f\vert ^{p}\,dw\bigg) ^{1/p};  \label{eqn:gLw-01}
\end{align}
in the second inequality we applied (\ref{eqn:eta-bound}) and
the off-diagonal estimates for $%
e^{-zL_{w}}$
from Corollary \ref{cor:holomorphic}, and the last inequality holds provided $2\,m>\theta_2$. Thus, if we take
$m>\theta _{1}+D$, where $D$ is the doubling order of $w$, the
operator $g_{L_{w}}$ satisfies (\ref{eqn:2.4}) in
Theorem~\ref{theorem:2.4} with
$g\left( j\right) =C\,2^{j(\theta _{1}-2m)}$.  Since
we already established~\eqref{eqn:2.5}
in~\eqref{eqn:2.5Ar} with
$g\left( j\right) =C\,2^{j\left( \theta _{1}+\theta _{2}\right)
}4^{-mj}$, the hypotheses of Theorem \ref%
{theorem:2.4} are satisfied if $m>\theta _{1}+\theta _{2}+D$.
Therefore,
for each $p_{-}<p< 2$ there
exists a constant $C$ such that
\begin{equation}
\left\Vert g_{L_{w}}f\right\Vert _{L^{p}(w) }\leq C\left\Vert
f\right\Vert _{L^{p}(w) }.  \label{ineq:gps}
\end{equation}%

\bigskip

In the second part of the proof we will show that if $p_-<p<p_+$ and 
$v \in A_{p/p_-}(w)\cap RH_{(p_+/p)'}(w)$, then the upper bound
inequality in \eqref{eqn:square-function-wt} holds.   If
we take $v\equiv 1$,
then we immediately get~\eqref{eqn:square-function-unwtd}.
To do so, first note that if we fix $p$ and $v$,  then by the openness
properties of weights  there exist $p_0$, $q_0$ such that
\[ p_{-}<p_{0}<\min\{p,2\} \leq \max\{p,2\}<q_{0}<p_{+} \]
and $v \in A_{p_0/p_-}(w)\cap
RH_{(q_0/p)'}(w)$.

We will apply Theorem~\ref{theorem:2.2} with $T=g_{L_w}$, $S=I$ and
$\D=L^{p_0}(w)$ (again, note that by \eqref{eqn:A-od}, $\mathcal{A}_r$ 
is bounded on $L^{p_0}(w)$).  We first prove that
inequality~\eqref{eqn:2.1} holds.  For each $j\geq 1$, let
$f_j=f\chi_{C_j(B)}$; then we can argue exactly as we did in the proof
of~\eqref{eqn:gLw-01}, exchanging the roles of $B$ and $C_j(B)$, to get
\[ \bigg( \avgint_B |g_{L_w}(I-\A_r)f_j|^p \,dw \bigg)^{\frac{1}{p}}
\lesssim 2^{j\theta_1}4^{-mj}
\bigg(\avgint_{2^{j+1}B} |f|^p\,dw \bigg)^{\frac{1}{p}}. \]
Inequality~\eqref{eqn:2.1} follows if we sum over all $j$ and take
$g(j)= 2^{j\theta_1}4^{-mj}$.

We will now show that inequality~\eqref{eqn:2.2} holds.  To do so, we need to
prove a vector-valued version of a key inequality.  By Proposition~\ref{prop:J}, given a
ball $B$ with radius $r$, then for all $j\geq 1$,  $g$ with $\supp(g)\subset C_j(B)$,
and  $1\leq k \leq m$,
\begin{equation} \label{eqn:scalar-offdiag}
 \bigg(\avgint_B |e^{-kr^2L_w}g|^{q_0}\,dw\bigg)^{\frac{1}{q_0}}
\leq C_02^{j(\theta_1+\theta_2)}e^{-\alpha 4^j}
\bigg(\avgint_{C_j(B)} |g|^{p_0}\,dw\bigg)^{\frac{1}{p_0}}.
\end{equation}
We now apply Lemma~\ref{lemma-7.4} with $S=I$ and $T :
L^{p_0}(w)\rightarrow L^{q_0}(w)$ given by
\[ Tg = (C_02^{j(\theta_1+\theta_2)}e^{-\alpha 4^j})^{-1}
\frac{ w(2^{j+1}B)^{\frac{1}{p_0}}}{w(B) ^{\frac{1}{q_0}}}
\chi_B e^{-kr^2L_w}(g\chi_{C_j(B)}). \]
This yields the $\mathbb{H}$-valued extension of~\eqref{eqn:scalar-offdiag}:  for
all $g\in L^{p_0}_{\mathbb{H}}(w)$ with $\supp(g(\cdot,t))\subset C_j(B)$, $t>0$, we have that
\begin{equation} \label{eqn:vector-offdiag}
  \bigg(\avgint_B \normH{e^{-kr^2L_w}g(x,\cdot)}^{q_0}\,dw\bigg)^{\frac{1}{q_0}}
\leq C_02^{j(\theta_1+\theta_2)}e^{-\alpha 4^j}
\bigg(\avgint_{C_j(B)}
\normH{g(x,\cdot)}^{p_0}\,dw\bigg)^{\frac{1}{p_0}}.
\end{equation}
Given an arbitrary $g\in L^{p_0}_{\mathbb{H}}(w)$, decompose it as
\[  g(x,t) = \sum_{j\geq 1} g(x,t)\chi_{C_j(B)}(x) = \sum_{j\geq 1}
g_j(x,t). \]
Then inequality~\eqref{eqn:vector-offdiag} yields
\begin{multline} \label{eqn:almost-there}
  \bigg(\avgint_B
  \normH{e^{-kr^2L_w}g(x,\cdot)}^{q_0}\,dw\bigg)^{\frac{1}{q_0}}
\leq \sum_{j\geq 1}
\bigg(\avgint_{B}
\normH{e^{-kr^2L_w} g_j(x,\cdot)}^{q_0}\,dw\bigg)^{\frac{1}{q_0}} \\
\lesssim  \sum_{j\geq 1} 2^{j(\theta_1+\theta_2)}e^{-\alpha 4^j}
\bigg(\avgint_{2^{j+1}B}
\normH{g(x,\cdot)}^{p_0}\,dw\bigg)^{\frac{1}{p_0}}.
\end{multline}

Define $g(x,t)=(tL_w)^{1/2}e^{-tL_w}f(x)$.  Then
$g_{L_w}f(x) = \normH{g(x,\cdot)}$;  by our choice of $p_0$ and the
first step of the proof we have that $g\in L^{p_0}_{\mathbb{H}}(w)$.
Moreover, since for each $t>0$,
$(tL_w)^{1/2}e^{-tL_w}$ and $e^{-kr^2L_w}$ commute,
\[ g_{L_w}(e^{-kr^2L_w}f)(x) = \normH{e^{-kr^2L_w} g(x,\cdot)}. \]
We can now use~\eqref{eqn:Asum} and~\eqref{eqn:almost-there} to get that
\begin{align*} \label{eqn:gLw-02}
\bigg(  \avgint_{B}%
|g_{L_{w}}\mathcal{A}_{r}f| ^{q_{0}}dw\bigg) ^{\frac{1}{q_{0}}}
& 
\lesssim
\sum_{k=1}^m
\bigg(\avgint_B
  \normH{e^{-kr^2L_w}g(x,\cdot)}^{q_0}\,dw\bigg)^{\frac{1}{q_0}}  \\
& \lesssim \sum_{j\geq 1}2^{j\left( \theta _{1}+\theta _{2}\right)
}e^{- \alpha 4^{j}}\left(
\avgint_{2^{j+1}B}
\left\vert g_{L_{w}}f\right\vert ^{p}dw\right) ^{1/p_0}.
\end{align*}
This proves \eqref{eqn:2.2} with $g\left( j\right) =C\,2^{j\left( \theta _{1}+\theta
_{2}\right) }e^{-c4^{j}}$.   Therefore, by Theorem~\ref{theorem:2.2}
we get that
\[ \|g_{L_w}f\|_{L^p(v\,dw)} \lesssim \|f\|_{L^p(v\,dw)}. \]

\bigskip

It remains to show the reverse inequalities. We will prove the
lower bound in~\eqref{eqn:square-function-wt}; then the lower bound
in~\eqref{eqn:square-function-unwtd} holds if we take $v\equiv 1$. Fix
$p_-<p<p_+$ and $v \in A_{p/p_-(L_w)}(w)\cap RH_{(p_+(L_w)/p)'}(w)$.  By
the duality properties of weights~\cite[Lemma 4.4]{auscher-martell07b}
and since $p_{\pm}(L_w)'=p_{\mp} (L_w^*)$, where $L_w^*$ is the adjoint (on $L^2(w)$) of $L_w$,
\begin{equation}
v^{1-p'}\in A_{p'/p_-(L^*)}(w)\cap RH_{(p_+(L^*)/p')'}(w). 
\label{eq:adede}
\end{equation}
We now proceed as in the proof of ~\cite[Theorem 7.3]{auscher-martell06}. Given $F\in L^p_{\mathbb{H}}(v\,dw)\cap L^2_{\mathbb{H}}(w)$, and $x\in \RR^n$ we set
\begin{equation}\label{defi-TLw}
T_{L_w} F(x)
=\int_0^\infty (t\,L_w)^{1/2}\,e^{-t\,L_w}F(x,t )\,\frac{dt}{t}.
\end{equation}
Recall that $(t\,L_w)^{1/2}\,e^{-t\,L_w}F(x,t )=(t\,L_w)^{1/2}\,e^{-t\,L_w}(F(\cdot,t ))(x)$. Hence, $T_{L_w}$ maps
$\HH$-valued functions to $\co$-valued functions.  For $h\in
L^{p'}(v^{1-p'}\,dw)\cap L^2(w)$ with
$\|h\|_{L^{p'}(v^{1-p'}\,dw)}=1$, we have that
\begin{align*}
\Big|\int_{\re^n} T_{L_w}F\, \overline h\, dw \Big|
&= 
\Big|\int_{\re^n} \int_0^\infty F(x,t) \, \overline{ (t\,L_w^*)^{1/2}\,e^{-t\,L_w^*}h(x )}\,\frac{dt}{t}dw(x)
\Big|
\\
&\le 
\int_{\re^n} \normH{F(x, \cdot)} \, g_{L_w^*} h(x) \, dw(x)
\\
&\lesssim
\|F\|_{L^p_{\mathbb{H}}(v\,dw)}
\|g_{L_w^*} h\|_{L^{p'}(v^{1-p'}\,dw)}
\\
&\lesssim
\|F\|_{L^p_{\mathbb{H}}(v\,dw)},
\end{align*}
where the last estimate uses the fact that $g_{L_w^*}$ is bounded on
$L^{p'}(v^{1-p'}\,dw)$.  This follows from the upper bound in
\eqref{eqn:square-function-wt} (with $L_w^*$ in place of $L_w$), which we
proved above, and
\eqref{eq:adede}. Taking the supremum over all such functions $h$ and using
an standard density argument we have obtained that $T_{L_w}$ is
bounded from $L^p_{\mathbb{H}}(v\,dw)$ to $L^p(v\,dw)$.

Next, given $f\in L^p(v\,dw)\cap L^2(dw)$, if we define $F( x,t) =(tL_{w}) ^{1/2}e^{-tL_{w}}f( x)$,  then
$F \in L^p_{\mathbb{H}}(v\,dw)\cap L^2_{\mathbb{H}}(w)$ since 
$\|F\|_{L^p_{\mathbb{H}}(v\,dw)} =
\|g_{L_w}f\|_{L^p(v\,dw)}$ and analogously for $L^2(w)$. Also, by the $L^2(w)$ functional calculus we have that
\begin{equation} \label{eqn:reverse-identity}
f(x)=2\int_{0}^{\infty }\left( tL_{w}\right) ^{1/2}e^{-tL_{w}}F\left( x,t\right) \frac{dt}{t}=2T_{L_w}F(x).
\end{equation}
Therefore, 
\[ \|f\|_{L^p(v\,dw)} = 2\|T_{L_w} F\|_{L^p(v\,dw)}
\lesssim \|F\|_{L^p_{\mathbb{H}}(v\,dw)} =
\|g_{L_w}f\|_{L^p(v\,dw)},  \]
and this completes the proof of~ \eqref{eqn:square-function-wt}.

\medskip

To finish the proof of Proposition~\ref{prop:square-function} we
need to show that the equivalence of norms in~\eqref{eqn:square-function-unwtd} implies that the
semigroup is uniformly bounded.  However, this follows immediately
from the definition of $g_{L_w}$ and the semigroup property:  for any $s>0$,
\[ g_{L_w}(e^{-sL_w}f)(x)
= \bigg(\int_0^\infty |L_w^{1/2} e^{-(s+t)L_w}f(x)|^2\,dt\bigg)^{1/2}
\leq  g_{L_w}f(x). \]
This completes the proof.
\end{proof}


\bigskip

We conclude this section by proving a version of
Proposition~\ref{prop:square-function} for the ``adjoint'' of a discrete square
function.  We will need this estimate in the proof of
Proposition~\ref{prop:reverseRiesz} below.

\begin{prop}\label{prop:dsicrtee-SF}
Define the holomorphic function $\psi$ on the sector $\Sigma_{\pi /2}$
by
\begin{equation}\label{eqpsiw}
\psi(z)=  \frac1{\sqrt{\pi}}\int_1^\infty z\, e^{-tz} \, \frac{dt}{\sqrt t}.
\end{equation}
If $p_-(L_w)<p<p_+(L_w)$, then for any sequence of functions
$\{\beta_k\}_{k\in \mathbb{Z}}$,
 \begin{equation}\label{eq23w}
\Big\| \sum_{k\in \mathbb{Z}} \psi(4^kL_w)\ \beta_k \Big\|_{L^p(w)}
\lesssim
\bigg\|\Big(\sum_{k\in \mathbb{Z}}  |\beta_k|^2\Big)^{\frac12}
\bigg\|_{L^p(w)}.
\end{equation}
\end{prop}
  
\begin{proof}
 By duality and since $p_\pm(L_w)'=p_{\mp}(L^*_w)$ ,  it will suffice to show that for every
 $p_-(L_w^*)<p<p_+(L_w^*)$,
\begin{equation}\label{discrete-SF}
\bigg\| \Big(\sum_{k\in \mathbb{Z}} |\overline{\psi}(4^kL_w^*)h|^2 \Big)^{\frac12}\bigg\|_{L^p(w)}
\lesssim
\|h\|_{L^p(w)},
\end{equation}
The function $\psi$ satisfies $|\psi(z)| \le C|z|^{1/2} e^{-c|z|}$
uniformly on subsectors $\Sigma_{\mu}$, $0\le  \mu < {\frac \pi 2}$.
Thus the operator on the lefthand side of~\eqref{discrete-SF} is a
discrete analog of the square function $g_{L_w^*}$ changing continuous times $t$ to discrete
times $4^k$ and  $z^{1/2}e^{-z}$ to $\overline{\psi}(z)$.  Since $\overline{\psi}(z)$ has the same  quantitative
properties as $z^{1/2}e^{-z}$   (decay at 0 and at
infinity), we can repeat the previous argument and obtain the desired
estimates as in 
the  proof of Proposition~\ref{prop:square-function}. 
\end{proof}

\begin{remark}
  In Proposition~\ref{prop:dsicrtee-SF} we can also get $L^p(v\,dw)$
  estimates, but in the proof of
  Proposition~\ref{prop:reverseRiesz} below we will only need the
  unweighted estimates. Further details and the precise statements are
  left to the interested reader.
\end{remark}

\section{Reverse inequalities}
\label{section:reverse}

In this section we will prove $L^p(w)$ estimates of the form
$\|L_w^{1/2}f\|_{L^p(w)} \le C\|\nabla f\|_{L^p(w)}$, which
generalize the $L^2(w)$ Kato estimates in
Theorem~\ref{theorem:degen-kato}.    These are
referred to as reverse inequalities since if we replace $f$ by
$L_w^{-1/2}f$, then formally we get a reverse-type inequality for the
Riesz transform:  $\|f\|_{L^p(w)} \le C\|\nabla L_w^{-1/2}
f\|_{L^p(w)}$.   

Since these estimates involve the gradient, in proving them we will
rely (implicitly and explicitly) on the weighted Poincar\'e
inequality~\eqref{w-Poincare}.   This will require an additional
assumption on $p$ when $p<2$.  To state it simply, define
\[ (p_-(L_w))_{w,*} = \frac{n\,r_w\,p_-(L_w)}{n\,r_w+p_-(L_w)} < p_-(L_w).  \]

\begin{prop}\label{prop:reverseRiesz}
Let
$\max\{r_w,(p_-(L_w))_{w,*}\}<p<
p_+(L_w)$.  Then for all $f\in \cals$,
\begin{equation}  \label{eq:reverseRiesz}
\|L_w^{1/2}f\|_{L^p(w)} \le C\, \|\nabla f\|_{L^p(w)}.
\end{equation}
with $C$ independent of $f$.
Furthermore, if $\max\{r_w,p_-(L_w)\}<p<p_+(L_w)$ and 
$v \in A_{{p}/{\max\{r_w,p_-(L_w)}\}}(w)\cap
RH_{({p_+(L_w)}/{p})'}(w)$, then for all $ f\in\cals$,
\begin{equation}\label{eq:reverseRiesz-vdw}
\|L_w^{1/2}f\|_{L^p(v\,dw)} \le C\, \|\nabla f\|_{L^p(v\,dw)}.
\end{equation}
\end{prop}

\begin{remark}
The quantity $\max\{r_w,(p_-(L_w))_{w,*}\}$  can be equal to either term.
For instance, it equals $r_w$ if $p_-(L_w) \le n'r_w$.  From
Proposition~\ref{prop:J} we know that $p_-(L_w) <
(2_w^*)'=\frac{2\,n\,r_w}{n\,r_w+2}$, but this only implies the previous
inequality for some values of $n$ and $r_w$.
\end{remark}

\medskip

\begin{proof}
As before, let  $p_-=p_-(L_w)$ and $p_+=p_+(L_w)$.  Fix 
$p$,   $\max\big\{r_w,(p_-)_{w,*}\big\}<p<2$, and $f\in\cals$. We will
first show that
\begin{equation}\label{eq:weak-type-reverse}
\|L_w^{1/2}f\|_{L^{p,\infty}(w)} \lesssim \|\nabla f\|_{L^p(w)}.
\end{equation}
First note that since $p>r_w$, $w\in A_p$.  Therefore, given $\alpha>0$ we can form the
Calder\'on-Zygmund decomposition given in \cite[Lemma
6.6]{auscher-martell06}:  there exist a collection of balls $\{B_i\}_i$,
smooth functions $\{b_i\}_i$ and a function $g\in L^1_{\rm loc}(w)$
such that
\begin{equation}\label{eqcsds1}
f= g+\sum_i b_i
\end{equation} 
and the following properties hold:
\begin{equation}\label{eqcsds2}
|\nabla g(x)| \le C\alpha
\quad
\text{for $w$-a.e. } x,
\end{equation}
\begin{equation}\label{eqcsds3}
\supp(b_i) \subset B_{i}
\quad\text{and}\quad
\int_{B_i} |\nabla
b_i|^p\, dw \le C\alpha^p w(B_i),\end{equation}
 \begin{equation}\label{eqcsds4}
\sum_i w(B_i) \le \frac{C}{\alpha^{p}} \int_{\re^n} |\nabla f|^p\, dw ,
\end{equation}
\begin{equation}\label{eqcsds5}
\sum_i \bigchi_{B_i} \le N,
\end{equation}
 \begin{equation}\label{eqcsds6}
\Big(\aver{B_i} |b_i|^q\,dw\Big)^\frac1q
\lesssim
C\,\alpha\, r(B_{i})\quad
\text{for }1\le q\le p^*_w,
\end{equation}
where $C$ and  $N$ depend only on $n$, $p$, $q$ and the doubling
constant of $w$.

To prove~\eqref{eq:weak-type-reverse} we
will prove the corresponding weak-type estimates with $f$ replaced by $g$ and
$b_i$.  For $g$, we use the $L^2(w)$ Kato estimate
\eqref{eqn:degen-kato1}, \eqref{eqcsds2}, and the fact that $p<2$ to
get
\begin{multline*}
w(\{|L_w^{1/2}g| > \alpha/3 \})
\lesssim
\frac{1}{\alpha^{2}}\int_{\re^n} |L_w^{1/2} g|^{2}\, dw
\lesssim
\frac{1}{\alpha^{2}}\int_{\re^n} |\nabla g|^{2}\, dw
\lesssim
\frac{1}{\alpha^p}\int_{\re^n} |\nabla g|^p\,  dw
\\
\lesssim
\frac{1}{\alpha^p}\int_{\re^n} |\nabla f|^p\,  dw +
\frac{1}{\alpha^p}\int_{\re^n}
\Big|\sum_i\nabla b_i\Big|^p\,  dw
\lesssim
\frac{1}{\alpha^p}\int_{\re^n} |\nabla f|^p\,  dw,
\end{multline*}
where the last estimate follows from  \eqref{eqcsds5},
\eqref{eqcsds3}, and \eqref{eqcsds4}.

To prove a weak-type estimate for  $L_w^{1/2}(\sum_i b_i)$, let $r_i=2^k$ if $2^k \le r(B_{i})
< 2^{k+1}$.  Then for all $i$,  $r_{i}\sim r(B_{i})$.  Write
$$
L_w^{1/2}
=
\frac 1 {\sqrt \pi}\,\int_0^{r_i^2} L_w e^{-t \, L_w}  \, \frac{dt}{\sqrt t}+ \frac 1 {\sqrt \pi}\,\int_{r_i^2}^\infty L_w e^{-t\,
L_w} \, \frac{dt}{\sqrt t}
=
T_i+U_i;
$$
then we have that
\begin{align*}
w\Big(\Big\{\Big|\sum_i L_w^{1/2} b_i\Big|>\frac{2\,\alpha}{3}\Big\}\Big)
&
\le
w\Big(\bigcup_i\, 4\,B_i\Big)+ w\Big(\Big\{\Big|\sum_i
U_ib_i\Big|>\frac{\alpha}{3}\Big\}\Big)
\\
&\hskip1.5cm
+ w\Big( \Big(\re^n\setminus\bigcup_i \, 4\,B_i \Big) \bigcap \Big\{
\Big|\sum_i T_ib_i\Big|
>\frac{\alpha}{3}\Big\}\Big)
\\
&
\lesssim
\frac1{\alpha^p}\,\int_{\re^n} |\nabla f|^p\,dw + I_1+I_2,
\end{align*}
where the last inequality follows from~\eqref{eqcsds4}. 

We first estimate $I_2$. Since $p> (p_-)_{w,*}$ then  $p_w^*>((p_-)_{w,*})_w^*=p_-$, and we can choose $q\in
\J(L_w)$ such that \eqref{eqcsds6} is satisfied. By Corollary~\ref{cor:holomorphic},
$t\,L_w\,e^{-t\, L_w}\in \offw{q}{q}$, and so
\begin{align*}
I_2
&
\lesssim
\frac{1}{\alpha}\,\sum_i\sum_{j\ge 2} \int_{C_j(B_i)} |T_ib_i| \, dw \\
& 
\lesssim
\frac{1}{\alpha}\,\sum_i\sum_{j\ge 2}
w(2^{j}\,B_i)\,\int_0^{r_i^2}\aver{C_{j}(B_{i})} |t\, L_w\, e^{-t \,L_w}
b_i |\, dw\,\frac{dt}{t^{3/2}}
\\
&\lesssim
\frac{1}{\alpha}\,\sum_i\sum_{j\ge 2} 2^{j\,D}\,w(B_i)\,\int_0^{r_i^2}
2^{j\,\theta_{1}}\, \dec{\frac{2^j\, r_i}{\sqrt t}}^{\theta_{2}} \,
\expt{-\frac{c\, 4^j\, r_{i}^2}{t}}\,\frac{dt}{t^{3/2}}
\, \bigg(\aver{B_{i}} |b_{i}|^q\, dw\bigg)^{\frac1q}
\\
&\lesssim
\sum_i\sum_{j\ge 2} 2^{j\,D}\,e^{-c\,4^j}\,w(B_i) \\
& \lesssim
\sum_i w(B_i) \\
& \lesssim \frac{1}{\alpha^p}\,\int_{\re^n}
|\nabla f|^p \, dw,
\end{align*}
where we have used  \eqref{eqcsds6} and \eqref{eqcsds4}, and $D$ is the doubling order of $dw$.

We will now estimate $I_1$. For $q$ as above, by
Proposition~\ref{prop:B-K:weights} we have an $L^q(w)$
functional calculus for $L_w$.  Therefore,  we can write $U_i$ as
$r_i^{-1}\psi(r_i^2L_w)$ with $\psi$ defined by \eqref{eqpsiw}. Let
$\beta_{k}= \sum_{i\, : \, r_i=2^k}\frac{b_i}{r_i}$;  then,
$$
\sum_i U_i \, b_i
=
\sum_{k\in \mathbb{Z}} \psi(4^k\, L_w)
\bigg(\sum_{i\, : \, r_i=2^k}\frac{b_i}{r_i}\bigg)
=
\sum_{k\in \mathbb{Z}} \psi(4^k\,L_w) \beta_k.
$$
Therefore, by Proposition \ref{prop:dsicrtee-SF}, \eqref{eqcsds5},
 \eqref{eqcsds6},  the fact that $r_{i}\sim r(B_{i})$ and \eqref{eqcsds4},
we have that
\begin{multline*}
I_1
\lesssim
\frac1{\alpha^q}\,
\Big\|\sum_i U_i b_i\Big\|_{L^q(w)}^q
\lesssim
\frac1{\alpha^q}\,
\bigg\|\Big(\sum_{k\in \mathbb{Z}}  |\beta_k|^2\Big)^{\frac12}
\bigg\|_{L^q(w)}^q \\
\lesssim
\frac1{\alpha^q}\, \int_{\re^n} \sum_{i}  \frac{|b_i|^q}{r_i^q}\ dw
\lesssim
\sum_{i}w(B_i)
\lesssim
\frac{1}{\alpha^p}\int_{\RR^n} |\nabla f|^p\,  dw.
\end{multline*}
If we combine all of the estimates we have obtained, we 
get \eqref{eq:weak-type-reverse} as desired.

\medskip

To prove~\eqref{eq:reverseRiesz} from the weak-type estimate
\eqref{eq:weak-type-reverse} we
will use an interpolation argument from \cite{auscher-martell06}. 
Fix $p$ and $r$ such that $\max\big\{r_w,(p_-)_{w,*}\big\}<r<p<2$. 
Then by \eqref{eq:weak-type-reverse} and
\eqref{eqn:degen-kato1} we have that for every $f\in\cals$,
\begin{equation}\label{eq:interpol-grad}
\|L_w^{1/2}f\|_{L^{r,\infty}(w)} \lesssim \|\nabla f\|_{L^r(w)},
\qquad
\|L_w^{1/2}f\|_{L^2(w)} \lesssim \|\nabla f\|_{L^2(w)}.
\end{equation}
Formally, to apply Marcinkiewicz interpolation, we let $g=\nabla f$ to
get a weak $(r,r)$ and strong $(2,2)$ inequality; this would
immediately yield a strong $(p,p)$ inequality.  To formalize this we
must justify this substitution.  

For every $q>r_w$ by \cite[Lemma 6.7]{auscher-martell06} we have that
$$
\mathcal{E}
=
\big\{(-\Delta)^{1/2}f\, : \, f\in \cals, \supp \widehat f \subset \RR^n\setminus \{0\}\big\}
$$
is dense in $L^q(w)$, where $\widehat f$ denotes the Fourier transform
of $f$.  Moreover, since $r>r_{w}$, $w\in A_r$ and the Riesz
transforms, $R_j = \partial_j (-\Delta)^{-1/2}$, are bounded on
$L^r(w)$~\cite{garciacuerva-rubiodefrancia85}.    It follows from this
and the identity $-I=R_1^2+\dots +R_n^2$ that for
$g\in L^r(w)$, 
$$
\|g\|_{L^r(w)} \sim \|\nabla (-\Delta)^{-1/2} g\|_{L^r(w)}.
$$
Thus, for $g\in \mathcal{E}$, $L_w^{1/2}(-\Delta)^{-1/2}g=L_w^{1/2}f$
if $f=(-\Delta)^{-1/2}g$ and
$\|\nabla f\|_{L^r(w)} \sim \|g \|_{L^r(w)}$ for $r>r_{w}$.   Thus 
\eqref{eq:interpol-grad} becomes weighted weak $(r,r)$
and strong $(2,2)$ inequalities for $T= L_w^{1/2}(-\Delta)^{-1/2}$,
and this operator is defined \textit{a
  priori} on $\mathcal{E}$. Since $\mathcal{E}$ is dense in
each $L^q(w)$, we can extend
$T$ by density in both cases and their restrictions to the space of
simple functions agree. Hence, we can apply Marcinkiewicz
interpolation and conclude, again by density, that
\eqref{eq:reverseRiesz} holds for all $p$ with $r<p<2$. Since $r$
is arbitrary, we get \eqref{eq:reverseRiesz} in the range $\max\big\{r_w,(p_-)_{w,*}\big\}<p<2$.

\bigskip

For the second step of the proof we will
prove~\eqref{eq:reverseRiesz-vdw} using Theorem~\ref{theorem:2.2}.
Inequality~\eqref{eq:reverseRiesz} for its full range of exponents then follows by letting $v=1$.
Define $\tilde{p}_-=\max\{r_w,p_-\}<2$, and fix $\tilde{p}_-<p<p_+$ and
$v\in A_{p/\tilde{p}_-}(w)\cap RH_{(p_+/p)'}(w)$. By the openness
properties of $A_q$ and $RH_s$ weights, there exist $p_0$, $q_0$ such
that
\[
\tilde{p}_-<p_0<\min\{p,2\}\le p<q_0<p_+,
\qquad
v \in A_{p/p_0}(w)\cap
RH_{(q_0/p)'}(w).  \]

To apply Theorem \ref{theorem:2.2}, let $T=L_w^{1/2}$,
$S=\nabla$, and $\mathcal{A}_{r}=I-( I-e^{-r^{2}L_{w}}) ^{m}$ where the
value of $m$ will be fixed below. We will first show that \eqref{eqn:2.2}
holds. By \eqref{eqn:A-od} we have that $\mathcal{A}_{r}\in \offw{p_0}{q_0}$
since $p_0$, $q_0\in \J(L_w)$.   Let $h=L_w^{1/2} f$ and decompose
$h$ as we decomposed $f$ in \eqref{decomp-f}. Then, since $L_w^{1/2}$
and $\mathcal{A}_{r}$ commute, it follows that
\begin{align*}
\left(
\avgint_{B}\left\vert {L_{w}^{1/2}}\mathcal{A}_{r} f\right\vert ^{q_{0}}dw\right) ^{\frac{1}{q_0}}
& \lesssim
\sum_{j\ge 1}
\left(
\avgint_{B} \left\vert \mathcal{A}_{r} h_j \right\vert ^{q_{0}}dw\right) ^{\frac{1}{q_0}}
\\
& \lesssim
\sum_{j\ge 1}
2^{j\theta_{1}}\Upsilon \left( 2^{j}\right) ^{\theta _{2}} e^{-c4^{j}}
\bigg(
\avgint_{C_{j}}\left\vert h\right\vert ^{p_{0}}dw
\bigg) ^{\frac{1}{p_0}} \\
& \le
\sum_{j\ge 1}
2^{j(\theta_{1}+\theta_2)} e^{-c4^{j}}
\bigg(
\avgint_{2^{j+1}\,B}\left\vert L_w^{1/2}f\right\vert ^{p_{0}}dw
\bigg) ^{\frac{1}{p_0}}.
\end{align*}
This gives us \eqref{eqn:2.2} with $g(j)=C\,2^{j\left(
\theta _{1}+\theta _{2}\right) }e^{-c4^{j}}$; clearly, $\sum
g(j)<\infty$.

We now prove that \eqref{eqn:2.1} holds.  Fix $f\in\cals$ and let
$\varphi(z)=z^{1/2} (1-e^{-r^2\, z})^m$ so that
$\varphi(L_w)f=L_w^{1/2}(I-e^{-r^2\,L_w})^mf$.  By the conservation
property (see \cite{DCU-CR2013} or \cite[Section
2.5]{auscher07}),
\begin{equation}\label{eq:rep}
\varphi(L_w)\, f=
\varphi(L_w)\, (f-f_{4\,B,w})
=
\sum_{j\ge 1} \varphi(L_w)\, h_j,
\end{equation}
where $h_{j}=(f-f_{4\,B,w})\,\phi_j$,
$\phi_j=\bigchi_{C_j(B)}$ for $j\ge 3$, $\phi_1$ is a  smooth
function with support in $4\,B$, $0\le \phi_1\le 1$, $\phi_{1}=1$ in $2\,B$
and $\|\nabla \phi_1\|_{\infty}\le C/r$, and  $\phi_2$ is
chosen so that $\sum_{j\ge 1} \phi_j=1$. 

We estimate each term in the righthand side of~\eqref{eq:rep} separately.
When $j=1$, since $ p_-<p_0< p_+$, by the bounded holomorphic
functional calculus on $L^{p_{0}}(w)$ (Proposition \ref{prop:B-K:weights}) and
the fact that $\varphi(L_w)\, h_{1}= (I-e^{-r^2\, L_w})^m\,
L_w^{1/2}h_{1}$,   we have that 
$$
\left\|\varphi(L_w)\, h_{1}\right\|_{L^{p_{0}}(w)}
 \lesssim
\|L_w^{1/2}h_{1}\|_{L^{p_{0}}(w)}
$$
uniformly in $r$.
By the above argument we have that~\eqref{eq:reverseRiesz} holds for $p=p_0$ since
$\tilde{p}_-<p_0<2$.
Further, since  $f\in\cals$, $h_1\in\cals$ by our choice of
$\phi_1$. This, together with the $L^{p_{0}}(w)$-Poincar\'e inequality
\eqref{w-Poincare} (since $p_0>r_w$, $w\in A_{p_0}$)
and the definition of $h_{1}$ yield
\begin{multline*}
\|L_w^{1/2}h_{1}\|_{L^{p_{0}}(w)}
\lesssim
\|\nabla h_{1}\|_{L^{p_{0}}(w)}
\\
\lesssim
\|(\nabla f) \bigchi_{4B}\|_{L^{p_{0}}(w)}
+
r^{-1}\,\|(f-f_{4\,B,w})\bigchi_{4B}\|_{L^{p_{0}}(w)}
\lesssim
\|(\nabla f) \bigchi_{4B}\|_{L^{p_{0}}(w)}.
\end{multline*}
Therefore,
$$
\bigg( \aver{B} |\varphi(L_w)\, h_{1}|^{p_0}\,dw\bigg)^{\frac1{p_0}}
\lesssim
\bigg( \aver{4\,B} |\nabla f |^{p_0}\,dw\bigg)^{\frac1{p_0}}.
$$

When $j\ge 3$, the functions  $\eta$
associated with $\varphi$  by \eqref{eqn:L2-holo-rep-eta}  satisfy
$$
|\eta(z)|
\lesssim \frac{r^{2\,m}}{|z|^{m+3/2}},
\qquad
z \in \Gamma_{\pi/2-\theta}.
$$
Since $p_0\in \J(L_w)$, by Corollary~\ref{cor:holomorphic},
$e^{-z\,L_w}\in \mathcal{O}\big(L^{p_0}(w) \rightarrow L^{p_0}(w),
\Sigma_{\mu}\big)$. This, together with the
representation~\eqref{eqn:L2-holo-rep}, gives us that
\begin{align*}
&\bigg( \aver{B} |\varphi(L_w)h_j|^{p_0}dw\bigg )^{\frac1{p_0}} \\
&\qquad\quad \le
\int_{\Gamma_{\pi/2-\theta}} \bigg(\aver{B} |e^{-z\,L}
h_j|^{p_0}\,dw\bigg)^{\frac1{p_0}}\, |\eta(z)|\,|dz|
\\
&\qquad\quad \lesssim
2^{j\,\theta_1} \int_{\Gamma_{\pi/2-\theta}}
\dec{\frac{2^j\,r}{\sqrt{|z|}}}^{\theta_2}\,
\expt{-\frac{\alpha\,4^j\,r^2}{|z|}}\,
\frac{r^{2\,m}}{|z|^{m+3/2}}\,  {|dz|} \,
\bigg(\aver{C_j(B)}
|h_{j}|^{p_0}\,dw\bigg)^{\frac1{p_0}}
\\
&
\qquad\quad\lesssim
2^{j\,(\theta_1- 2\,m-1)} \, r^{-1}\,
\bigg(\aver{2^{j+1}\,B}
|f-f_{4\,B,w}|^{p_0}\,dw\bigg)^{\frac1{p_0}}
\\
&
\qquad\quad\lesssim
2^{j\,(\theta_1- 2\,m-1)} \, \sum_{l=1}^{j} 2^{l}\,
\bigg(\aver{2^{l+1}\,B}
|\nabla f|^{p_0}\,dx\bigg)^{\frac1{p_0}},
\end{align*}
provided $2\, m+1>\theta_2$.  
The last estimate follows from  $L^{p_{0}}(w)$-Poincar\'e inequality
\eqref{w-Poincare} (here we again use that $p_0>r_w$ and so $w\in
A_{p_0}$):
\begin{align}\label{eq:poincare}
&\bigg(\aver{2^{j+1}\,B}
|f-f_{4\,B,w}|^{p_0}\,dw\bigg)^{\frac1{p_0}}
\nonumber
\\
&\qquad\quad
\le
\bigg(\aver{2^{j+1}\,B}\hskip-7pt
|f-f_{2^{j+1}\,B,w}|^{p_0}\,dw\bigg)^{\frac1{p_0}} +\sum_{l=2}^j
|f_{2^l\,B,w}-f_{2^{l+1}\,B,w}|
\nonumber
\\
&\qquad\quad\lesssim
\sum_{l=1}^{j}
\bigg(\aver{2^{l+1}\,B}
|f-f_{2^{l+1}\,B}|^{p_0}\,dx\bigg)^{\frac1{p_0}}
\nonumber
\\
&\qquad\quad\lesssim
r\,
\sum_{l=1}^{j} 2^{l}\,
\bigg(\aver{2^{l+1}\,B}
|\nabla f|^{p_0}\,dx\bigg)^{\frac1{p_0}}.
\end{align}

When $j=2$ we can argue similarly,  using the fact that
$$|h_2|\le |f-f_{4\,B,w}|\,\bigchi_{8\,B\setminus 2\,B} \le  |f-f_{2\,B,w}|\,\bigchi_{8\,B\setminus 2\,B}
+  |f_{4\, B,w}-f_{2\,B,w}|\,\bigchi_{8\,B\setminus 2\,B}.
$$

If we combine these estimates, then by~\eqref{eq:rep} and Minkowski's
inequality we get
$$
\bigg( \aver{B} |\varphi(L_w)h|^{p_0}dw\bigg )^{\frac1{p_0}}
\lesssim
\sum_{j\ge 1}
\bigg( \aver{B} |\varphi(L_w)h_j|^{p_0}dw\bigg )^{\frac1{p_0}}
\le
\sum_{j\ge 1}
g(j)\bigg( \aver{B} |\nabla f|^{p_0}dw\bigg )^{\frac1{p_0}}
$$
with $g(j)=C_m\,2^{j\,(\theta_1-2\,m)}$ provided $2\,m+1>\theta_2$.
If  we further assume that $2\,m>\theta_1$,  then $\sum_j g(j)<\infty$.
This proves that \eqref{eqn:2.1} holds.  Therefore, by  Theorem
\ref{theorem:2.2} we get~\eqref{eq:reverseRiesz-vdw} as desired.
\end{proof}


\section{The gradient of the semigroup $\sqrt{t}\grad e^{-tL_w}$}
\label{section:gradient}

Let $\widetilde \K(L_w)\subset [1,\infty]$ be the set of all exponents
$p$ such that $\sqrt{t} \grad e^{-t\,L_w} : L^p(w)\rightarrow L^p(w)$
is uniformly bounded for all $t>0$. By Theorem~\ref{thm:L2-gaffney}
and Lemma~\ref{lemma:full}, $2\in \widetilde \K(L_w)$ and if it
contains more than one point, then  by
interpolation $\widetilde \K(L_w)$ is an interval.  In this section
we give a partial description of the set of $(p,q)$ such that
$\sqrt{t} \grad e^{-t\,L_w} \in \offw{p}{q}$.

\begin{prop}\label{prop:K}
There exists an interval  $\K(L_w)$ such that if $p,\,q \in \K(L_w)$,
$p\le q$,  then $\sqrt{t}\,\nabla e^{-t\,L_w}\in \offw{p}{q}$.
Moreover, $\K(L_w)$ has the following properties:
\begin{enumerate}
\item $\K(L_w)\subset \widetilde\K(L_w)$;

\item if $q_-(L_w)$ and $q_+(L_w)$ are the left and right endpoints
  of $\K(L_w)$, then $q_-(L_w)=p_-(L_w)$, $2\le q_+(L_w)\le
  (q_+(L_w))^*_w\le p_+(L_w)$.  In particular, $2\in \K(L_w)$ and
  $\K(L_w)\subset \J(L_w)$;

\item If $q\geq 2$ and $p<q$, and if $\sqrt{t}\,\nabla e^{-t\,L_w} \in
  \offw{p}{q}$, then $p,\,q \in \K(L_w)$;

\item $\sup \widetilde \K(L_w) = q_+(L_w)$.
\end{enumerate}
\end{prop}

\begin{remark}
  Unlike in the unweighted case (see~\cite{auscher-martell07}) we are
  unable to give a complete characterization of $\K(L_w)$.  More
  precisely, if we have an off-diagonal estimate and $p<q<2$, then we
  cannot prove that $p,\,q\in \K(L_w)$.
\end{remark}

\begin{remark}
In Section~\ref{section:q-plus} below we will show that $q_+(L_w)>2$; in particular, this gives that
$2\in \Int\K(L_w)$.
\end{remark}

\medskip

As an immediate consequence of Proposition~\ref{prop:K} we get
weighted inequalities for the gradient of the semigroup.  The proof is
identical to the proof of Corollaries~\ref{corollary-weighted-offd} and \ref{cor:holomorphic}.

\begin{corol}\label{corollary-grad-weighted-offd} 
  Let $q_-(L_w)<p\le q<q_+(L_w)$.  If
  $v\in A_{p/q_-(L_w)}(w)\cap RH_{(q_+(L_w)/q)'}(w)$, then 
  $\sqrt{t}\,\nabla e^{-tL_w}\in \mathcal{O}\big(L^{p}(v\,dw)\rightarrow L^{q}(v\,dw)\big)$ and 
	$\sqrt{z}\,\nabla e^{-zL_w}\in \mathcal{O}\big(L^{p}(v\,dw)\rightarrow L^{q}(v\,dw),\Sigma_\nu\big)$
	for all $\nu$,   $0<\nu<\frac{\pi}{2}-\vartheta$. 
\end{corol}

\medskip

The proof of Proposition~\ref{prop:K} requires two lemmas.

\begin{lemma}\label{off-w:sel-impro}
Given $w\in A_\infty$ and a family of sublinear operators  $\{T_t\}_{t>0}$
such that $T_t\in \offw{p}{q}$,   with $1\le p<q\le \infty$,  there exist $\alpha$, $\beta>0$ such that for any ball $B$ with radius $r$ and for any $t>0$,
\begin{equation}\label{w:off:B-B:improved}
\left(\aver{B} |T_t( \bigchi_B \, f) |^{q}\,dw\right)^{\frac 1 q}
\lesssim
\max\bigg\{\left(\frac{r}{\sqrt{t}}\right)^{\alpha},\left(\frac{r}{\sqrt{t}}\right)^{\beta}  \bigg\} \,\left(\aver{B}
|f|^{p}\,dw\right)^{\frac 1 p }.
\end{equation}
\end{lemma}

\begin{proof}
  This result is implicit in \cite[Proof of Proposition 2.4,
  p. 306]{auscher-martell07}; here we reprove it with a small
  improvement in the constant. There it was shown that in
  Definition~\ref{defi:off-d:weights} it is sufficient to consider
  the case where $r\approx \sqrt{t}$.  But in this case we get that
  $\Upsilon(r/\sqrt{t})\approx 1$ and for all $j\geq 2$,
  $\Upsilon(2^j\,r/\sqrt{t})\approx 2^{j}$. The argument in
  \cite[p. 306]{auscher-martell07} shows that if we assume that
  \eqref{w:off:B-B}, \eqref{w:off:C-B}, \eqref{w:off:B-C} hold when
  $r\approx \sqrt{t}$, then \eqref{w:off:B-B} holds in general with
  constant $\max\{1,(r/\sqrt{t})^\alpha\}$ for some $\alpha>0$
  depending on $p$, $q$ and $w$. In this maximum the $1$ occurs when
  $r\le \sqrt{t}$;  therefore, to prove~\eqref{w:off:B-B:improved} we
  need to show that if $r\le \sqrt{t}$, then we can replace $1$ by the
  better constant $(r/\sqrt{t})^\beta$ for some $\beta>0$.

Fix $r\le \sqrt{t}$.  If $B=B(x,r)$, then $B\subset
  B_t=B(x,\sqrt{t})$.   As in \cite[p. 306]{auscher-martell07} we apply
  \eqref{w:off:B-B}  to $T_t$ and $B_t$; this yields
\begin{multline*}
\left( \aver B |T_t(\bigchi_B \,f)|^q\,dw\right)^{{\frac 1 q}}
\le
\left( \frac{w(B_t)}{w(B)}\right)^{\frac 1 q} \left( \aver {B_t }|T_t(\bigchi_B \,f)|^q\,dw\right)^{{\frac 1 q}}
\\
\lesssim \left( \frac{w(B_t)}{w(B)}\right)^{\frac 1 q} \left( \aver
  {B_t }|\bigchi_B f|^p\,dw\right)^{\frac 1 p}
\le
\left( \frac{w(B)}{w(B_t)}\right)^{\frac 1 p - \frac 1 q} \left( \aver {B }|f|^p\,dw\right)^{\frac 1 p}.
\end{multline*}
Since $w\in A_\infty$, we
have that for some $\theta>0$,
\[ \frac{w(B)}{w(B_t)} \lesssim
\left(\frac{|B|}{|B_t|}\right)^{\theta}=\left(\frac{r}{\sqrt{t}}\right)^{\theta\,n}. \]
Since $p<q$ we have that
$$
\left( \aver B |T_t(\bigchi_B \,f)|^q\,dw\right)^{{\frac 1 q}}
\lesssim
\left( \frac{r}{\sqrt{t}}\right)^{(\frac 1 p - \frac 1 q)\,\theta\,n} \left( \aver {B }|f|^p\,dw\right)^{\frac 1 p}.
$$
Therefore, if we combine this with the argument
from~\cite[p. 306]{auscher-martell07} described above, we get that
~\eqref{w:off:B-B:improved} holds with
$\beta=(1/p-1/q)\,\theta\,n$.
\end{proof}

The second lemma gives the close connection between off-diagonal
estimates for $e^{-tL_w}$ and $\sqrt{t}\nabla e^{-tL_w}$ for $p<2$.

\begin{lemma}\label{lemma:fode}
Given $1\le p<2$ the following are equivalent:
\begin{enumerate}
\item $e^{-t\,L_w} \in \offw{p}{2}$.

\item $\sqrt {t} \, \nabla e^{-t\,L_w}\in  \offw{p}{2}$.

\item $t\,  L_w\, e^{-t\,L_w}\in \offw{p}{2}$.

\end{enumerate}

\end{lemma}

\begin{proof}
We follow the proof of \cite[Lemma 5.3]{auscher-martell07}.
To prove that $(1)$ implies $(2)$, note that by
Theorem~\ref{thm:L2-gaffney}, $\sqrt{t} \, \nabla
e^{-t\,L_w}\in \offw{2}{2}$. If we compose this with $(1)$, by
Lemma~\ref{lemma:unif-comp}  and the
semigroup property, we get $(2)$.

To prove that $(2)$ implies $(3)$, define
$S_t\vec{f}=\sqrt{t}\,e^{-t\,L_w}(w^{-1}\,\div (A\vec f)) $.   By
duality, we have that
\begin{multline*}
\langle S_t\vec{f},g\rangle_{L^2(w)}
= \langle w^{-1} \div  (A\vec f)),
\sqrt{t}e^{-t\,L_w^*}g\rangle_{L^2(w)}
=\langle \div  (A\vec f)),
\sqrt{t}e^{-t\,L_w^*}g\rangle_{L^2} \\
= -\langle \vec{f}, A^* \sqrt{t}\nabla e^{-t\,L_w^*}g\rangle_{L^2}
= \langle \vec{f}, w^{-1} A^* \sqrt{t}\nabla
e^{-t\,L_w^*}g\rangle_{L^2(w)} .
\end{multline*}
The matrix $w^{-1} A^*$ is uniformly elliptic, and so multiplication by
it is bounded on $L^2(w)$.  Furthermore,
$\sqrt{t}\,\nabla e^{-t\,L_w^*}\in\offw{2}{2}$.  Therefore, it follows
that $S_{t}\in\offw{2}{2}$. If we combine this with $(2)$, we get that $-t\,L_w\, e^{-2\,t\, L_w} =S_t
\circ \sqrt t \, \nabla e^{-t\,L_w}\in \offw{p}{2}$. This proves $(3)$.

Finally we show that $(3)$ implies $(1)$. We first prove
\eqref{w:off:B-B}. Fix $B$ and $f,\,g$ such that $\left(\aver{B} |f|^p\,dw\right)^{\frac1p}=
\left(\aver{B} |g|^2\,dw\right)^{\frac12}=1$, and assume also that $f\in L^2(B,dw)$. Define
\[ h(t)=
\aver{B} e^{-t\,L_w}( \bigchi_B \, f)(x)\,g(x)\,dw(x). \]
By duality it will suffice to show that $|h(t)|\lesssim
\Upsilon(r/\sqrt{t})^\theta$.  (Note that our assumption implies that $t\,h'(t)$ satisfies such a bound.)  First,  we claim that
\[ \lim_{t\to\infty}
h(t)=0. \]
To see this we use the fact (discussed in
Section~\ref{section:prelim}) that $L_w$  has a bounded
holomorphic functional calculus  on $L^2(w)$.  Given this, since $z\mapsto
e^{-tz}$ converges to 0 uniformly on compact subsets of $\Re z>0$, we
get the desired limit. 

Hence, we can write $ h(t)= - \int_{t}^\infty h'(s) \, ds. $ Notice
that $|t\,h'(t)|\lesssim \Upsilon(r/\sqrt{t})^{\theta_2}$ but this
does not give a convergent integral. However, if we apply
Lemma~\ref{off-w:sel-impro} to $t\,  L_w\, e^{-t\,L_w}\in \offw{p}{2}$,
we get that $|t\,h'(t)|\lesssim
\widetilde{\Upsilon}(r/\sqrt{t})$ with
$\widetilde{\Upsilon}(s)=\max\{s^\alpha,s^\beta\}$.   It follows from
this estimate that
\begin{multline*}
|h(t)|
\le
\int_{t}^\infty |h'(s)| \, ds
\lesssim
\int_{t}^\infty
\widetilde{\Upsilon}\left(\frac{r}{\sqrt{s}}\right)
\frac{ds}{s}
\approx
\int_{0}^{\frac{r}{\sqrt{t}}}
\widetilde{\Upsilon}(s)
\frac{ds}{s}
\lesssim
\widetilde{\Upsilon}\left(\frac{r}{\sqrt{t}}\right)
\lesssim
\dec{\frac{r}{\sqrt{t}}}^{\alpha+\beta}.
\end{multline*}

To prove \eqref{w:off:C-B} we argue as before, but with $\big(\aver{C_j(B)} |f|^p\,dw\big)^{\frac1p}=
\left(\aver{B} |g|^2\,dw\right)^{\frac12}=1$ and
\[ h(t)=
\aver{B} e^{-t\,L_w}( \bigchi_{C_j(B)} \, f)(x)\,g(x)\,dw(x). \]
Since $d(B,C_j(B))>0$, by Theorem~\ref{thm:L2-gaffney} and H\"older's
inequality, $ h(t) \rightarrow 0$ as $t\rightarrow 0$. Therefore, $h(t)=\int_{0}^t
 h'(s)\, ds$. Since $t\,L_w\,e^{-t\,L_w}\in\offw{p}{2}$, we have that
\begin{multline*}
h(t)
\le
\int_{0}^t |h'(s)|\, ds
\lesssim
2^{j\,\theta_1}
\int_{0}^t
\dec{\frac{2^j\,r}{\sqrt{s}}}^{\theta_2}\,
\expt{-\frac{c\,4^{j}\,r^2}{s}} \,\frac{ds}{s}
\\
\approx
2^{j\,\theta_1}\int_{\frac{2^j\,r}{\sqrt{t}}}^\infty
\Upsilon(s)^{\theta_2}\, e^{-c\,s^2}\,\frac{ds}{s}
\lesssim
2^{j\,\theta_1}
\dec{\frac{2^j\,r}{\sqrt{t}}}^{\theta_2}\,
\expt{-\frac{c\,4^{j}\,r^2}{t}}.
\end{multline*}
This is \eqref{w:off:C-B}.

Finally, the proof of \eqref{w:off:B-C} is essentially the same and we
omit the details. This completes the proof that $(3)$ implies $(1)$.
\end{proof}

\begin{proof}[Proof of Proposition \textup{\ref{prop:K}}]
Define the sets $\K_{-}(L_w)$  and $\K_{+}(L_w)$ to be
\begin{gather*}
\K_{-}(L_w) = \{ p\in [1,2] : \sqrt t \, \nabla
e^{-t\,L_w}\in\offw{p}{2} \} \\
\K_{+}(L_w) = \{ p\in [2,\infty] : \sqrt t \, \nabla
e^{-t\,L_w}\in\offw{2}{p} \},
\end{gather*}
and let $\K(L_w)= \K_{-}(L_w)\cup \K_{+}(L_w)$.
The set is non-empty, since $2\in \K(L_w)$.  By
Lemma~\ref{lemma:nested} it is an interval.
Now fix $p,q \in \K(L_w)$ with
$p<q$. If $p<q\le  2$ or $2 \le p < q $, then by Lemma~
\ref{lemma:nested}, $\sqrt t \, \nabla
e^{-t\,L}\in \offw{p}{q}$ since $p, q \in
\K_{-}(L_w)$ or $p,q \in \K_{+}(L_w)$. If $p \le 2 < q$, then
$\sqrt t \, \nabla e^{-t\,L}\in \offw{2}{q}$ and by Lemma~\ref{lemma:fode}, $e^{-t\,L}\in
\offw{p}{2}$. Hence, by Lemma~\ref{lemma:unif-comp} and
the semigroup property, $\sqrt t \, \nabla e^{-t\,L}\in
\offw{p}{q}$.  Thus, in every case we get the desired off-diagonal estimate.

We now prove (1)-(4).  By Lemma~\ref{lemma:unif-comp}, off-diagonal
  estimates on balls imply uniform boundedness, and so $\K(L_w)\subset
  \widetilde \K(L_w)$.  This proves (1).

  To prove (2), we first note that if $p<2$, then by Lemma
  \ref{lemma:fode}, $p\in \J(L_w)$ if and only if $p\in
  \K_-(L_w)$. Thus $\J(L_w)\cap[1,2]=\K_-(L_w)$ and so
  $q_-(L_w)=p_-(L_w)$.  To show that $(q_+(L_w))^*_w\le p_+(L_w)$,
  first note that  if
  $q_+(L_w)=2$, then by Proposition~\ref{prop:J} we have that
$(q_+(L_w))^*_w=2_w^* \leq p_+(L_w)$.   If $q_+(L_w)>2$, then we
proceed as in the proof of this proposition.
 Let
  $2<p<q_+(L_w)$ and $p<q<p^*_w$. Then by  \eqref{w-Poincare},
  $e^{-t\,L_w}\in\offw{2}{2}$, and $\sqrt{t}\,\nabla
  e^{-t\,L_w}\in\offw{2}{p}$, we get that
\begin{align*}
&\left(\aver{B} |e^{-t\,L_w}( \bigchi_B \, f)|^{q}\,dw\right)^{\frac1q}
\\
&\qquad
\lesssim
\left(\aver{B} |e^{-t\,L_w}( \bigchi_B \, f) |^{2}\,dw\right)^{\frac12}
+
r\,\left(\aver{B} |\nabla\,e^{-t\,L_w}( \bigchi_B \, f) |^{p}\,dw\right)^{\frac1p}
\\
&\qquad
\lesssim
\dec{\frac{r}{\sqrt{t}}}^{1+\theta_2} \,\left(\aver{B} |f|^{2}\,dw\right)^{\frac12}.
\end{align*}
This gives us inequality~\eqref{w:off:B-B}.  The other two
inequalities in Definition \ref{defi:off-d:weights} can be proved in
exactly the same way. Thus
$e^{-t\,L_w}\in\offw{2}{q}$ which implies $q\le
p_+(L_w)$. Letting $p\nearrow q_+(L_w)$ and $q\nearrow p^*_w$ we
conclude that $(q_+(L_w))^*_w\le p_+(L_w)$.

The last estimate implies in particular that $q_+(L_w)\le
p_+(L_w)$. If $q_+(L_w)<\infty$ we clearly have that $q_+(L_w)<
p_+(L_w)$ and so $\K_+(L_w)\subset \J(L_w)$. Otherwise,
$p_+(L)=\infty$ and again we have that $\K_+(L_w)\subset \J(L_w)$.
This completes the proof of (2).

\medskip

To prove (3),   suppose
first that $2\leq p < q$ and $\sqrt t
\, \nabla e^{-t\,L}\in \offw{p}{q}$.  We will show that $p,q \in
\K(L_w)$. Since we also have that $\sqrt t
\, \nabla e^{-t\,L}\in \offw{2}{2}$, by interpolation (Lemma \ref{lemma:interpolate}),
$\sqrt t \, \nabla e^{-t\,L}\in \offw{p_{\theta}}{q_{\theta}}$ where
$1/p _{\theta}= (1-\theta)/ p + \theta/2$, $1/q
_{\theta}=(1-\theta)/q+\theta/2$ and $\theta \in (0,1)$.  If $p
\notin \K_{+}(L_w)$, then $q>\sup \K_{+}(L_w)$.  We can choose $\theta$
such that $p_{\theta}<\sup\K_{+}(L_w)<q_{\theta}$. Since
$\K_{+}(L_w)\subset \J(L_w)$,   $p_{\theta}\in \J(L_w)$: i.e.,
$e^{-t\,L}\in \offw{2}{p_{\theta}}$. By composition and the
semigroup property,  $\sqrt t \, \nabla e^{-t\,L_w}\in
\offw{2}{q_{\theta}}$;  hence, $q_{\theta} \in \K_{+}(L_w)$, a
contradiction.  Therefore, $p\in \K_{+}(L_w)$. As we
have $\sqrt t \, \nabla e^{-t\,L_w}\in \offw{p}{q}$ by assumption  and
$e^{-t\,L_w}\in \offw{2}{p}$ since $p\in \J(L_w)$, by composition and
the semigroup property, $\sqrt t \, \nabla e^{-t\,L_w}\in \offw{2}{q}$. Hence,
$q\in \K_{+}(L_w)$.

The case $p<2 \le q$ is straightforward.   Since $\sqrt t \, \nabla e^{-t\,L_w}\in
\offw{p}{q}$,  by Lemma~\ref{lemma:nested} we have that $\sqrt t \, \nabla e^{-t\,L_w}\in
\offw{2}{q}$  and $\sqrt t \, \nabla e^{-t\,L_w}\in \offw{p}{2}$.
Hence, $p\in \K_{-}(L_w)$ and $q\in \K_{+}(L_w)$.

\medskip

Finally, we prove (4).   Suppose to the contrary that $\sup \widetilde
\K(L_w)>q_+(L_w)$. Then there exist $p$, $q$ such that
$q_+(L_w)<p<q<\sup \widetilde \K(L_w)$.  Fix $r$ such that $p_-(L_w)=q_-(L_w)<r<2$.
Then we have that $\sqrt{t}\nabla e^{-t\,L_w}$ is uniformly bounded on
$L^q(w)$ and in $\offw{r}{2}$.    By Lemma~\ref{lemma:interpolate} we
can interpolate between these to get that $\sqrt{t}\,\nabla
e^{-t\,L_w}\in \offw{s}{p}$ for some $s<p$.   But then by the above
converse, we have
that $p\in \K(L_w)$ which is a contradiction.
\end{proof}

\section{An upper bound for $\K(L_w)$}
\label{section:q-plus}

In this section we will prove that $q_+(L_w)>2$:  that is, the set
$\K(L_w)$ contains $2$ in its interior.
In general, all we can say is that $q_+(L_w)>2$: as noted
in~\cite[Section~4.5]{auscher07}, even in the unweighted case this is the best
possible bound, since given any $\epsilon>0$ it is possible to find an
operator $L$ such that $q_+(L)<2+\epsilon$.  In
Section~\ref{section:L2-kato} below we will give some estimates for
$q_+(L_w)$ in terms of $[w]_{A_2}$.

We have broken the proof  that $q_+(L_w)>2$ into
a series of discrete steps where we borrow some ideas from \cite{Auscher-Coulhon2005}. We first prove a reverse H\"older
inequality and use Gehring's inequality to get a higher integrability estimate.
We then prove that the Hodge projection is bounded on $L^q(w)$ for a
range of $q>2$ and use this to prove the Riesz transform is also
bounded for exponents greater than~2.  (In Section~\ref{section:riesz}
we give a more complete discussion of the Riesz transform.)  From this
we deduce that $q_+(L_w)>2$.

\subsection*{A reverse H\"older inequality}
Fix a ball $B_0$ and let $u\in H^1(w)$ be any weak solution of $L_wu=0$ in $4B_0$.   Then for any
ball $B$ such that $3B\subset 4B_0$,  we can again prove via a standard argument a
Caccioppoli inequality: 
\[ \bigg(\avgint_B |\grad u|^2 dw\bigg)^{1/2} \leq
\frac{C_1}{r} \bigg(\avgint_{2\,B} |u-u_{2B,w}|^2\,dw\bigg)^{1/2}, \]
where $C_1=C(n,\Lambda/\lambda)[w]_{A_2}^{1/2}\ge 1$.  Fix $q$ such that
\begin{equation}
\max\Big\{\frac{2\,(n-1)}{n}, r_w, \frac{2\,n\,r_w}{2+n\,r_w}\Big\}<q<2;
\label{eq:q-choice}
\end{equation}
such a $q$ exists since $r_w<2$. Our choice of $q$ guarantees that $2<q_w^*$
and also that $2<n\,q/(q-1)$.  Then, by the weighted Poincar\'e inequality,~Theorem~\ref{thm:wtd-poincare},
\begin{equation}
\frac{1}{r} \bigg(\avgint_{2\,B} |u-u_{2B,w}|^2\,dw\bigg)^{1/2}
\leq C_2\bigg(\avgint_{2B} |\grad u|^q\,dw\bigg)^{1/q},
\label{eq:w-Poincare-a}
\end{equation} 
where $C_2=C(n)[w]_{A_2}^\kappa\ge 1$ and
$\kappa=\frac{n\,q-1}{n\,q\,(q-1)}$.  (By our choice of $q$ we can get
this sharp estimate:  see Remark \ref{remark-best-Poi}.   Since $q<2$
we could write $[w]_{A_q}$, but we use that $[w]_{A_q}\le [w]_{A_2}$.)
If we combine these inequalities, we get a reverse H\"older inequality:
\[ \bigg(\avgint_{B} |\grad u|^2\,dw\bigg)^{1/2} \leq
C_1C_2\bigg(\avgint_{2B} |\grad u|^q\,dw\bigg)^{1/q}.  \]

We now apply Gehring's lemma in the setting of spaces of
homogeneous type (see Bj\"orn and
Bj\"orn~\cite[Theorem~3.22]{MR2867756}) to get that there exists $p_0>2$
such that for every such $B$,
\begin{equation} \label{eqn:gehring-bump}
\bigg(\avgint_B |\grad u|^{p_0}\,dw\bigg)^{1/p_0} \leq
C_0\bigg(\avgint_{2B} |\grad u|^2\,dw \bigg)^{1/2}.
\end{equation}
Moreover, we can take the following values:
$C_0=8C_1^2C_2^2[w]_{A_2}^{31}$ and
\begin{equation}
p_0 = 2 + \frac{2-q}{2^{4/q+1}C_1^2C_2^2 [w]_{A_2}^{6/q+17}}. 
\label{eq:value-p0}
\end{equation}

In Section~\ref{section:L2-kato} below we will need these precise
values.  Here, it suffices to note that in
inequality~\eqref{eqn:gehring-bump} we have $p_0>2$.

\subsection*{The Hodge projection}
Define the Hodge projection operator by
\[ T = \grad L_w^{-1/2} (\grad ( L_w^*)^{-1/2})^*,\]
where the  adjoint operators are defined with respect to the inner product in $L^2(w)$.
As we noted in Section~\ref{section:prelim}, the Riesz transform is
bounded on $L^2(w)$; hence, the Hodge
projection is also bounded.  By duality, $(\grad ( L_w^*)^{-1/2})^*\vec{f}=
-L_w^{-1/2} (w^{-1}\div(w\vec{f}))$, and so
\[ T \vec{f} = -\grad L_w^{-1/2} L_w^{-1/2} (w^{-1}\div(w\vec{f}) = -\grad
L_w^{-1} (w^{-1}\div(w\vec{f})). \]

Now fix $\vec{f} \in L^2(w,\C^n) \cap L^{p_0}(w,\C^n)$ such that
$\supp(\vec{f})\subset \rn \setminus 4B_0$.   Let $u\in H^1(w)$ be a
solution to the equation
\[ L_w u = w^{-1} \div(w\vec{f}); \]
by a standard Lax-Milgram argument because $A$ satisfies~\eqref{eqn:degen} 
(cf.~\cite[Theorem~2.2]{fabes-kenig-serapioni82}), we know $u$ exists.
Then
\[ T\vec{f} = -\grad L_w^{-1} L_w u = -\grad u, \]
where equality is in the sense of distributions.  In particular, since
$f=0$ on $4B_0$, $L_w u=0$ on $4B_0$.  Therefore, we can
apply~\eqref{eqn:gehring-bump} to $u$:  on any ball $B$ such that $3B
\subset 4B_0$,
\begin{multline*}
\bigg(\avgint_B |T\vec{f}|^{p_0}\, dw\bigg)^{1/p_0}
= \bigg(\avgint_B |\grad u|^{p_0}\, dw\bigg)^{1/p_0} 
\leq C_0 \bigg(\avgint_{2B} |\grad u|^{2}\, dw\bigg)^{1/2}
= \bigg(\avgint_{2B} |T\vec{f}|^{2}\, dw\bigg)^{1/2}.\!
\end{multline*}
As a consequence of this inequality, we have
by~\cite[Theorem~3.14]{auscher-martell07b} (see also Section~5 of the
same paper) that for all $q$, $2 \leq q < p_0$, $T :
L^q(w,\C^n)\rightarrow L^q(w,\C^n)$.

\subsection*{Boundedness of the Riesz transform}
To show that the Riesz transform $\grad L^{-1/2}_w$  is
bounded, fix $q$ such that 
\[
\max\left( p_-(L_w^*), r_w, p_0'\right)
	= 
\max\left( p_-(L_w^*), r_w, p_0',
  \frac{nr_wp_-(L_w^*)}{nr_w+p_-(L_w^*)}\right) < q' < 2.  \]
(The reason for including $p_-(L_w^*)$ will be made clear below.)
By the above argument we have that $T^*$ is bounded on $L^{q'}(w)$,
where $T^*\vec{f} = -\grad (L^*_w)^{-1}
(w^{-1}\div(w\vec{f}))$. Furthermore, by
Proposition~\ref{prop:reverseRiesz}, we have that
\[ \| (L_w^*)^{1/2} f \|_{L^{q'}(w)}
\leq C\|\grad f\|_{L^{q'}(w)}.  \]
Therefore,
\begin{align*}
\| (\grad L_w^{-1/2})^* \vec{f} \|_{L^{q'}(w)}
& = \| (L_w^*)^{-1/2}(w^{-1} \div(w\vec{f}))\|_{L^{q'}(w)} \\
& = \|  (L_w^*)^{1/2} (L_w^*)^{-1}(w^{-1} \div(w\vec{f}))\|_{L^{q'}(w)}
\\
& \lesssim \| \grad (L_w^*)^{-1}(w^{-1} \div(w\vec{f}))\|_{L^{q'}(w)} \\
& = \|T^* \vec{f}\|_{L^{q'}(w)} \\
&\lesssim \|\vec{f}\|_{L^{q'}(w)}.
\end{align*}
Hence, by duality we have that $\grad L_w^{-1/2} : L^q(w)
\rightarrow L^q(w)$ for all $q$ such that
\[ 2 < q < \min\big(p_+(L_w), r_w', p_0 \big) = q_w; \]
here we have used the fact that by duality, $p_-(L_w^*)' = p_+(L_w)$. 

\subsection*{Boundedness of the gradient of the semigroup}
Finally, we show that if $2< q < q_w$, then $\sqrt{t}\grad e^{-tL_w} :
L^q(w)\rightarrow L^q(w)$.  The desired estimate for $q_+(L_w)$
follows from this:  by Proposition~\ref{prop:K}, part (4),
\[ q_+(L_w) =\sup \tilde{\K}(L_w) \geq q_w > 2.  \]

Fix such a $q$; then by the above estimate for the Riesz transform,
\begin{multline*}
\|\sqrt{t} \grad e^{-tL_w} f\|_{L^q(w)}
= \|\grad L_w^{-1/2} (tL_w)^{1/2} e^{-tL_w} f\|_{L^q(w)} \\
\lesssim \|(tL_w)^{1/2} e^{-tL_w} f\|_{L^q(w)} = \|\varphi_t(L_w)
f\|_{L^q(w)},
\end{multline*}
where $\varphi_t(z) = (tz)^{1/2} e^{-tz}$.  For all $t>0$ this is a
uniformly bounded holomorphic function in the right half plane.
Therefore, since $2<q<p_+(L_w)$, by Proposition~\ref{prop:B-K:weights}
we have that
\[ \|\sqrt{t} \grad e^{-tL_w} f\|_{L^q(w)}  \lesssim 
\|\varphi_t\|_{\infty}
\|f\|_{L^q(w)} 
\lesssim 
\|f\|_{L^q(w)} 
\]
and the bound is independent of $t$.  This completes the proof that
$q_+(L_w)>2$.

\section{Riesz transform estimates}
\label{section:riesz}

In this section we prove $L^p(w)$ norm inequalities for the Riesz
transform $\grad L_w^{-1/2}$.  We have already proved such
inequalities for a small range of values $q>2$ in
Section~\ref{section:q-plus}.  Here we prove the following result.

\begin{prop}\label{prop:ext-RT}
Let $q_-(L_w) <p < q_+(L_w)$.  Then  there exists a constant $C$ such that
\begin{equation}
\label{eq:Riesz} \|\nabla L_w^{-1/2} f\|_{L^p(w)} \le C\|f\|_{L^p(w)}.
\end{equation}
Furthermore, if $v \in A_{p/q_-(L_w)}(w)\cap
RH_{(q_+(L_w)/p)'}(w)$,  then
\begin{equation}
\label{eq:Riesz:w} \|\nabla L_w^{-1/2} f\|_{L^p(v\,dw)} \le C\|f\|_{L^p(v\,dw)}.
\end{equation}
\end{prop}

To prove Proposition \ref{prop:ext-RT} we would like to follow the
same outline as the proof of Proposition \ref{prop:B-K:weights}.  The first step---i.e.,
proving~\eqref{eq:Riesz} holds when  $q_-(L_w)<p<2$--- does work with
the appropriate
changes. However, the second step (i.e., the proof that
\eqref{eq:Riesz:w} holds) runs into difficulties since $\nabla
L_w^{-1/2}$ and the auxiliary operators
$\mathcal{A}_r$ do not commute. One approach to overcoming this
obstacle would be to adapt the
proof in \cite{auscher-martell06} (see also \cite{auscher07}).  In this
case we would need to use an $L^{p_0}(w)$-Poincar\'e inequality which
may not hold unless we assume $w\in A_{p_0}$.  This would yield
estimates in the range $\max\{r_w,q_-(L_w)\}<p<q_+(L_w)$, analogous to
those in Proposition~\ref{prop:reverseRiesz}.

There is, however, an alternative approach.  In
\cite{auscher-martell-08} the authors considered Riesz
transforms associated with the Laplace-Beltrami operator of a complete,
non-compact Riemannian manifold. Their proof avoids Poincar\'e
inequalities for $p$ close to $1$ as these may
not hold. Instead, they use a duality argument based
on ideas in \cite{bernicot-zhao}; this requires that they first prove
that the Riesz transform is bounded for $p>2$ in the appropriate range of
values.  This reverses the order used in the proof of
Proposition~\ref{prop:B-K:weights}.

\begin{proof}[Proof of Proposition \ref{prop:ext-RT}]
  For brevity, let $q_-=q_-(L_w)$ and $q_+=q_+(L_w)$.  To implement
  the approach sketched above, we divide the proof in two steps. First
  we will prove that~\eqref{eq:Riesz} holds when $2<p<q_+$.  We do so
  using Theorem~\ref{theorem:2.2} and some ideas
  from~\cite{auscher07,auscher-martell06}.   We note that since the Riesz transform and $\mathcal{A}_r$ do not commute, we will
  use an $L^2(w)$-Poincar\'e inequality.  This holds since $w\in A_2$:
   the problem with using the Poincar\'e inequality only occurs with exponents
   less than $2$.    The second step is to prove that
   \eqref{eq:Riesz:w} holds by
  adapting the proof in \cite{auscher-martell-08}.  Here we will use
 duality and a result from \cite{auscher-martell07b} that is
  based on good-$\lambda$ inequalities.  Inequality~\eqref{eq:Riesz}
  then holds when $q_-<p<2$ by taking $v\equiv 1$.

\medskip

To apply Theorem \ref{theorem:2.2}, fix $2<p<q_+$
and let $T=\nabla L_w^{-1/2}$, $S=I$ and $\mathcal{D}=L^\infty_c$.  Let
$p_0=2$ and fix $q_0$ such that $2<p<q_0<q_+$.  As before we take
$\mathcal{A}_{r}=I-( I-e^{-r^{2}L_{w}})^{m}$, where $m$ 
will be chosen below. We first show that
\eqref{eqn:2.1} holds. Let $f\in L^\infty_c$ and decompose it as in \eqref{decomp-f}; then we
have
\begin{align*}
\bigg(\aver{B} |\nabla L_w^{-1/2} (I-e^{-r^2\,L_w})^m
f|^{2}\,dw\bigg)^{\frac1{2}}
\le
\sum_{j\ge 1}
\bigg(
\aver{B} |\nabla L_w^{-1/2} (I-e^{-r^2\,L_w})^m f_j|^{2}\,dw
\bigg)^{\frac1{2}}.
\end{align*}
To estimate the first term, note that  $\nabla L_w^{-1/2}$ and
$e^{-r^2\,L_w}$ are bounded on $L^{2}(w)$ by
Theorems~\ref{thm:L2-gaffney} and~\ref{theorem:degen-kato}.
Hence,
\begin{equation}\label{RT-est:f1}
\Big( \aver{B} |\nabla L_w^{-1/2}(I-e^{-r^2\,L_w})^m
f_1|^{2}\,dw \Big)^{\frac1{2}}
\lesssim
\Big( \aver{4\,B} |f|^{2}\,dw\Big)^{\frac1{2}}.
\end{equation}

Fix $j\ge 2$; to get the desired $L^2$ estimates  we will use the
$L^2$ bounds for the gradient of the square function.. If $h\in L^2(w)$, by
\eqref{defi-RT}
\begin{equation}\label{RT-repre}
\nabla L_w^{-1/2}(I-e^{-r^2\,L_w})^m h
=
\frac1{\sqrt{\pi}} \int_0^\infty \sqrt{t}\,\nabla \varphi(L_w,t)
h\,\frac{dt}{t},
\end{equation}
where $\varphi(z,t)=e^{-t\,z}\,(1-e^{-r^2\,z})^m \in
\H^\infty_0(\Sigma_\mu)$.   We can therefore use the integral
representation \eqref{eqn:L2-holo-rep} for $\varphi(\cdot, t)$.
The function
$\eta(\cdot, t)$ in this representation satisfies
\begin{equation*}\label{eta-RT}
|\eta(z,t)|
\lesssim
\frac{r^{2\,m}}{(|z|+t)^{m+1}},
\qquad
z\in \Gamma, \; t>0.
\end{equation*}
By Theorem \ref{thm:L2-gaffney}, $\sqrt{z}\,\nabla
e^{-z\,L_w}\in\offw{2}{2}$; hence,
\begin{align}
&\bigg( \aver{B}
\bigg|
\int_{\Gamma} \eta(z)\, \sqrt{t}\,\nabla e^{-z\,L_w} f_j\,dz
\bigg|^{2}\,dw\bigg)^{\frac1{2}} \nonumber
\\
&\quad\le
\int_{\Gamma}
\bigg(
\aver{B} |\sqrt{z}\,\nabla e^{-z\,L_w}
f_j|^{2}\,dw\bigg)^{\frac1{2}}\,
\frac{\sqrt{t}}{\sqrt{|z|}}\,|\eta(z)|\,|dz| \nonumber
\\
&\quad\lesssim
2^{j\,\theta_1} \int_{\Gamma}
\dec{\frac{2^j\,r}{\sqrt{|z|}}}^{\theta_2}
\expt{-\frac{\alpha\,4^j\,r^2}{|z|}}\,
\frac{\sqrt{t}}{\sqrt{|z|}}\,|\eta(z)|\,|dz| \,
\bigg(\aver{C_j(B)}
|f|^{2}\,dw\bigg)^{\frac1{2}} \nonumber
\\
&\quad\lesssim
2^{j\,\theta_1} \int_0^\infty
\dec{\frac{2^j\,r}{\sqrt{s}}}^{\theta_2}
\expt{-\frac{\alpha\,4^j\,r^2}{s}}\, \frac{\sqrt{t}}{\sqrt{s}}
\,\frac{r^{2\,m}}{(s+t)^{m+1}}\,ds\, \bigg(\aver{C_{j}(B)}
|f|^{2}\,dw\bigg)^{\frac1{2}}.
\label{est-Riesz-aux}
\end{align}
When  $2\, m>\theta_{2}$,
\begin{align}\label{eqn:integral}
\int_{0}^\infty\int_0^\infty
\dec{\frac{2^j\,r}{\sqrt{s}}}^{\theta_2}
\expt{-\frac{\alpha\,4^j\,r^2}{s}}\,
  \frac{\sqrt{t}}{\sqrt{s}}\,\frac{r^{2\,m}}{(s+t)^{m+1}}\,ds\frac
  {dt}t = C\, 4^{-j\, m}.
\end{align}
If we insert this into the representation \eqref{eqn:L2-holo-rep}
we get
\begin{align}
\bigg( \aver{B} |\nabla e^{-t\,L_w}(I-e^{-r^2\,L_w})^m  f_j|^{2}\,dw\bigg)^{\frac1{2}}
&\lesssim
\int_0^\infty
\Big( \aver{B}
|\sqrt{t}\,\nabla \varphi(L_w,t) f_j|^{2}\,dw\Big)^{\frac1{2}} \, \frac{dt}{t} \nonumber
\\
&\lesssim 2^{j\,(\theta_1-2\,m)}
\Big(\aver{C_{j}(B)}
|f|^{2}\,dw\Big)^{\frac1{2}}. \label{RT-est:fj}
\end{align}
If we now combine~\eqref{RT-est:f1} and \eqref{RT-est:fj} we get
\eqref{eqn:2.1} with $g(j)=C_m\,2^{j\,(\theta_1-2\,m)}$;  if we also
fix $2m>\theta_1$, we get that $\sum g(j)<\infty$.

\medskip

We now show that~\eqref{eqn:2.2} holds.    As we remarked above, the
Riesz transform does not commute with $\mathcal{A}_r$.  To overcome
this obstacle, we will prove an off-diagonal estimate for the gradient
of the semigroup (using the $L^2(w)$-Poincar\'e inequality), and
then use an approximation argument to get the desired estimate for the
Riesz transform.

More precisely, we claim
that for every $f \in H^1(w)$ and
$1\le k\le m$,
\begin{equation}
\label{eq:est-Riesz:2}
\bigg( \aver{B} |\nabla
e^{-k\,r^2\,L_w}f|^{q_0}\,dw\bigg)^{\frac1{q_0}}
\le
\sum_{j\ge 1} g(j)\,
\bigg( \aver{2^{j+1}\,B}
|\nabla f|^{2}\,dw\bigg)^{\frac1{2}},
\end{equation}
where
$g(j)=C_m\,2^j\,\sum_{l\ge j} 2^{l\,\theta}\,e^{-\alpha\,4^l}$.
Assume for the moment that~\eqref{eq:est-Riesz:2} holds. Then for
every $\epsilon>0$ we can apply this estimate to $S_{\ep}f$ (defined
by~\eqref{defi-RT:trunc}) since $S_{\ep}f\in H^{1}(w)$.   Moreover, we
have that $\A_r$ and $S_{\ep}$ commute, and so if we expand
$\A_r=I-(I-e^{-r^2\,L})^m$ and apply~\eqref{eq:est-Riesz:2}, we get
\[ \bigg( \aver{B} |\nabla S_{\ep} \A_{r}
f|^{q_0}\,dw\bigg)^{\frac1{q_0}}
\le
C_m\sum_{j\ge 1} g(j)\,
\bigg( \aver{2^{j+1}\,B}
|\nabla S_{\ep} f|^{2}\,dw\bigg)^{\frac1{2}}. \]
If we let $\ep$ go to $0$, we obtain \eqref{eqn:2.2}.  (The
justification of this uses the observations made in
Section~\ref{section:prelim} after~\eqref{defi-RT:trunc} and is left
to the reader.)  Moreover, we have that $\sum_{j\ge 1}g(j)<\infty$,
and so by Theorem~\ref{theorem:2.2} with $v\equiv 1$ (which
trivially satisfies $v\in A_{p/2}(w)\cap RH_{(q_0/p)'}(w)$) we
have that \eqref{eq:Riesz} holds for $f\in L^\infty_c$ and for
every $2<p<q_+$.

To complete this step we need to prove~\eqref{eq:est-Riesz:2}. Fix
$1\le k\le m$ and $f\in H^{1}(w)$.
Let $h=f-f_{4B,w}$, where $f_{4B,w}=\avgint_{4B} f \,dw$.   Then by the conservation property (see
\cite{DCU-CR2013}, or the proof in \cite[Section 2.5]{auscher07}),
 $e^{-t\, L_w}1=1$ for all $t>0$, and so
\[ \nabla e^{-k\,r^2\,L_w} f
=
\nabla e^{-k\,r^2\,L_w} (f-f_{4\,B,w})
=
\nabla e^{-k\,r^2\,L_w} h
=
\sum_{j\ge 1} \nabla e^{-k\,r^2\,L_w}h_j, \]
where $h_j=h\,\bigchi_{C_j(B)}$. Hence,
$$
\bigg(\aver{B} |\nabla
e^{-k\,r^2\,L_w}f|^{q_0}\,dw\bigg)^{\frac1{q_0}}
\le
\sum_{j\ge 1}
\bigg(\aver{B}  |\nabla
e^{-k\,r^2\,L_w}h_j|^{q_0}\,dw\bigg)^{\frac1{q_0}}.
$$
Since $2<q_0<q_+$, by Proposition~\ref{prop:K},
$\sqrt{t}\,\nabla e^{-t\,L_w}\in\offw{2}{q_0}$. If we apply this and
the $L^{2}(w)$-Poincar\'e inequality (see Remark
\ref{remark:Poincare-non-smooth} with $p=q=2$), then for each $j\geq
1$ we get
\begin{align*}
&\bigg(\aver{B}
|\nabla e^{-k\,r^2\,L_w}h_j|^{q_0}\,dw\bigg)^{\frac1{q_0}} \\
& \qquad \lesssim
\frac{2^{j\,(\theta_1+\theta_2)}\,e^{-\alpha\,4^j}}r\,
\bigg(\aver{C_j(B)}
|h_j|^{2}\,dw\bigg)^{\frac1{2}}
\nonumber
\\
&\qquad\le
\frac{2^{j\,(\theta_1+\theta_2)}\,e^{-\alpha\,4^j}}r\,
\bigg(\aver{2^{j+1}\,B}|f-f_{4\,B,w}|^{2}\,dw\bigg)^{\frac1{2}}
\nonumber
\\
&\qquad\le
\frac{2^{j\,(\theta_1+\theta_2)}\,e^{-\alpha\,4^j}}r\,
\bigg(
\bigg(\aver{2^{j+1}\,B}\hskip-7pt
|f-f_{2^{j+1}\,B,w}|^{2}\,dw\bigg)^{\frac1{2}} +\sum_{l=2}^j
|f_{2^l\,B,w}-f_{2^{l+1}\,B,w}| \bigg)
\nonumber
\\
&\qquad\lesssim
\frac{2^{j\,(\theta_1+\theta_2)}\,e^{-\alpha\,4^j}}r\,\sum_{l=1}^{j}
\bigg(\aver{2^{l+1}\,B}
|f-f_{2^{l+1}\,B,w}|^{2}\,dw\bigg)^{\frac1{2}}
\nonumber
\\
&\qquad\lesssim
2^{j\,(\theta_1+\theta_2)}\,e^{-\alpha\,4^j}\,\sum_{l=1}^{j} 2^{l}\,
\bigg(\aver{2^{l+1}\,B}
|\nabla f|^{2}\,dw\bigg)^{\frac1{2}}.
\end{align*}
If we combine these two estimates and exchange the order of summation
we get~\eqref{eq:est-Riesz:2} with
$\theta=\theta_{1}+\theta_{2}$. This completes the proof
that~\eqref{eq:Riesz} holds when $2<p<q_+$.

\medskip

For the second step of our proof we show
that~\eqref{eq:Riesz:w} holds for all $p$, $q_-<p<q_+$ and
$v\in A_{p/q_-}(w)\cap RH_{(q_+/p)'}(w)$.  Fix such a $p$ and $v$;
then by the openness
properties of $A_q$ and $RH_s$ weights,
there exist
$p_0, q_0$ such that
\begin{equation*}
q_-<p_0<\min\{p,2\}\le \max\{p,2\}<q_0<q_+
\quad \mbox{and} \quad
v\in A_{p/p_0}(w) \cap RH_{(q_0/p)'}(w).
\end{equation*}
By the duality properties of weights~\cite[Lemma 4.4]{auscher-martell07b},
\[ u=v^{1-p'}\in A_{p'/q_0'}(w)\cap RH_{(p_0'/p')'}(w). \]
Let
$T=\nabla L_w^{-1/2}$; then  $T$ is bounded from
$L^p(\re^n,v\,dw)$ to $L^p(\re^n;\co^n, v\,dw)$ if and only if
 $T^*$ is bounded from $L^{p'}(\re^n;\co^n, u\,dw) $ to
$L^{p'}(\re^n; u\,dw)$.  (Note that $T$ takes scalar valued
functions to vector functions valued and $T^*$ the opposite.)

Therefore, it will suffice to prove the boundedness of $T^*$.  We will
do so using a particular case of \cite[Theorem
3.1]{auscher-martell07b}. This result is stated there in the
Euclidean setting but it extends to spaces of homogeneous type. Here  we
give the weighted version we need:
see~\cite[Section 5]{auscher-martell07b}.

\begin{theor}\label{theor:good-lambda:w}
Fix $1<q<\infty$, $a\ge 1$ and $u\in RH_{s'}(w)$, $1<s<\infty$.
Then there exists $C>1$ with the
following property:  suppose $F\in L^1(w)$ and  $G$ are
non-negative measurable functions such that for any ball $B$ there
exist non-negative functions $G_B$ and $H_B$ with $F(x)\le G_B(x)+
H_B(x)$ for a.e. $x\in B$ and, for all $x\in B$,
\begin{equation}\label{H-Q:G-Q}
\Big(\aver{B} H_B^q\, dw \Big)^{\frac1q}
\le
a\, M_w F(x),
\qquad\qquad
\aver{B} G_B\, dw
\le
G(x),
\end{equation}
where $M_w$ is the Hardy-Littlewood maximal function with respect to $dw$.
Then for  $1<t< q/s$,
\begin{equation}\label{good-lambda:Lp:w}
\|M_w F\|_{L^t(u\,dw)}
\le
C\,\|G\|_{L^t(u\,dw)}.
\end{equation}
\end{theor}

\medskip

To apply Theorem~\ref{theor:good-lambda:w}, fix
$\vec{f}\in L^\infty_c (\re^n;\co^n)$, and let $h=T^* \vec{f}$
and $F=|h|^{q_0'}$.    Then $F\in L^1(w)$:  by the argument above,
 since $2<q_0<q_+$, $T$ is bounded from $L^{q_0}(\re^n, w)$ to
$L^{q_0}(\re^n;\co^n, w)$, thus  $T^*$ is bounded
from $L^{q_0'}(\re^n;\co^n, w)$ to $L^{q_0'}(\re^n, w)$.

Now let
$\A_{r}=I- (I-e^{-r^2\, L_w})^m$, where $m>0$ will be fixed below.
Given a ball $B$ with radius $r$, we define
\[ F
\le
2^{q_0'-1}\,|(I-\A_{r})^* h|^{q_0'}
+
2^{q_0'-1}\,|\A_{r}^* h|^{q_0'}
\equiv G_B+H_B,  \]
where, as before, the  adjoint  is with
respect to $L^2(w)$.   To complete the proof, suppose for the
moment that we could prove~\eqref{H-Q:G-Q} with $q=p_0'/q_0'$
and $G=M_w(|\vec{f}|^{q_0'})$.
Since $u \in RH_{(p_0'/p')'}(w)$, by the openness property of reverse
H\"older weights, $u \in RH_{s'}(w)$ for some $s<p_0'/p'$.
Then if we let $t=p'/q_0'=(p_0'/q_0')/(p_0'/p')<q/s$, we have
$u \in A_t(w)$, and so $M_w$ is bounded on $L^t(u\,dw)$.  Therefore,
by~\eqref{good-lambda:Lp:w},
\begin{equation*}
\|T^* \vec{f}\|_{L^{p'}(u\,dw)}^{q_0'}
\le
\|M_w F\|_{L^t(u\,dw)}
\le 
C\,
\|G\|_{L^t(u\,dw)}
=
C\,
\|M_w(|\vec{f}|^{q_0'})\|_{L^t(u\,dw)}
\lesssim
\big\|\vec{f}\big\|_{L^{p'}(u\,dw)}^{q_0'}.
\end{equation*}

To complete the proof we need to show that ~\eqref{H-Q:G-Q} holds.  We
first estimate $H_B$.  By
duality there exists $g\in L^{p_0}(B,dw/w(B))$ with norm $1$ such that
for all $x\in B$,
\begin{align*}
\bigg(\aver{B} H_B^q\,dw\bigg)^{\frac1{q\,q_0'}}
&
\lesssim
w(B)^{-1}\,\int_{\re^n} |h|\,|\A_{r} g|\,dw
\\
&\lesssim
\sum_{j=1}^\infty 2^{j\,D} \bigg(\aver{C_j(B)} |h|^{q_0'}\,dw\bigg)^\frac1{q_0'}\,\bigg(\aver{C_j(B)} |\A_{r} g|^{q_0}\,dw\bigg)^\frac1{q_0}
\\
&\lesssim
M_w F(x)^\frac1{q_0'}\sum_{j=1}^\infty 2^{j\,(D+\theta_1+\theta_2)} e^{-\alpha\,4^{j}} \bigg(\aver{B} |g|^{p_0}\,dw\bigg)^\frac1{p_0}
\\
&\lesssim
M_w F(x)^\frac1{q_0'}
,
\end{align*}
where in the second to last inequality we used the fact that by our choice of $p_0,\,q_0$,
$e^{-t\,L_w}\in\offw{p_0}{q_0}$, and so $\A_r$ is as well.

\medskip

We now estimate $G_B$.  Again by duality there exists
$g\in L^{q_0}(B,dw/w(B))$ with norm $1$ such that for all $x\in B$,
\begin{align}
\bigg(\aver{B} G_B\,dw\bigg)^{\frac1{q_0'}}
&
\lesssim
w(B)^{-1}\,\int_{\re^n} |\vec{f}|\,|T(I-\A_{r}) g|\,dw
\nonumber\\
&\lesssim
\sum_{j=1}^\infty 2^{j\,D} \bigg(\aver{C_j(B)} |\vec{f}|^{q_0'}\,dw\bigg)^\frac1{q_0'}\,\bigg(\aver{C_j(B)} |T(I-\A_{r}) g|^{q_0}\,dw\bigg)^\frac1{q_0}
\nonumber\\
&\le
M_w(|\vec{f}|^{q_0'})(x)^\frac1{q_0'}\sum_{j=1}^\infty 2^{j\,D} \bigg(\aver{C_j(B)} |T(I-\A_{r}) g|^{q_0}\,d\mu\bigg)^\frac1{q_0}.
\label{est:GB}
\end{align}

To estimate each term in the sum we argue as in the first half of the
proof.  When $j=1$, $\nabla L_w^{-1/2}$ and $e^{-r^2\,L_w}$ are
bounded on $L^{q_0}(w)$ by the first part of the proof and
Theorem~\ref{thm:L2-gaffney}. Hence,
\begin{equation}\label{RT-est:GB-1}
\bigg( \aver{4\,B} |\nabla L_w^{-1/2}(I-e^{-r^2\,L_w})^m g|^{q_0}\,dw \bigg)^{\frac1{q_0}}
\lesssim
\bigg( \aver{B} |g|^{q_0}\,dw\bigg)^{\frac1{q_0}}
=1.
\end{equation}

For $j\ge 2$ we use the integral representation
\eqref{RT-repre}. If we estimate as in \eqref{est-Riesz-aux}, with
the roles of $B$ and $C_j(B)$ switched and using the fact that
$\sqrt{z}\,\nabla e^{-z\,L_w}\in\offw{q_0}{q_0}$ since $2<q_0<q_+$, we
see that
\begin{align*}
&\bigg( \aver{C_j(B)}
\bigg|
\int_{\Gamma} \eta(z)\, \sqrt{t}\,\nabla e^{-z\,L_w} g\,dz
\bigg|^{q_0}\,dw\bigg)^{\frac1{q_0}} \nonumber
\\
&\quad\le
\int_{\Gamma}
\bigg(
\aver{C_j(B)} |\sqrt{z}\,\nabla e^{-z\,L_w}
g|^{q_0}\,dw\bigg)^{\frac1{q_0}}\,
\frac{\sqrt{t}}{\sqrt{|z|}}\,|\eta(z)|\,|dz| \nonumber
\\
&\quad\lesssim
2^{j\,\theta_1} \int_{\Gamma}
\dec{\frac{2^j\,r}{\sqrt{|z|}}}^{\theta_2}
\expt{-\frac{\alpha\,4^j\,r^2}{|z|}}\,
\frac{\sqrt{t}}{\sqrt{|z|}}\,|\eta(z)|\,|dz| \,
\bigg(\aver{B}
|g|^{q_0}\,dw\bigg)^{\frac1{2}} \nonumber
\\
&\quad\lesssim
2^{j\,\theta_1} \int_0^\infty
\dec{\frac{2^j\,r}{\sqrt{s}}}^{\theta_2}
\expt{-\frac{\alpha\,4^j\,r^2}{s}}\, \frac{\sqrt{t}}{\sqrt{s}}
\,\frac{r^{2\,m}}{(s+t)^{m+1}}\,ds.
\end{align*}
If we take $2\, m>\theta_{2}$, we can combine this
with~\eqref{eqn:integral}.  We can then insert this estimate into the representation
\eqref{eqn:L2-holo-rep} to get that for every $j\ge 2$,
\begin{multline}
\bigg( \aver{C_j(B)} |\nabla e^{-t\,L_w}(I-e^{-r^2\,L_w})^m  g|^{q_0}\,dw\bigg)^{\frac1{q_0}}
\\
\lesssim
\int_0^\infty
\bigg( \aver{C_j(B)}
|\sqrt{t}\,\nabla \varphi(L_w,t) g|^{q_0}\,dw\bigg)^{\frac1{q_0}} \, \frac{dt}{t}
\lesssim 2^{j\,(\theta_1-2\,m)}.
\label{RT-est:GB-j}
\end{multline}
Taken together, \eqref{est:GB}, \eqref{RT-est:GB-1} and  \eqref{RT-est:GB-j} yield
$$
\bigg(\aver{B} G_B\,dw\bigg)^{\frac1{q_0'}}
\lesssim
M_w(|\vec{f}|^{q_0'})(x)^\frac1{q_0'}\sum_{j=1}^\infty 2^{j\,(D+\theta_1-2\,m)}
\lesssim
M_w(|\vec{f}|^{q_0'})(x)^\frac1{q_0'}
=
G(x)^\frac1{q_0'}
,
$$
provided we take $m$ large enough so that $D+\theta_1-2\,m<0$.  This
completes the estimate of $H_B$ and $G_B$ and so completes our proof.
\end{proof}

\section{Square function estimates for the gradient of the semigroup}
\label{section:square-function-gradient}

In this section we prove $L^p(w)$ estimates for the vertical square function
\[ G_{L_{w}}f( x)
=\bigg( \int_{0}^{\infty }|t^{1/2}\nabla e^{-tL_{w}}f(x)|^{2}\frac{dt}{t}
\bigg) ^{1/2}. \]

\begin{prop} \label{prop:grad-square-function}
Let $q_-(L_w) < p < q_+(L_w)$.  Then
\begin{equation} \label{eqn:grad-square-function-unwtd}
 \left\Vert G_{L_{w}}f\right\Vert
_{L^{p}\left( w\right) }\lesssim \left\Vert f\right\Vert _{L^{p}\left(
    w\right)}.
\end{equation}
Furthermore,  if $v \in A_{p/q_-(L_w)}(w)\cap
RH_{(q_+(L_w)/p)'}(w)$,  then
\begin{equation} \label{eqn:-grad-square-function-wt}
 \left\Vert G_{L_{w}}f\right\Vert
_{L^{p}(v\,dw) }\lesssim \left\Vert f\right\Vert _{L^{p}(v\,dw)}.
\end{equation}
\end{prop}

We can also prove a reverse inequality for $G_{L_w}$.  To do so we need
to  introduce an auxiliary operator.  Define the
weighted Laplacian by $\Delta _{w}=- w^{-1}\div w\nabla $: i.e.,
$\Delta_w$ is the
operator $L_w$ if we take the matrix $A$ to
be $wI$, where $I$ is the identity matrix.

\begin{prop} \label{prop:grad-square-reverse}
Let $q_+(\Delta_w)' < p < \infty$.  Then
\begin{equation} \label{eqn:grad-square-reverse-unwtd}
 \Vert f \Vert_{L^{p}(w) }
\lesssim \Vert G_{L_{w}} f \Vert _{L^{p}( w)}.
\end{equation}
Furthermore,  if $v \in A_{p/q_+(\Delta_w)'}(w)$,  then
\begin{equation} \label{eqn:-grad-square-reverse-wt}
\Vert f\Vert_{L^{p}(v\,dw) }
\lesssim \Vert  G_{L_{w}} f\Vert _{L^{p}(v\,dw)}.
\end{equation}
\end{prop}

\medskip

\begin{proof}[Proof of Proposition~\textup{\ref{prop:grad-square-function}}]
The proof could be done in a way similar to those for the square
function $g_{L_w}$ in Section~\ref{section:square-function}.  However,
we will give a shorter proof that uses the Riesz transform estimates
from Section~\ref{section:riesz}.

Let $q_-=q_-(L_w)$ and $q_+=q_+(L_w)$.  Fix $p$,
\[ q_{-}=p_-(L_w)<p<q_+\leq p_+(L_w), \]
and $v\in A_{p/q_{-}}(w) \cap
RH_{\left( q_{+}/p\right) ^{\prime }}(w)$.   Then by Proposition
\ref{prop:ext-RT}, the Riesz transform is bounded on $L^{p}(v\,dw)$, and so by
Lemma \ref{lemma-7.4} it has a bounded extension to
$L_{\mathbb{H}}^{p}(v\,dw) $: i.e., if
$g\left( x,t\right) \in L_{\mathbb{H}}^{p}(v\,dw) $, then
$\| \nabla L_{w}^{-1/2}g\| _{L_{\mathbb{H}}^{p}(v\,dw) }\lesssim
\|g\|_{L_{\mathbb{H}}^{p}(v\,w) }$, where the extension of 
$\nabla L_{w}^{-1/2}$ to
$\HH$-valued  functions is defined  for $x\in \RR^n$ and $t>0$ by $(\nabla L_{w}^{-1/2}
g)(x,t)= \nabla L_{w}^{-1/2}\big( g(\cdot,t)\big)(x)$.  

Define $g_{f}\left( x,t\right) =\left( tL_{w}\right) ^{1/2%
}e^{-tL_{w}}f(x) $ and $G_{f}\left( x,t\right) =t^{1/2%
}\nabla e^{-tL_{w}}f(x) $;
then we clearly have
$\left\Vert g_{L_{w}}f\right\Vert _{L^{p}\left( v\,dw\right)
}=\left\Vert g_{f}\right\Vert _{L_{\mathbb{H}}^{p}\left( v\,dw\right)
}$ and
$\left\Vert G_{L_{w}}f\right\Vert _{L^{p}\left( v\,dw\right)
}=\left\Vert G_{f}\right\Vert _{L_{\mathbb{H}}^{p}\left( v\,dw\right)
}$.  Furthermore, $G_f(x,t)=\nabla L_{w}^{-1/2} (g_f(\cdot,t))(x)=(\nabla L_{w}^{-1/2} g_f)(x,t)$. Hence, 
\begin{multline*}
\left\Vert G_{L_{w}}f\right\Vert _{L^{p}(v\,dw) }
=
\left\Vert
G_f \right\Vert _{L_{\mathbb{H}}^{p}(v\,dw)} 
=\left\Vert
\nabla L_{w}^{-1/2} g_f\right\Vert _{L_{\mathbb{H}}^{p}(v\,dw)}
\\
\lesssim \left\Vert g_f\right\Vert
_{L_{\mathbb{H}}^{p}(v\,dw)}=\left\Vert g_{L_{w}}f\right\Vert _{L^{p}\left(
v\,dw\right) }
\lesssim
\left\Vert f\right\Vert _{L^{p}\left(
v\,dw\right) }
.
\end{multline*}
To prove the last inequality we used Proposition
\ref{prop:square-function}; we also used the fact that 
$q_{-}=p_-(L_w)<p<q_+\leq p_+(L_w)$ and $v\in A_{p/q_{-}}(w) \cap
RH_{\left( q_{+}/p\right) ^{\prime }}(w)$, which together imply  that $v\in A_{p/p_{-}(L_w)}(w) \cap
RH_{\left( p_{+}(L_w)/p\right) ^{\prime }}(w)$. 
This proves~\eqref{eqn:-grad-square-function-wt}.  To prove
inequality~\eqref{eqn:grad-square-function-unwtd}, we take $v\equiv 1$.
\end{proof}

\bigskip

To prove Proposition~\ref{prop:grad-square-reverse} we need the following
identity relating $G_{L_w}$ and $\Delta_w$.  It is a straightforward
extension of a similar unweighted result given in
\cite[Section~7.1]{auscher07}. For completeness we include the proof.

\begin{lemma}
\label{lemma-dual-square}If $f,g\in L_{c}^{\infty }(w) $ then%
\begin{equation*}
\left\vert \int_{\mathbb{R}^{n}}f(x) \overline{g(x) }%
~dw\right\vert \leq \left( \Lambda +1\right) \int_{\mathbb{R}%
^{n}}G_{L_{w}}f(x) ~\overline{G_{\Delta _{w}}g(x) }%
~dw.
\end{equation*}%
\end{lemma}

\begin{proof}
By the definition and properties of the operators $L_{w}$ and
  $\Delta_{w}$ we have that
\begin{align*}
& \int_{\mathbb{R}^{n}}f(x) \overline{g(x)}\,dw \\
& \qquad \qquad =\lim_{\varepsilon \downarrow 0}\int_{\mathbb{R}^{n}}e^{-\varepsilon
L_{w}}f(x)\overline{e^{-\varepsilon \Delta_w }g(x)}\,dw-\lim_{R\uparrow \infty }\int_{%
\mathbb{R}^{n}}e^{-RL_{w}}f(x)\overline{e^{-R\Delta_w }g(x)}\,dw \\
&\qquad \qquad =-\int_{0}^{\infty }\frac{d}{dt}\int_{\mathbb{R}^{n}}e^{-tL_{w}}f(x)\overline{%
e^{-t\Delta_w }g(x)}\,dw\,dt \\
&\qquad \qquad =\int_{0}^{\infty }\int_{\mathbb{R}^{n}}\left( L_{w}e^{-tL_{w}}f(x)\overline{%
e^{-t\Delta_w }g(x)}+e^{-tL_{w}}f(x)\overline{\Delta_w e^{-t\Delta_w }g(x)}\right)\,dw\,dt \\
&\qquad \qquad =\int_{0}^{\infty }\int_{\mathbb{R}^{n}}\left( A(x)w(x)^{-1} +I\right) \big( \nabla e^{-tL_{w}}f(x) ~\overline{%
\nabla e^{-t\Delta_w }g(x) }\big)\,dw\,dt.
\end{align*}
Since $\| Aw^{-1}\| _{\infty }\leq \Lambda $, if we apply
H\"older's inequality in the $t$ variable we get the desired result.
\end{proof}

\begin{proof}[Proof of Proposition~\textup{\ref{prop:grad-square-reverse}}]
  As a consequence of the Gaussian estimate for weighted operators
  with real symmetric coefficients that were proved
  in~\cite{cruz-uribe-riosP,DCU-CR2014}, we have that  $\Delta_w\in \offw{1}{\infty}$.  In
  particular,   $q_-(\Delta_w) = p_-(L_{\Delta_w})=1$.  Further, by
  the results in Section~\ref{section:q-plus} we have that
  $q_+(\Delta_w)>2$. 

Therefore, by Proposition~\ref{prop:grad-square-function},
if $1<p' < q_+(\Delta_w)$,  and
\begin{equation} \label{eqn:dual-wt}
 u \in A_{p'}(w)\cap RH_{(q_+(\Delta_w)/p')'}(w), 
\end{equation}
then 
\begin{equation} \label{eqn:foobar}
 \|G_{\Delta_w} f\|_{L^{p'}(u\,dw)} \lesssim \|f\|_{L^{p'}(u\,dw)}.
\end{equation}
We want to apply inequality~\eqref{eqn:foobar} with $u=v^{1-p'}$.
By~\cite[Lemma~4.4]{auscher-martell07b}, the
condition~\eqref{eqn:dual-wt} is equivalent to
$v \in A_{p/q_+(w)'}(w)$.

Now fix $f,g\in L_{c}^{\infty }$, and a weight
$v\in A_{p/q_+(w)'}(w) $.  Then by Lemma~\ref{lemma-dual-square}, for
$q_{+}(\Delta_w)'<p<\infty$,
\begin{align*}
\left\vert \int_{\mathbb{R}^{n}}f(x) g(x)
~dw\right\vert
&\leq \left( \Lambda +1\right) \int_{\mathbb{R}%
^{n}}G_{L_{w}}f(x) ~G_{\Delta _{w}}g(x) ~dw \\
&=\left( \Lambda +1\right) \int_{\mathbb{R}^{n}}G_{L_{w}}f(x)
~G_{\Delta _{w}}g(x) \,v^{1/p}\,v^{-1/p}\,dw \\
&\leq \left( \Lambda +1\right) \left\Vert G_{L_{w}}f\right\Vert
_{L^{p}( v\,dw) }\left\Vert G_{\Delta _{w}}g\right\Vert
_{L^{p^{\prime }}( v^{1-p'}\,dw) } \\
&\lesssim \left\Vert G_{L_{w}}f\right\Vert _{L^{p}( v\,dw)}
\left\Vert g\right\Vert _{L^{p^{\prime }}( v^{1-p'}\,dw)};
\end{align*}
the last inequality follows from~\eqref{eqn:foobar}.
If we take $g=\mathrm{sign}\left( f\right) \left\vert f\right\vert
^{p-1}v$, we get
\[
\left\Vert f\right\Vert _{L^{p}( v\,dw) } ^p
\lesssim \left\Vert
G_{L_{w}}f\right\Vert _{L^{p}( v\,w) }\left\Vert \left\vert
f\right\vert ^{p-1}v\right\Vert _{L^{p^{\prime }}( v^{1-p'}\,dw) }
=\left\Vert G_{L_{w}}f\right\Vert _{L^{p}( v\,dw) }\left\Vert
f\right\Vert _{L^{p}( v\,dw) }^{p/p^{\prime }}.
\]
This immediately gives us
the desired inequality.
\end{proof}

\section{Unweighted $L^2$ Kato estimates}
\label{section:L2-kato}

In this section we prove unweighted $L^2$ estimates for the
operators we have considered in the previous sections.  These will all
be consequences of the weighted $L^p(v\,dw)$ estimates we have already
proved:  it will only be necessary to find further conditions on $w\in
A_2$ so that
the weight $v=w^{-1}$ satisfies the requisite conditions.  

We are particularly interested in power weights and we recall some
well-known facts about them.   Define 
$w_\alpha(x)=|x|^{\alpha}$, $\alpha>-n$; this restriction guarantees that
$w_\alpha$ is locally integrable.   We can exactly determine the
Muckenhoupt $A_p$ and reverse H\"older $RH_s$ classes of these weights
in terms of $\alpha$:   if $-n< \alpha\le 0$, then $w\in A_1$; for
$1<p<\infty$,  $w\in A_p$ if $-n<\alpha <n\,(p-1)$.  Furthermore,  if $0\le
\alpha<\infty$, $w\in RH_\infty$;  for  $1<q<\infty$, $w\in RH_q$.  if
$-n/q<\alpha<\infty$. Hence, we easily see that
\begin{equation}\label{sw-rw:w-alpha}
r_{w_{\alpha}}=\max\{1, 1+\alpha/n\},
\qquad\quad
s_{w_{\alpha}}=\big(\max\{1, (1+\alpha/n)^{-1}\}\big)'.
\end{equation}  

We first consider the semigroup $e^{-tL_w}$, the functional calculus,
and the square function $g_{L_w}$, since these estimates will depend
on $p_-(L_w)$ and $p_+(L_w)$ and we have good estimates for these
quantities.

\begin{theor} \label{thm:semigroup-nowt}
Given a weight $w\in A_2$, suppose $1\leq r_w<1+\frac{2}{n}$ and
$s_w>\frac{n}{2}r_w+1$.  Then $e^{-tL_w} : L^2 \rightarrow L^2$ is
uniformly bounded for all $t>0$.    Similarly,  $\varphi(L_w)  : L^2 \rightarrow L^2$, where $\varphi$ is any
bounded holomorphic function on $\Sigma_\mu$, $\mu \in
(\vartheta,\pi)$, and $g_{L_w}  : L^2 \rightarrow L^2$.  

In particular, these $L^2$ estimates hold  if we assume that  $w\in A_1\cap RH_{1+\frac{n}{2}}$, or 
  more generally if $w\in A_{r}\cap RH_{\frac{n}{2}\,r+1}$ for $1<r\le 1+\frac2{n}$,  or if we take the power weights
$$
 w_\alpha(x)=|x|^{\alpha}, 
\qquad -\frac{2\,n}{n+2}<\alpha<2.
$$
\end{theor}

\begin{proof}
 Let $p=q=2$, $p_0=(2^*_w)'$, $q_0=2_w^*$, and let $v=w^{-1}$.  Then
 by Proposition~\ref{prop:J}, Corollary~\ref{corollary-weighted-offd}
 and the nesting properties of weights, $e^{-tL_w}\in \off{2}{2}$ if
$w^{-1} \in  A_{2/p_0}(w)\cap RH_{(q_0/2)'}(w)$; in particular, by
Lemma~\ref{lemma:unif-comp}, $e^{-tL_w} : L^2 \rightarrow L^2$ is uniformly
bounded.   However, this weight condition is equivalent
to
\[ w\in RH_{(2/p_0)'}\cap A_{q_0/2} . \]
A straightforward computation shows that
\[ \frac{q_0}{2} = \frac{nr_w}{nr_w-2}, \qquad
\left(\frac{2}{p_0}\right)' = \frac{n}{2}r_w+1.  \]
Since $r_w<1+\frac{2}{n}$, we have that $r_w < \frac{nr_w}{nr_w-2}$,
so we automatically have that $w\in A_{q_0/2}$.  Therefore, the
desired bounds hold if we have $s_w>\frac{n}{2}r_w+1$.
  If $w\in A_{r}\cap RH_{\frac{n}{2}\,r+1}$ with $1\le r\le 1+\frac2{n}$, then
$r_w \le  r$ and $s_w> \frac{n}{2}\,r+1 \ge \frac{n}{2}r_w+1$. 
The desired conclusion for power weights follows at once from~\eqref{sw-rw:w-alpha}.

\smallskip

The same argument holds for $\varphi(L_w)$ and $g_{L_w}$, using
Proposition~\ref{prop:B-K:weights} or
Proposition~\ref{prop:square-function}, respectively.
\end{proof}

It is straightforward to construct weights more general than power weights that satisfy the
conditions on $r_w$ and $s_w$ in the above theorems.   For instance, $w\in
  A_{1+\frac{2}{n}}\cap RH_{2+\frac{n}{2}}$   (which corresponds to the choice $r=1+\frac2n$)  if and only if there exist
 $u_1,\,u_2\in A_1$ such that
\[ w = u_1^{\frac{2}{n+4}}u_2^{-\frac{2}{n}}.\]
This follows from the Jones factorization theorem and the properties
of $A_1$ weights:  cf.~\cite{MR1308005}.

\begin{remark} \label{remark:Lpbounds}
We can modify the proof of
Theorem~\ref{thm:semigroup-nowt} to get unweighted $L^p$ estimates
for values of $p$ close to $2$.  We leave the details to the
interested reader.
\end{remark}

\medskip

For the reverse inequalities we need to take into account the slightly
stronger hypotheses in Proposition~\ref{prop:reverseRiesz};
otherwise, the proof of the following result follows exactly as in the
proof of Theorem~\ref{thm:semigroup-nowt}.

\begin{theor} \label{thm:reverseRiesz-nowt}
Given a weight $w\in A_2$, suppose that 
\[ 1\leq r_w<1+\frac{2}{n} \quad \text{ and }
s_w>\max\Big\{\Big(\frac{2}{r_w}\Big)', \frac{n}2\,r_w+1\Big\}. \]
Then
\begin{equation}\label{eq:reverseRiesz-dx}
\|L_w^{1/2}f\|_{L^2} \le C\, \|\nabla f\|_{L^2},
\qquad f\in\cals.
\end{equation}
In particular, this is the case if we either
assume that $w\in A_{1}\cap RH_{ 1+\frac{n}{2}}$,   or more generally
that 
$w\in A_{r}\cap RH_{\max\{(\frac{2}{r})', \frac{n}2\,r+1\} }$, with $1<r\le 1+\frac{2}{n}$, 
or 
 for power weights if we take
$$
w_\alpha(x)=|x|^{\alpha}, 
\qquad -\frac{2n}{n+2}=-\min\Big\{\frac{n}{2},\frac{2n}{n+2}\Big\}<\alpha<2.
$$ 
\end{theor}

\begin{remark}
  Note that $\max\{(\frac{2}{r})', \frac{n}2\,r+1\}=\frac{n}2\,r+1$ provided $r\le 2-\frac{2}{n}$ and this always holds if $n\ge 4$ as $1+\frac2n\le 2-\frac{2}{n}$. In this case, the conditions in the second part of Theorem~\ref{thm:reverseRiesz-nowt} simplify to the same conditions as in
Theorem~\ref{thm:semigroup-nowt}.  
\end{remark}

\begin{remark}\label{remark:p-vs-q}
We note that in Theorems \ref{thm:semigroup-nowt} and \ref{thm:reverseRiesz-nowt} we can replace $1\le r_w<1+\frac2{n}$ with the possibly weaker condition $1\le r_w<\frac{p_+(L_w)}{2}$. The proof only requires us to  take $q_0=p_+(L_w)$. 
\end{remark}

\bigskip

For the gradient of the semigroup $\sqrt{t} \grad e^{-tL_w}$, the
Riesz transform $\grad L_w^{-1/2}$, and the square
function $G_{L_w}$, our estimates depend on $q_+(L_w)$.  

\begin{theor} \label{thm:grad-square-nowt}
Given a weight $w\in A_2$, suppose $1\leq r_w<\frac{q_+(L_w)}2$ and
$s_w>\frac{n}{2}r_w+1$.   Then $\sqrt{t} \grad e^{-tL_w} : L^2
\rightarrow L^2$ is uniformly bounded for all $t>0$. 
Similarly, we have that  $\grad  L_w^{-1/2} : L^2 \rightarrow L^2$ and
$G_{L_w} : L^2 \rightarrow L^2$. 

In particular, this is the case if we assume that
$w\in A_1\cap RH_{\frac{n}{2}+1}$.  Furthermore, these $L^2$ estimates
hold if the following is true: given $\Theta \ge 1$ there exists
$\epsilon_0=\epsilon_0(\Theta, n, \Lambda/\lambda)$,
$0<\epsilon_0\le \frac1{2\,n}$, such that  
$w\in A_{1+\epsilon}\cap RH_{\frac{n}{2}\,(1+\epsilon)+1}$, 
$0\le \epsilon<\epsilon_0$, and $[w]_{A_2}\le \Theta$.  

For power weights, there exists $\epsilon_1=\epsilon_1(n, \Lambda/\lambda)$, $0<\epsilon_1\le \frac1{2}$, such that 
these estimate holds for 
\[ w_\alpha(x) = |x|^{\alpha}, \qquad  \qquad -\frac{2\,n}{n+2}<\alpha<\epsilon_1.\] 
\end{theor}

\begin{proof} 
We will prove this result for  $\sqrt{t} \grad e^{-tL_w}$ using 
Proposition~\ref{prop:K}.  The proof for $\grad  L_w^{-1/2}$ or $G_{L_w}$ is exactly the same,
using Proposition~\ref{prop:ext-RT} or
Proposition~\ref{prop:grad-square-function}.

By Proposition~\ref{prop:K},  $\sqrt{t} \grad e^{-tL_w} : L^2
\rightarrow L^2$ if $w^{-1}=v \in A_{2/q_-(L_w)}(w) \cap
RH_{q_+(L_w)/2)'}(w)$, which is equivalent to
\[ w\in RH_{(2/q_-(L_w))'}\cap A_{q_+(L_w)/2}.  \]
Therefore, we need $r_w<q_+(L_w)/2$.  Furthermore, since we have that
$q_-(L_w)=p_-(L_w)\leq (2_w^*)'$, we can take
\[ s_w > \left(\frac{2}{(2_w^*)'}\right)' =  \frac{n}{2}r_w+1. \]

\medskip

To get the particular examples stated in the theorem, note first that if we let $r_w=1$,
then it clearly suffices to assume $w\in A_1\cap
RH_{\frac{n}{2}+1}$, since we showed in Section~\ref{section:q-plus} that $q_+(L_w)>2$ for every $w\in A_2$.

\medskip

We now prove the condition for weights $w\in A_{1+\epsilon}$.  In this
case it is more difficult to satisfy the condition $r_w<q_+(L_w)/2$
since the righthand side can be very close to $1$ depending on
$w$.  Assume then that   
$w\in A_{1+\epsilon}\cap RH_{\frac{n}{2}\,(1+\epsilon)+1}$, with
$0\le \epsilon<\epsilon_0\le \frac1{2\,n}$, $[w]_{A_2}\le \Theta$, 
and with $\epsilon_0>0$ to be fixed below. Then we have that
$$
s_w
>
\frac{n}{2}\,(1+\epsilon)+1
\ge 
\frac{n}{2}\,r_w+1.
$$
Therefore, in order to apply the first half of the theorem we need to
show that we can choose $\epsilon_0$ sufficiently small so that
$r_w< q_+(L_w)/2$.   To do so we will use the notation and computations from
Section~\ref{section:q-plus}.  There  we showed that $q_+(L_w)\geq q_w$,
and so it will suffice to show that
\begin{equation} \label{eq:deadde}
2r_w < q_w = \min(r_w', p_+(L_w), p_0).  
\end{equation}
We will compare $r_w$ to each term in the minimum in turn.  

The first two terms are straightforward.  First, we have that
$r_w<1+\epsilon<1+\frac1{2\,n}<\frac{3}{2}$ and so $2r_w<r_w'$.
Second, $r_w<1+\frac1{2\,n}<1+\frac2{n}$, and it follows at once from
this that $2\, r_w<2_w^*$. By Proposition~\ref{prop:J},
$2_w^*\le p_+(L_w)$ and so $2\, r_w<p_+(L_w)$.  

Finally, we estimate $p_0$, the exponent from the higher integrability
condition~\eqref{eqn:gehring-bump}.  We will use the
formula~\eqref{eq:value-p0}.   First, we need to fix the exponent $q$
from the Poincar\'e inequality~\eqref{eq:w-Poincare-a}.   Let 
$q=2-1/n$; this value satisfies \eqref{eq:q-choice} since  $r_w<1+\frac1{2\,n}<1+\frac1n$.  
With this choice of $q$ (that only depends on $n$), we have that
\[
p_0 
= 
2 + \frac{2-q}{2^{4/q+1}C_1^2C_2^2 [w]_{A_2}^{6/q+17}}
=
2+\frac{1}{n\,C(n,\Lambda/\lambda)\,[w]_{A_2}^{\theta_n}}
\]
where $C(n,\Lambda/\lambda)\ge 1$ depends only on $n$ and the ratio $\Lambda/\lambda$ of the
ellipticity constants of the matrix $A$ used to define $L_w$, and where $\theta_n\ge 1$ depends only on $n$. 
Then, since we also assumed that  $[w]_{A_2}\le \Theta$, 
we get that
$$
p_0 
=
2+\frac{1}{n\,C(n,\Lambda/\lambda)\,[w]_{A_2}^{\theta_n}}
\ge
2+\frac{1}{n\,C(n,\Lambda/\lambda)\,\Theta^{\theta_n}}
=
2+2\,\epsilon_0,
$$
and $\epsilon_0= (2\,
n\,C(n,\Lambda/\lambda)\,\Theta^{\theta_n})^{-1}$ is such that $0<\epsilon_0\le \frac1{2\,n}$. Thus
$2\,r_w<2\,(1+\epsilon)<2\,(1+\epsilon_0)\le p_0$ and so
$2\,r_w<p_0$. This completes the proof that \eqref{eq:deadde} is satisfied,
and so the $L^2$ estimates hold for weights that satisfy   $w\in A_{1+\epsilon}\cap RH_{\frac{n}{2}\,(1+\epsilon)+1}$.  

\medskip

Finally, we consider power weights. First, it is easy to see that
\[ w_\alpha(x) = |x|^{\alpha}, \qquad  \qquad
\frac{-2\,n}{n+2}<\alpha\le 0 \] 
yields the desired estimates,  since in this case $r_w=1$ and
$s_w>\frac{n}2+1= \frac{n}2r_w+1$.

Now consider the case $\alpha>0$.  If we assume that $\alpha<\frac12$,
then $w\in A_{1+\frac1{2\,n}}\cap RH_\infty$. Moreover, it is
straightforward to show that  for all such $\alpha$, there exists
$\Theta$, depending only on $n$, such that 
$[w_\alpha]_{A_2}\le \Theta$.  
Now  apply the above argument to find $\epsilon_0\in(0,\frac1{2\,n}]$;
this value will only depend on $n$ and the ratio $\Lambda/\lambda$.   If we let
$\epsilon_1=n\,\epsilon_0$ and assume that $0<\alpha<\epsilon_1$, 
then $\alpha<\frac12$ and $w_\alpha\in A_{1+\epsilon}$ for some
$\epsilon<\epsilon_0$ as desired.
\end{proof}

\smallskip

To find examples of weights other than power weights  to which
Theorem~\ref{thm:grad-square-nowt} apply,  we argue as before.   If $u_1\in A_1$, then
\[ w = u_1^{\frac{2}{n+2}} \in A_1 \cap RH_{\frac{n}{2}+1}. \]
To get weights that are not in $A_1$, take $u\in A_2$ and
let $w=u^\theta$. If $\theta$ is sufficiently small (depending on $n$,
the ratio
$\Lambda/\lambda$ and $[u]_{A_2}$) we can show that $w$ satisfies the
final conditions given in~Theorem~\ref{thm:grad-square-nowt}.
Details are left to the interested reader.

\medskip

\begin{remark}
To get the unweighted lower estimate 
\[ \|f\|_{L^2} \leq C\|G_{L_w}f\|_{L^2}, \]
we note that by \eqref{eqn:-grad-square-reverse-wt} we
need $w^{-1}\in A_{2/q_+(\Delta_w)'}(w)$, or equivalently, 
$w\in RH_{(2/q_+(\Delta_w)')'}$.  Hence, it suffices to assume
\[ s_w > 1+ \frac{q_+(\Delta_w)}{q_+(\Delta_w)-2}. \]
Arguing as above we can construct weights that satisfy this condition;
details are left to the interested reader.
\end{remark}

\bigskip

If we combine  Theorems \ref{thm:reverseRiesz-nowt}, 
\ref{thm:grad-square-nowt},  and Remark \ref{remark:p-vs-q} we solve the Kato square root problem for
degenerate elliptic operators.

\begin{theor}\label{corol:super-Kato}
 Let $L_w=-w^{-1}\div A\grad$ be a degenerate elliptic operator with
  $w\in A_2$.   If 
\[ 1\leq r_w<\frac{q_+(L_w)}2 \qquad \text{ and} \qquad 
s_w>\max\Big\{\Big(\frac{2}{r_w}\Big)', \frac{n}{2}r_w+1\Big\}. \]
then the Kato  problem can be solved for $L_w$:  that is, for every $f\in H^1(\re^n)$. 
\begin{equation}\label{Kato-L2}
 \|L_w^{1/2} f\|_{L^2(\re^n)} \approx \|\grad f\|_{L^2(\re^n)},
\end{equation}
where the implicit constants depend only on the dimension, the ellipticity constants
 $\lambda$, $\Lambda$,  and $w$. 

In particular, \eqref{Kato-L2} holds if $w\in A_1\cap RH_{\frac{n}{2}+1}$.  Further, \eqref{Kato-L2} holds if the following is true: given $\Theta \ge 1$ there exists
$\epsilon_0=\epsilon_0(\Theta, n, \Lambda/\lambda)$,
$0<\epsilon_0\le \frac1{2\,n}$, such that   
$w\in A_{1+\epsilon}\cap RH_{\max\{(\frac{2}{1+\epsilon})', \frac{n}{2}\,(1+\epsilon)+1\}}$, 
$0\le \epsilon<\epsilon_0$, and $[w]_{A_2}\le \Theta$.  

For power weights, there exists $\epsilon_1=\epsilon_1(n, \Lambda/\lambda)$, $0<\epsilon_1\le \frac1{2}$, such that 
inequality \eqref{Kato-L2}  holds (with $w_\alpha$ in place of $w$) if 
\[ w_\alpha(x) = |x|^{\alpha}, \qquad  \qquad
-\frac{2\,n}{n+2}<\alpha<\epsilon_1. \] 
\end{theor}

We can restate the final part of Theorem~\ref{corol:super-Kato} as follows:
consider the family of operators
$L_\gamma=-|x|^{\gamma}\div(|x|^{-\gamma} B(x)\grad)$, where $B$ is an
$n\times n$ complex-valued matrix that satisfies the uniform
ellipticity condition
\[  \lambda | \xi | ^{2}\leq {\Re}\langle B(x)\xi
    ,\xi \rangle, \qquad
  |\langle B(x)\xi ,\eta \rangle |\leq \Lambda |\xi
  ||\eta |, \quad \xi ,\,\eta \in \mathbb{C}^{n},
	\ \mbox{a.e.~}x\in\re^n.   \]
Then,
\begin{equation} \label{eqn:more}
 \|L_\gamma^{1/2} f\|_{L^2(\re^n)} \approx \|\grad f\|_{L^2(\re^n)},\qquad
-\epsilon_1< \gamma<\frac{2n}{n+2}.
\end{equation}
When $\gamma=0$ we get the classical Kato
square root problem solved by Auscher, Hofmann, Lacey, McIntosh, and
Tchamitchian~\cite{auscher-hofmann-lacey-mcintosh-tchamitchian02}.
Inequality~\eqref{eqn:more}
shows that we can find an open interval containing $0$ such that if
$\gamma$ is in this interval, the same estimate holds.

\section{Applications to $L^2$ boundary value problems}
\label {section:BVP}


In this section we apply the results from the previous section to some
$L^2$ boundary value problems involving the degenerate elliptic
operator $L_w$.  We follow the ideas in
\cite{auscher-tchamitchian98} and consider semigroup solutions: for
the Dirichlet or Regularity problems we let
$u(x,t)=e^{-tL_w^{1/2}}f(x)$; for the Neumann problem we let
$u(x,t)=-L_w^{-1/2}\,e^{-tL_w^{1/2}}f(x)$.  In each case, for $t>0$ fixed
$L_w u(\cdot,t)$ makes sense in a weak sense  since
$u(\cdot,t)$ is in the domain of $L_w$.   Further, derivatives in $t$ are
well defined  because of the semigroup properties.   Finally, note that by the
strong continuity of the semigroup and the off-diagonal
estimates, in the context of the following results we have that 
$e^{-tL_w^{1/2}}f\to f$ as $t\to 0^+$ in $L^2$; see
\cite[Section~4.2]{auscher-martell07}. Further details are left to the
interested reader. 

\bigskip

We first consider the Dirichlet problem on $\re^{n+1}_+=\rn\times
[0,\infty)$: 
\begin{equation} \label{eqn:dirichlet}
 \begin{cases}
\partial_t^2 u - L_w u = 0, & \text{on } \rn \\
u \big|_{\partial\re^{n+1}_+}= f & \text{on } \partial\re^{n+1}_+ =\re^n.
\end{cases}
\end{equation}

\begin{theor} \label{thm:dirichlet} Given a weight $w\in A_2$, suppose
  $1\leq r_w<1+\frac{2}{n}$ and $s_w>\frac{n}{2}r_w+1$.  Then for any
  $f\in L^2(\re^n)$, $u(x,t)=e^{-tL_w^{1/2}}f(x)$ is a solution
  of~\eqref{eqn:dirichlet}   with convergence to the boundary data as $t\to 0^+$ in the $L^2$-sense.  Furthermore, we have that
\begin{equation} \label{eqn:dirichlet-bound}
 \sup_{t>0} \|u(\cdot,t)\|_{L^2} \leq C\|f\|_{L^2}. 
\end{equation}
   In particular, this is the case if we
assume that   $w\in A_1\cap RH_{1+\frac{n}{2}}$, or    $w\in A_{r}\cap RH_{\frac{n}{2}\,r+1}$ with $1<r\le 1+\frac{2}{n}$,  or if we take the power weights
$$
 w_\alpha(x)=|x|^{\alpha}, 
\qquad -\frac{2\,n}{n+2}<\alpha<2.
$$ 
\end{theor}

\begin{proof}
Formally, it is clear that $u$ is a solution to \eqref{eqn:dirichlet},
and this formalism can be justified by appealing to the theory of
maximal accretive operators:  see~Kato~\cite{kato66}. Alternatively, the 
weighted estimates for the functional calculus in Proposition~\ref{prop:B-K:weights} 
show that both $\frac{\partial^2}{\partial t^2}u(\cdot,t)$ and $L_w u(\cdot,t)$ belong to $L^2$ 
for each $t>0$ and that they are equal in the $L^2$-sense. 
To see that inequality~\eqref{eqn:dirichlet-bound} holds, it suffices
to let $\varphi_t(z)=e^{-t\sqrt{z}}$.  Then $\varphi_t$ is a bounded
  holomorphic function on $\Sigma_\mu$, and so by
  Theorem~\ref{thm:semigroup-nowt} we get the desired bound.  
\end{proof}

\begin{remark}\label{remark-extra}
Note that as observed in Remark \ref{remark:p-vs-q}, in the previous result we can replace $1\le r_w<1+\frac2{n}$ with the possibly weaker condition $1\le r_w<\frac{p_+(L_w)}{2}$. Also, by Proposition~\ref{prop:B-K:weights} we also have that for $u$ as in Theorem~\ref{thm:dirichlet} and all $k\ge 1$
\begin{equation}
\sup_{t>0} \left\| t^k\frac{\partial^k}{\partial t^k} u(\cdot,t)\right\|_{L^2} 
= 
\sup_{t>0} \left\| (t^k\, L_w^{1/2})^k e^{-tL_w^{1/2}}f(\cdot)\right\|_{L^2} \leq C\|f\|_{L^2}.
\label{eq:eqn-CR}
\end{equation}
\end{remark}

\medskip

For the regularity problem we have the following.  

\begin{theor} \label{thm:regularity} 
Given a weight $w\in A_2$, suppose
\[ 1\leq r_w<\frac{q_+(L_w)}2 \qquad \text{ and} \qquad 
s_w>\max\Big\{\Big(\frac{2}{r_w}\Big)', \frac{n}{2}r_w+1\Big\}. \]
 Then for any
  $f\in H^1(\re^n)$, $u(x,t)=e^{-tL_w^{1/2}}f(x)$ is a solution
  of~\eqref{eqn:dirichlet}  with convergence to the boundary data as $t\to 0^+$ in the $L^2$-sense.   Furthermore, we have that
\begin{equation} \label{eqn:regularity-bound}
 \sup_{t>0} \|\grad_{x,t}u(\cdot,t)\|_{L^2} \leq C\|\nabla f\|_{L^2}. 
\end{equation}

In particular, \eqref{eqn:regularity-bound} holds if we assume that
$w\in A_1\cap RH_{1+\frac{n}{2}}$. Furthermore, it holds if the
following is true: given $\Theta \ge 1$ there exists
$\epsilon_0=\epsilon_0(\Theta, n, \Lambda/\lambda)$,
$0<\epsilon_0\le \frac1{2\,n}$, such that  
$w\in A_{1+\epsilon}\cap RH_{\max\{(\frac{2}{1+\epsilon})', \frac{n}{2}\,(1+\epsilon)+1\}}$, 
$0\le \epsilon<\epsilon_0$, and $[w]_{A_2}\le \Theta$.

For power weights, there exists $\epsilon_1=\epsilon_1(n, \Lambda/\lambda)$, $0<\epsilon_1\le \frac1{2}$, such that 
\eqref{eqn:regularity-bound} holds if
\[ w_\alpha(x) = |x|^{\alpha}, \qquad  \qquad -\frac{n}{2}<\alpha<\epsilon_1.\] 
\end{theor}
 
\begin{proof}

Arguing as before, it suffices to prove that \eqref{eqn:regularity-bound} holds.  For any
$t>0$ we have, by Theorem~\ref{corol:super-Kato}
\begin{multline*}
\|\grad_{x,t}u(\cdot,t)\|_{L^2}
 \leq 
\|\grad L_w^{-1/2} L_w^{1/2}e^{-tL_w^{1/2}}f \|_{L^2}
+ \|L_w^{1/2}e^{-tL_w^{1/2}}f \|_{L^2} \\
 \lesssim
\|L_w^{1/2}e^{-tL_w^{1/2}}f \|_{L^2} 
=
\|e^{-tL_w^{1/2}}\,L_w^{1/2}f \|_{L^2} 
 \lesssim
\|L_w^{1/2}f \|_{L^2} 
\lesssim
\|\grad f\|_{L^2}.
\end{multline*}
\end{proof}

Note that under the hypothesis of Theorem \ref{thm:regularity}, and as observed in Remark \ref{remark-extra}, 
we have that $u(\cdot,t)=e^{-tL_w^{1/2}}f$ satisfies \eqref{eqn:dirichlet-bound} and \eqref{eq:eqn-CR}. Additionally, from the functional calculus estimates on $L^2$ it follows that
\begin{equation} \label{eqn:regularity-bound2}
\sup_{t>0} \|t\grad_{x,t}u(\cdot,t)\|_{L^2} 
\lesssim
\| t L_w^{1/2}e^{-tL_w^{1/2}}f \|_{L^2} \lesssim \| f\|_{L^2}.
\end{equation}

Finally, we consider the Neumann problem
\begin{equation} \label{eqn:neumann}
 \begin{cases}
\partial_t^2 u - L_w u = 0, & \text{on } \rn \\
\partial_t u \big|_{\partial\re^{n+1}_+}= f & \text{on } \partial\re^{n+1}_+ =\re^n.
\end{cases}
\end{equation}

\begin{theor} \label{thm:neumann} 
Given a weight $w\in A_2$, suppose $1\leq r_w<\frac{q_+(L_w)}2$ and
$s_w>\frac{n}{2}r_w+1$.   Then for any
  $f\in L^2(\re^n)$, $u(x,t)=-L_w^{-1/2}e^{-tL_w^{1/2}}f(x)$ is a solution of~\eqref{eqn:neumann}   with convergence of $\partial_t u(\cdot,t)\to f$ as $t\to 0^+$ in the $L^2$-sense.   Furthermore, we have that
\begin{equation} \label{eqn:neumann-bound}
 \sup_{t>0} \|\grad_{x,t}u(\cdot,t)\|_{L^2} \leq C\|f\|_{L^2}. 
\end{equation}

In particular, \eqref{eqn:neumann-bound} holds if we
assume that $w\in A_1\cap RH_{1+\frac{n}{2}}$. Furthermore, it holds
if the following is true:   given $\Theta \ge 1$ there exists
$\epsilon_0=\epsilon_0(\Theta, n, \Lambda/\lambda)$, $0<\epsilon_0\le
\frac1{2\,n}$, such that $w\in A_{1+\epsilon}\cap
RH_{\frac{n}2\,(1+\epsilon)+1}$, $0\le \epsilon<\epsilon_0$,  and $[w]_{A_2}\le
\Theta$.  

For power weights, there exists $\epsilon_1=\epsilon_1(n, \Lambda/\lambda)$, $0<\epsilon_1\le \frac1{2}$, such that 
\eqref{eqn:neumann-bound} holds if
\[ w_\alpha(x) = |x|^{\alpha}, \qquad  \qquad -\frac{2\,n}{n+2}<\alpha<\epsilon_1.\] 
\end{theor}

\begin{proof}
Again, $u$ is clearly a formal solution of~\eqref{eqn:neumann};
see~\cite{kato66}.  The proof that~\eqref{eqn:neumann-bound} holds is
similar to the proof of~\eqref{eqn:regularity-bound}:
$$
\|\grad_{x,t}u(\cdot,t)\|_{L^2}
 \leq 
\|\grad L_w^{-1/2} e^{-tL_w^{1/2}}f \|_{L^2}
+ \|e^{-tL_w^{1/2}}f \|_{L^2} 
 \lesssim
\|e^{-tL_w^{1/2}}f \|_{L^2} 
\lesssim
\|f\|_{L^2},
$$
where we have used Theorem~\ref{thm:grad-square-nowt} (for the Riesz
transform) and Theorem~\ref{thm:semigroup-nowt} (for the functional calculus with $\varphi(z)=e^{-t\sqrt{z}}$).  
\end{proof}

\begin{remark}
As we noted in Remark~\ref{remark:Lpbounds}, we can also get
unweighted $L^p$ bounds for these operators for values of $p$ close to
2.  As a consequence we can also get estimates $L^p$ boundary value
problems for the same values of $p$.  Details are left to the reader.
\end{remark}

\bibliographystyle{plain}
\bibliography{CMR}

\end{document}